\documentclass[11pt, a4paper, twoside]{amsart}
\usepackage{amsfonts}
\usepackage{mathrsfs}
\usepackage{amsmath}
\usepackage{amssymb}
\usepackage{fancyhdr}
\usepackage{graphicx}

\usepackage{enumerate}

\allowdisplaybreaks

\newcommand\Id{\mathrm{Id}}
\newcommand\lb{\mathrm{lb}}
\newcommand\rb{\mathrm{rb}}

\newcommand\diag{\mathrm{diag}}

\renewcommand\Re{\operatorname{Re}}
\renewcommand\Im{\operatorname{Im}}

\newcommand\lf{\mathrm{lf}}
\newcommand\rf{\mathrm{rf}}
\newcommand\mf{\mathrm{mf}}
\newcommand\bfz{\mathrm{bf}_0}

\newtheorem{theorem}{Theorem}%[section]
\newtheorem{lemma}[theorem]{Lemma}
\newtheorem{proposition}[theorem]{Proposition}

\newtheorem{definition}[theorem]{Definition} 

\newtheorem{remark}[theorem]{Remark}

\usepackage[colorlinks=true, citecolor=blue, linkcolor=blue, urlcolor=blue]{hyperref}
\begin{document}

\title{\textbf{
The Semiclassical Resolvent  on Conic Manifolds and Application to Schr\"odinger Equations} }
\author[X. Chen]{Xi Chen}
\address{ Shanghai Center for Mathematical Sciences, Fudan University, Shanghai 200438, China.  {\it E-mail address: \bf \tt xi\_chen@fudan.edu.cn}}
\address{ Department of Pure Mathematics and Mathematical Statistics, University of Cambridge, Cambridge CB3 0WB, UK.  {\it E-mail address: \bf \tt xi.chen@dpmms.cam.ac.uk}}

\begin{abstract}In this article, we shall construct the resolvent of the Laplacian at high energies near the spectrum on non-product conic manifolds with a single cone tip. Microlocally, the resolvent kernel is the sum of $b$-pseudodifferential operators, scattering pseudodifferential operators and intersecting Legendrian distributions. As an application, we shall establish Strichartz estimates for Schr\"odinger equations on non-compact manifolds with multiple non-product conic singularities.

\end{abstract}\maketitle
\tableofcontents

\section{Introduction}

Let $\mathcal{C} = [0, 1) \times \mathcal{Y}$ be a  conic $n$-manifold with a single cone tip and endowed with a non-product conic metric $g$,  where
\begin{itemize}
\item $\dim \mathcal{C} = n \geq 3$,	
\item the submanifold $\mathcal{Y}$ is a compact boundaryless $(n-1)$-manifold,
\item $x$ is the defining function of the cone tip, i.e. $\partial \mathcal{C} = \{x = 0\} = 0 \times \mathcal{Y}$,
\item $g$ is a smooth Riemannian metric outside the vicinity of the cone tip,
\item $g$, in a compact neighbourhood $\mathcal{K}$ of $\partial \mathcal{C}$, has the expression, $$g = dx^2 + x^2 h(x, y, dy),\footnote{This can be slightly more general, i.e. we can assume $h$ is also dependent of $dx$. In fact, one can   reduce the case $h(x, y, dx, dy)$ to the case $h(x, y, dy)$, by choosing appropriate coordinates in a collar neighbourhood of $\partial \mathcal{C}$. The latter metric is called the normal form of the former metric. See \cite[Theorem 1.2]{Melrose-Wunsch}.}$$ where $h(x, y, dy) = h_{ij}(x, y) dy^idy^j$ is a family of smooth metrics on $\mathcal{Y}$ with a smooth parameter $x$, 
\item  $\mathcal{K}$ is taken to be a sufficiently small neighbourhood of $\partial \mathcal{C}$ such that $h$ is stable in the sense \begin{equation}\label{eqn : stability}  \mathbf{e} =   \sup_{(x, y) \in \mathcal{K}} \frac{x  \partial_x \det (h_{ij}(x, y))_{ij} }{2\det (h_{ij}(x, y))_{ij}} << n-1, \end{equation}\item the non-productness of $g$ refers to the dependence of $h$ in $x$, 
\item all geodesics in $\mathcal{C}$ are simple, i.e. there are no conjugate points. \footnote{This assumption is virtually unnecessary to understand the effect of conic singularities. But conjugate points cause extra technicalities away from the cone tip, despite being tractable, which are beyond the scope of the present paper.}
\end{itemize}  

The Friedrichs extension of the Laplacian on $\mathcal{C}$, near the cone tip (i.e. in $\mathcal{K}$), reads
$$\Delta_{\mathcal{C}} = - \partial_x^2 - \frac{(n - 1) + xe}{x} \partial_x + \frac{\Delta_h}{x^2},$$
where $\Delta_h$ is the Laplacian with respect to $(\mathcal{Y}, h)$ with a parameter $x$. The non-productness of the metric in $x$ creates the  error function $e$, \begin{equation}\label{def : e}e(x, y) = \frac{1}{2} \frac{\partial \log \det (h_{ij}(x, y))_{ij}}{\partial x} = \frac{1}{2} \mathrm{tr} \bigg(h^{ik}(x, y) \frac{\partial h_{kj}(x, y)}{\partial x}\bigg)_{ij},\end{equation} where we denote the inverse matrix $(h_{ij})^{-1}_{ij} = (h^{ij})_{ij}$. We view $\Delta_\mathcal{C}$ as an operator acting on smooth sections of the Riemannian half density bundle of the conic metric $g$, which takes the form, $$f |dg|^{1/2} = f x^{(n - 1)/2}|dx dy|^{1/2} \in C^\infty(\mathcal{C}; |dg|^{1/2}).$$

%The Laplacian $\Delta$ can be converted into a $b$-operator. Consider the operator $\Delta_b = x \Delta x$, acting on the $b$-half densities $f |dg_b|^{1/2}$, where $|dg_b|^{1/2}$, near the boundary, can be expressed as $x^{-1/2} |dxdy|^{1/2}.$ In fact, it follows that \begin{eqnarray*}\Delta_b(f|dg_b|^{1/2}) &=& x^{-n/2} \bigg( x^{1 + n/2} \Big( - \partial^2_x - \frac{(n - 1) + xe}{x} \partial_x +  \frac{\Delta_h}{x^2} \Big) x^{1 - n/2}\bigg) \big( f|dg|^{1/2} \big) \\ &=&  \big( - (x \partial_x)^2 - x^2e \partial_x +  \Delta_h + (n/2 - 1)^2 + xe (n/2 - 1)  \big) ( f  |dg_b|^{1/2}).   \end{eqnarray*} Therefore, it yields that  \begin{equation}\label{eqn : delta_b}\Delta_b =  - (x \partial_x)^2 - x^2e \partial_x +  \Delta_h + (n/2 - 1)^2 + xe (n/2 - 1),\end{equation} which is a $b$-differential operator.

 %On the other hand, $\Delta$, away from the cone tips, behaves like a scattering Laplacian. In fact, we denote $\rho = 1/x$ and whence get the following scattering operator, $$\Delta_{sc} = - (\rho^2 \partial_\rho)^2 + ((n - 1)\rho^3 + e(1/\rho, y)\rho^2)  \partial_\rho + \rho^2 \Delta_h.$$

% These observations motivate us to employ Melrose's $b$-calculus and scattering calculus.

%\begin{center}\begin{figure}
	%	\includegraphics[width=0.5\textwidth]{singlecone.png}\caption{\label{fig : singlecone}The conic manifold $\mathcal{C}$ with a single cone tip}\end{figure} 

%\end{center}

\subsection{The semiclassical resolvent}

The main problem we want to solve in this paper is to describe the asymptotic microlocal structure of the resolvent kernel of $$(\Delta_{\mathcal{C}} - (\lambda \pm \imath 0)^2)^{-1},$$ at high energies near the spectrum, i.e. $|\lambda| \rightarrow \infty$.

In the product case i.e. $\partial_x h \equiv 0$,     the resolvent kernel can be computed by the method of separation of variables and expressed explicitly in terms of Bessel functions. See for example \cite[(8.30)]{Taylor2}. However, this strategy can not be extended to the non-product case, as the inhomogeneity of $\Delta_{\mathcal{C}}$ in $x$ (the presence of the error term $e(x, y)$) makes an essential challenge. To overcome this, we shall employ microlocal analysis to symbolically construct the resolvent kernel. 

A natural microlocal approach  to tackle this type of high energy problem is to convert the Helmholtz operator into a semiclassical operator by denoting $\hbar = |\lambda|^{-1}$. This conversion changes the high energy asymptotic of the resolvent of $\Delta_{\mathcal{C}}$ to the semiclassical resolvent of \begin{equation}\label{eqn : semiclassical conversion}P_\hbar(\pm \imath 0) = - \hbar^2\partial^2_x - \frac{(n - 1) + xe}{x} \hbar^2\partial_x + \hbar^2\frac{\Delta_h}{x^2} - (1 \pm \imath 0).\end{equation} 
To compare the conic singularity when $x \rightarrow 0$ and the semiclassical singularity when $\hbar \rightarrow 0$, we further introduce $r = x/\hbar$ and adopt coordinates $\{(r, \hbar)\}$ instead of $\{(x, \hbar)\}$. The semiclassical operator $P_\hbar(\pm \imath 0)$ will lift to
$$ \tilde{P}_\hbar(\pm\imath 0) = - \partial^2_r - \frac{(n - 1) + \tilde{e}r}{r} \partial_r + \frac{\Delta_h}{r^2} - (1 \pm \imath 0),$$
where the lifted error term is $$\tilde{e}(\hbar, r, y) = \hbar e(\hbar r, y).$$ We consider the incoming operator \begin{equation}\label{eqn : blownup-semiclassical-operator} \tilde{P}_\hbar = - \partial^2_r - \frac{(n - 1) + \tilde{e}r}{r} \partial_r + \frac{\Delta_h}{r^2} - (1 + \imath 0).\end{equation}

 The operator $\tilde{P}_\hbar$ now lives in a space with an artificial infinity and exhibits distinct behaviours around the two endpoints $\{r = 0\}$ and $\{r = \infty\}$ : \begin{itemize}
\item For small $r$, $\tilde{P}_\hbar$ reduces to a $b$-operator,  $P_b = r \tilde{P}_\hbar r$, where \begin{equation}\label{eqn : b-operator}  P_b  =  -(r\partial_r)^2 - r\tilde{e} (r\partial_r) + \Delta_h -  r \tilde{e}(1 - n/2) - r^2 (1 \pm \imath 0) +  (1 - n/2)^2.\end{equation} To see this, one has to note that $\tilde{P}_\hbar$ acts on the conic half densities, whilst the $b$-operator $P_b$ acts on the $b$-half densities $f |drdy/r|^{1/2} \in C^\infty({}^b\Omega^{1/2})$. By changing half densities, we have 
\begin{eqnarray*}
\lefteqn{P_b \bigg(f \Big|\frac{drdy}{r}\Big|^{1/2}\bigg)}\\
&=& r^{-n/2} \bigg( r^{1+n/2} \Big( - \partial_r^2 - \frac{(n-1) + \tilde{e}r}{r} \partial_r + \frac{\Delta_h}{r^2} - (1 + \imath 0) \Big) r^{1-n/2} \bigg) \big(f  |dg|^{1/2}\big)\\
&=& \bigg( \Big(-(r\partial_r)^2 - r\tilde{e} (r\partial_r) + \Delta_h -  r \tilde{e}(1 - n/2) - r^2 (1 + \imath 0) +  (1 - n/2)^2\Big) f \bigg) \Big|\frac{drdy}{r}\Big|^{1/2}.
\end{eqnarray*}

\item For large $r$, $\tilde{P}_\hbar$ turns out to be a (scattering) $sc$-operator. To see this, we change coordinates by making $\rho = r^{-1}$. Since $\partial_r$ lifts to $-\rho^2\partial_\rho$, this transformation yields  \begin{equation}\label{eqn : scattering-operator}P_{sc} = - (\rho^2\partial_\rho)^2 +\Big((n - 1)\rho  -  \tilde{\tilde{e}}(x, \rho, y)\Big) \rho^2\partial_\rho + \rho^{2}\Delta_h - (1 + \imath 0), \quad \mbox{for $\rho \rightarrow 0$,}\end{equation} where the error term reads $$ \tilde{\tilde{e}}(\rho, x, y) = \tilde{e}(\hbar, r, y)|_{\hbar = \rho x, \,r = 1/\rho} = \rho x e(x, y).$$
This operator acts on the $sc$-half densities $f|d\rho dy/\rho^{n+1}|^{1/2} \in C^\infty({}^{sc}\Omega^{1/2}).$
    \end{itemize}
    
The operators $P_b$ and $P_{sc}$ fit into Melrose's framework of differential analysis on manifolds with corners. Specifically, the resolvent of $b$-operators is a full $b$-pseudodifferential operator \cite[Chapter 4-6]{Melrose-APS}, whilst the resolvent of $sc$-operators is a $sc$-pseudodifferential operator \cite{Melrose-SC} plus Legendrian distributions \cite{Melrose-Zworski, Hassell-Vasy-JAM, Hassell-Vasy-AnnFourier}.

These observations suggest that the resolvent $\tilde{P}_\hbar^{-1}$ can be constructed by the following scheme: 
\begin{itemize}
\item use the full $b$-calculus to construct a local parametrix $G_b$ for $\tilde{P}_\hbar = r^{-1} P_b r^{-1}$ with a compact remainder $E_b$ for small $r$; 
\item use the $sc$-calculus and Legendrian distributions to construct a local parametrix $G_{sc}$ for $P_{sc}$ with a compact remainder $E_{sc}$ for large $r$;
\item improve the parametrices by understanding the transition between the $b$-regime (small $r$ region) and the $sc$-regime (large $r$ region);
\item use Fredholm theory to get the true resolvent from these local parametrices.  
\end{itemize}

Geometrically, the resolvent kernel  $\tilde{P}_\hbar^{-1}$ has to live in a multi-stretched version of $\mathcal{C}^2 \times [0,1)$ to incorporate these distinct microlocal structures. It is well-known that Melrose's framework of differential analysis on manifolds with corners, including $b$-calculus and $sc$-calculus, rests on some stretched spaces to desingularize the corners and express the resolvent kernel smoothly. However a simple combination of the $sc$-stretched and $b$-stretched spaces is insufficient to address the transition between the two regimes. Motivated by an unpublished idea of Melrose-S\'a Barreto, previously used in \cite{Guillarmou-Hassell-MathAnn2008, Guillarmou-Hassell-AnnFourier2009, Guillarmou-Hassell-Sikora-TAMS2013, Hintz},  we perform additional blow-ups on $\mathcal{C}^2 \times [0,1)$ and end up with the multi-stretched spaces $M_{b, 0}$ and $M_{sc, b, 0}$ defined in \eqref{eqn : M_{b, 0}} and \eqref{fig : M_{sc, b, 0}} respectively. We divide such spaces into \begin{itemize}
	\item 
 the $b$-regime $\{r < 1\} \cap \{r' < 1\}$, \item the $sc$-regime $\{\rho < 1\} \cap \{\rho' < 1\}$, \item the $(b, sc)$-regime $\{r < 1\} \cap \{\rho' < 1\}$, \item the $(sc,b)$-regime $\{\rho < 1\} \cap \{r' < 1\}$, \end{itemize}where we adopt the following projective coordinates on $\mathcal{C}$, $M_{b, 0}$ and $M_{sc, b, 0}$, $$r = x/\hbar, \quad r' = x'/\hbar, \quad \rho = \hbar/x, \quad \rho' = \hbar/x',$$ and use $\cdot'$ to denote the coordinates for the right subspace of the double space throughout the paper. The regime borders at the boundaries are the dashed curves in Figure \ref{blown-upspace(new)}.

We shall construct the semiclassical resolvent as follows.

\begin{theorem}\label{thm : resolvent}Let $\chi_b$ and $\chi_{sc}$ be smooth bump functions on $\mathcal{C}$ supported near $\{r < 1\}$ and $\{\rho < 1\}$ respectively, but also a partition of unity $\chi_b + \chi_{sc} = 1$. The resolvent kernel $R$ of $\tilde{P}_\hbar$ is a distribution on $M_{b, 0}$ consisting of
\begin{itemize}
\item $R_b = \chi_b R \chi_b$ is a distribution supported near the $b$-regime such that
\begin{itemize}
\item $R_b$ is  a  $b$-pseudodifferential operator near the diagonal,
\item $R_b$ is polyhomogeneous conormal to $\lb$ and $\rb$,
\item $R_b$ vanishes to order $2$ at $\mathrm{bf}$, to order $n/2$ at $\lb$ and $\rb$;
\end{itemize}

\item $R_{(b,sc)} = \chi_{b} R \chi_{sc}$ is a distribution supported near the $(b, sc)$-regime such that \begin{itemize}
	\item $R_{(b,sc)}$ vanishes to order $n/2$ at $\lb$ and to order $(n-1)/2$ at $\lb_0$;
\end{itemize}

\item $R_{sc} = \chi_{sc} R \chi_{sc}$ is a distribution supported near the $sc$-regime such that
\begin{itemize}
\item $R_{sc}$ is a $sc$-pseudodifferential operator near the diagonal of $M_{sc, b, 0}$,\footnote{The space $M_{sc, b, 0}$ is only needed for this part of the resolvent.}
\item $R_{sc}$ is intersecting Legendrian distributions associated with geometric Legendrians $L$ and diffractive Legendrians $L^\sharp$ over $\mf \cup \lf \cup \rf$, 
\item the principal symbol of $R_{sc}$, in the non-product case i.e. $e \not\equiv 0$, exhibits additional oscillations near $\lf$, $\rf$, and $L^\sharp$ respectively, determined by the error function $e$, 
\item $R_{sc}$ vanishes to order $(n-1-\mathbf{e})/2$  at both $\lf$ and $\rf$;
\end{itemize}

\item $R_{(sc,b)} = \chi_{sc} R \chi_{b}$ is a distribution supported near the $(sc, b)$-regime such that
\begin{itemize}
	\item $R_{(sc,b)}$ vanishes to order $n/2$ at $\rb$ and to order $(n-1-\mathbf{e})/2$ at  $\rb_0$.\end{itemize}
\end{itemize}
 \end{theorem} The precise statement of the resolvent kernel will be given in Theorem \ref{thm : full main theorem}. The geometry of $M_{b, 0}$ and $M_{sc, b, 0}$ will be described in Section \ref{sec : blow-up space}. The involved microlocal machinery  will all be defined in Section \ref{sec : abstract b calculus}-\ref{sec : Legendrian distributions on conic manifolds}.
It is also worth pointing out that the stability condition \eqref{eqn : stability}, i.e. $\sup_{\mathcal{K}} xe(x, y) = \mathbf{e}  << n-1$,  is essential to dominate the additional oscillations near $\lf$, $\rf$, and $L^\sharp$.  Otherwise it will go out of control and produce excess singularities.

Such a resolvent is inspired by a model problem, the resolvent construction at finite energies on non-compact metric cones, due to Guillarmou-Hassell-Sikora \cite{Guillarmou-Hassell-Sikora-TAMS2013}. The non-compact metric cone $\mathfrak{C} = [0, \infty] \times \mathcal{D}$ features a global metric $$\mathfrak{g} = d\mathfrak{x}^2 + \mathfrak{x}^2 \mathfrak{h}(y, dy)$$ with $\mathfrak{h}$ independent of $\mathfrak{x} \in [0, \infty)$ (in contrast with our local non-product case without infinity). Then the Friedrichs extension of the Laplacian on $\mathfrak{C}$ reads $$\Delta_{\mathfrak{C}} = - \partial_\mathfrak{x}^2 - \frac{(n - 1)}{\mathfrak{x}} \partial_{\mathfrak{x}} + \frac{\Delta_{\mathfrak{h}}}{\mathfrak{x}^2}.$$ The inhomogeneous $e$ term vanishes in $\Delta_{\mathfrak{C}}$ due to the productness of $(\mathfrak{C}, \mathfrak{g})$. The resolvent kernel $(\Delta_{\mathfrak{C}} - 1 \pm 0 \imath)^{-1}$ is a full $b$-pseudodifferential operator near the small end $\{\mathfrak{x}=0, \mathfrak{x}'=0\}$, and a $sc$-pseudodifferential operator plus intersecting Legendrian distributions associated with geometric Legendrians $L$ and diffractive Legendrians $L^\sharp$ near the big end $\{\mathfrak{x}=\infty, \mathfrak{x}'=\infty\}$.  See Theorem \ref{thm : resovlent on metric cones} for a full description.

The Legendrian structure at the big end of $\mathfrak{C}^2$ is developed from the geodesic flow at infinity on asymptotically conic manifolds, introduced by Melrose-Zworski \cite{Melrose-Zworski} for the scattering matrix, and further formulated by Hassell-Vasy \cite{Hassell-Vasy-JAM, Hassell-Vasy-AnnFourier} and Hassell-Wunsch \cite{Hassell-Wunsch} for the resolvent and spectral measure. (Intersecting) Legendrian distributions on the cotangent bundle restricted to the boundary are the contact analogue of (intersecting) Lagrangian distributions on the cotangent bundle restricted in the interior of the space, pioneered by H\"ormander \cite{Hormander-Acta-1971}, Duistermaat-H\"ormander \cite{Duistermaat-Hormander-Acta-1972} and Melrose-Uhlmann \cite{Melrose-Uhlmann-CPAM-1979}.

It is useful to note that the productness of $\mathfrak{h}$ grants the Laplacian $\Delta_{\mathfrak{C}}$ homogeneity in $x$.  Taking advantage of this, one can get the microlocal structure of the semiclassical resolvent $(\hbar^2 \Delta_{\mathfrak{C}} - 1 \pm \imath 0)^{-1}$ on metric cones, which, at the small end, agrees with our result. This indeed implies that the Legendrian structure of the semiclassical resolvent over the cone tip, without the additional oscillations in the product case, is related to the conic singularities over infinity of non-compact metric cones. This observation is essential in the present paper to understand the semiclassical conic singularities over the cone tip on non-product conic manifolds.

In addition, the Legendrian structure in the $sc$-regime characterises the singularities of the conic geometry and diffractions at the cone tip. The Legendrian pair $(L, L^\sharp)$ corresponds to the geometric propagation and diffractive propagation, in terms of the language of Melrose-Wunsch \cite{Melrose-Wunsch}. Microlocally,  'diffraction' refers
to the splitting of singularities (waves). In the conic geometry, it is well-known that the incoming waves split, at the cone tip, into geometric propagating waves and diffractive propagating waves. The wavefront sets of such two beams of waves are indeed inherited from  the geometric Legendrian $L$ and the diffractive Legendrian $L^\sharp$ respectively.

The resolvent on this class of local non-product conic manifolds, but away from the spectrum, has been considered by Loya \cite{Loya} and lately by Hintz \cite{Hintz, Hintz2}. In the classical case, Loya introduced the cone calculus to construct the resolvent $(\Delta_{\mathcal{C}} - \lambda^2)^{-1}$ for $|\arg \lambda^2| > \theta$ and $|\lambda| < C$. The cone calculus plays the role of $b$-calculus and $sc$-calculus discussed above and establishes a 'global' elliptic pseudodifferential calculus for the cone Laplacian away from the spectrum. This idea has been further developed by Hintz for the complex power of the semiclassical resolvent $(\hbar^2 \Delta_{\mathcal{C}} + 1)^{-1}$, which were used to investigate the propagation of semiclassical regularity and high energy estimates for potential scattering. The semiclassical resolvent, living also in the stretched double space of Melrose-S\'a Barreto type,  exhibits only polyhomogeneous conormal singularities due to the elliptic nature (away from the spectrum), in contrast with the Legendrian propagation in light of the vanishing boundary principal symbols (near the spectrum). Combining Hintz's work and the present paper together will give a complete description of the semiclassical resolvent on conic manifolds. 

Beyond Melrose's framework,  Schulze's cone calculus \cite{Schulze91, Schulze94, Schulze98} also deals with the differential analysis on manifolds with conic singularities. Schrohe-Seiler \cite{Schrohe-Seiler2005, Schrohe-Seiler2018},   utilizing this machinery, constructed a resolvent of the cone Laplacian in the off-spectrum case, which was applied to the study of the heat equation and the porous medium equation on conic manifolds.

\begin{center}\begin{figure} 
		\includegraphics[width=0.6\textwidth]{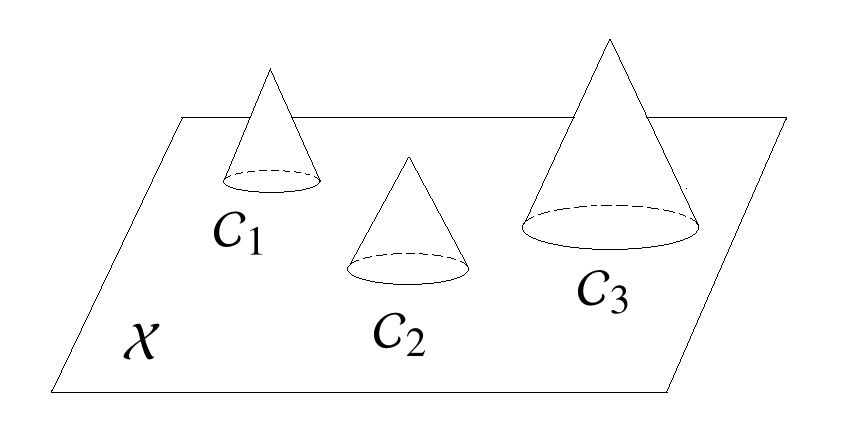}\caption{\label{fig : polycone}The non-compact manifold $\mathcal{X}$ with multiple conic singularities}\end{figure} \end{center}

\subsection{The Schr\"odinger equation}

Consider the Schr\"odinger equation on a non-compact manifold $\mathcal{X}$ with multiple conic singularities as in Figure \ref{fig : polycone}. We assume that $\mathcal{X}$ is an $n$-manifold satisfying \begin{itemize}
	\item the boundary $\partial \mathcal{X}$ has a finite number of components $\{\partial_1\mathcal{X}, \cdots \partial_m\mathcal{X}\}$,
	\item for each $i = 1, \cdots, m$, there exists an open neighbourhood $\mathcal{C}_i$ of $ \partial_i\mathcal{X}$ which is a conic manifold with a single cone tip $ \partial_i\mathcal{X}$ and endowed with a non-product conic metric $g_i = dx_i^2 + x_i^2 h_i (x_i, y, dy)$, 
	\item $\{\mathcal{C}_1, \cdots, \mathcal{C}_m\}$ are pairwise disjoint, 
	\item for each $i = 1, \cdots, m$, there exists a compact subset $\mathcal{K}_i \subset \mathcal{C}_i$ as well as a compact ball $\mathbb{B}_i \subset \mathbb{R}^n$  such that  $\mathcal{X} \setminus \bigcup_{i=1}^m\mathcal{K}_i  $ is isometric to $\mathbb{R}^n \setminus \bigcup_{i=1}^m \mathbb{B}_i$ and \eqref{eqn : stability} holds for the metric $h_i$ in $\mathcal{K}_i$,
	\item no geometric geodesic passes through three cone tips, where geometric geodesics are those missing the cone tips and their local limits.
\end{itemize}

The Schr\"odinger equation on $\mathcal{X}$ with a Cauchy data $f\in L^2(\mathcal{X})$ reads
 \begin{equation}\label{eqn : Schrodinger equation}\left\{\begin{array}{ll}\imath \partial_t u(t, z) + \Delta_{\mathcal{X}} u(t, z) = 0, & t > 0\\
u(t, z) = f(z), & t = 0\end{array}\right.,\end{equation}  where $\Delta_{\mathcal{X}}$ is the Friedrichs extension of the Laplacian on $\mathcal{X}$.
Using the precise semiclassical resolvent in Theorem \ref{thm : full main theorem}, we can show
\begin{theorem}\label{thm : Strichartz} 
	For any Schr\"{o}dinger admissible pair $(q, r)$, i.e. $(q, r)$ satisfies $$ \mbox{$r \geq 2$, $q \geq 2$, $(q, r) \neq (2, \infty)$, }\, \frac{2}{q} + \frac{n}{r} = \frac{n}{2},$$  the solution to \eqref{eqn : Schrodinger equation} on $\mathcal{X}$ satisfies the short time Strichartz estimates, i.e.  \begin{equation}\label{eqn : Strichartz without loss} 
	\|u\|_{L^q((0,1]; L^{r}( \mathcal{X}))} \lesssim \|f\|_{H^{\mathbf{k}}(\mathcal{X})}  
 \end{equation}where we use the usual notations of $L^2$-Sobolev spaces and space-time norms \begin{eqnarray*}  &H^{k}(\mathcal{X}) = \displaystyle \bigg\{f \in L^2(\mathcal{X}) : \sum_{l = 0}^{[k]} \|D^l f\|_{L^2(\mathcal{X})} + \Big(\int_{\mathcal{X}^2} \frac{|f(z) - f(z')|^2}{|z-z'|^{ 2[k] + n}} \,dg_{\mathcal{X}^2}(z, z')  \Big)^{1/2} < \infty\bigg\},&\\ &\displaystyle \|u(t, z)\|_{L^q(\mathcal{T}; L^{r}( \mathcal{X}))} =   \bigg(\int_{\mathcal{T}}\Big(\int_{\mathcal{X}} |u(t, z)|^r dg_{\mathcal{X}}(z) \Big)^{q/r} dt\bigg)^{1/q}, &
  \end{eqnarray*}
and the constant $\mathbf{k}$ is defined by
 \begin{eqnarray}
  \label{def : bf k}
   \displaystyle \mathbf{k} &=&     \frac{\mathbf{e}}{q(n + \mathbf{e}/2)},  \\ \label{def : bf e} \displaystyle \mathbf{e} &=& \max_{i=1, \cdots, m}\bigg\{\sup_{(x_i, y) \in \mathcal{K}_i} \Big( x_ie_i(x_i, y) \Big) \bigg\},  
 \end{eqnarray}
 where $e_i(x, y)$ is defined by $h_i(x_i, y, dy)$ in $ \mathcal{K}_i$ as in \eqref{def : e} and the same constant $\mathbf{e}$ is defined  in \eqref{eqn : stability} for the case of single cone tip.
\end{theorem}

Strichartz estimates   have been proved in a variety of product type conic manifolds in \cite{Ford, BFHM, BMW, Zhang-Zheng-Schrodinger}. Nevertheless, their approaches all rely on the explicit expression of the Schr\"odinger propagator in terms of Bessel and Hankel functions. In the meantime, it is also possible to take advantage of the homogeneity in $\mathfrak{x}$ of $\Delta_{\mathfrak{C}}$ on product metric cone $\mathfrak{C}$ to prove Strichartz estimates   by using Guillarmou-Hassell-Sikora's resolvent construction at finite energies. However, these methods all break down on non-product cones due to the presence of the error term $e$ in $\Delta_{\mathcal{C}}$ . Hence Strichartz estimates on non-product type conic manifolds were still unknown.
Theorem \ref{thm : Strichartz} indeed solves this open problem by utilizing the $b$-calculus and $sc$-calculus.  

The proof of Theorem  \ref{thm : Strichartz} integrates several existing techniques, which appeared for example in \cite{Keel-Tao,  Staffilani-Tataru-CPDE-2002,  Guillarmou-Hassell-Sikora, BMW, Hassell-Zhang}. The strategy  involves the $TT^\ast$ argument and Keel-Tao endpoint method for \eqref{eqn : Strichartz without loss} locally on each $\mathcal{C}_i$ as well as (dual) local smoothing for \eqref{eqn : Strichartz without loss} globally on $\mathcal{X}$. Specifically, we break Theorem  \ref{thm : Strichartz} into the following steps: \begin{itemize}
	\item obtain the structure of the spectral measure $dE_{\sqrt{\Delta_{\mathcal{C}_i}}}(\lambda)$ on each $\mathcal{C}_i$ by applying Stone's formula to the semiclassical resolvent,
	\item (micro)localize the spectral measure around the diagonal of $\mathcal{C}_i^2$  by conjugation of (microlocal) smooth cut-offs on $T^\ast \mathcal{C}_i$,
	\item prove dispersive and energy estimates on $\mathcal{C}_i$ for the cut-off of the Schr\"odinger propagator $e^{\imath t \Delta_{\mathcal{C}_i}}$, which, through the endpoint method,  leads to the local Strichartz estimates on each $\mathcal{C}_i \subset \mathcal{X}$,
	\item prove Strichartz estimates on $\mathcal{X}$ by combining the local estimates on $\mathcal{C}_i$ and the local smoothing estimates on $\mathcal{X}$.
\end{itemize}
The key ingredient in this program, carrying the information from the conic geometry, is the microlocal structure of the resolvent/spectral measure on  $\mathcal{C}$ for large spectral parameters. 
 
We also remark that   \eqref{eqn : Strichartz without loss} is indeed lossless when $\mathbf{e} = 0$ (i.e. $e(x, y) \leq 0$). Otherwise, the loss $\mathbf{k}$ appears to be inevitable. This is because the resolvent $\tilde{P}_\hbar^{-1}$ exhibits extra growth near $\lf \cup \rf \subset \{\hbar = 0\}$ as pointed out in the main theorem. Consequently, this produces an excess energy growth  $\lambda^{\mathbf{e}/2}$ as $\lambda \rightarrow \infty$  for the spectral measure $dE_{\sqrt{\Delta_{\mathcal{C}_i}}}(\lambda)$ in Proposition \ref{prop : microlocalized spectral measure}, and subsequently an excess time growth $t^{ - \mathbf{e}/4}$ as $t \rightarrow 0$ for $e^{\imath t \Delta_{\mathcal{C}_i}}$ in Proposition \ref{prop : microlocalized propagator}. Then the power of time in dispersive estimates is no longer the usual $t^{-n/2}$ but $t^{-n/2 - \mathbf{e}/4}$. To compensate this larger bound, the Cauchy data thus must bring more regularities.

Alternatively, one can adapt the framework in exterior domains formulated by  Burq-G\'{e}rard-Tzvetkov \cite{BGT-2004} and Burq \cite{Burq-Acta-1998, Burq-Duke2004} to show Strichartz estimates on $\mathcal{X}$ as follows: 
\begin{itemize}
\item establish the high energy cut-off resolvent estimates, 
\begin{equation}\label{eqn : L2 resolvent estimates}\|\chi (\Delta_{\mathcal{X}} - (\lambda \pm \imath 0)^2)^{-1}\chi\|_{L^2(\mathcal{X}) \rightarrow L^2(\mathcal{X})} \lesssim |\lambda|^{-1}, \quad \mbox{$|\lambda| > 1$,}\end{equation} 
for any $\chi \in C_c^\infty(\mathcal{X})$,
\item make use of the resolvent estimates  \eqref{eqn : L2 resolvent estimates}   to show the local smoothing estimates for the solution to Schr\"odinger equations,  
\begin{equation}\label{eqn : local smoothing}
\int_0^1 \|\chi u(t, z)\|^2_{\mathcal{D}^{1/2}(\mathcal{X})}\,dt \lesssim \|f\|_{L^2(\mathcal{X})},
\end{equation}
where $\mathcal{D}^{s/2}(\mathcal{X})$ is the Friedrichs domain of $\Delta_{\mathcal{X}}^{s/2}$, which is the Sobolev space $H^s(\mathcal{X})$ away from the boundary and locally a weighted $b$-Sobolev space $x^{-n/2 + s} H_b^s(\mathcal{X})$ near the boundary,\footnote{Here $x$ is a local boundary defining function and the $b$-Sobolev space $H_b^s(\mathcal{X})$ will be defined in \eqref{eqn : b-Sobolev}. See \cite[Section 3]{Melrose-Wunsch} for details.}
\item prove the Strichartz estimates by utilizing \eqref{eqn : local smoothing} and interpolation.
\end{itemize}
On $\mathcal{X}$, Baskin-Wunsch \cite{Baskin-Wunsch-JDG} established \eqref{eqn : L2 resolvent estimates} as well as \eqref{eqn : local smoothing}. However, this program  only yields Strichartz estimates with a $1/q$ loss, even if $\mathbf{e} = 0$. This is far less accurate than \eqref{eqn : Strichartz without loss}, since the resolvent estimates \eqref{eqn : L2 resolvent estimates} do not contain the complete information of the semiclassical resolvent on each $\mathcal{C}_i$. For example the Legendrian structure described in Theorem \ref{thm : full main theorem} is not addressed by the resolvent estimates. It is the propagation along the geometric Legendrian that prompts the various oscillations of the resolvent / spectral measure, and then  leads to the dispersive estimates \eqref{eqn : dispersive estimates}.

 %In the case of exterior domains, Ivanovici \cite{Ivanovici-apde} remedied the issue of loss  by showing semiclassical Strichartz estimates without loss on compact manifolds with boundaries, in contrast with the well-known semiclassical estimates with a loss  on compact manifolds without boundaries in \cite{Burq-Gerard-Tzvetkov-AmerJMath}. 
 
%$M$ of boundaries (in analogy with $\mathcal{K}_i$ in our case) \begin{equation}\label{eqn : semiclassical Strichartz}\|v\|_{L^q((0, 1], L^r(M))} \leq C \hbar^{-1/q} \|\Psi(\hbar^2 \Delta_{M}) v_0\|_{L^2(M)},\end{equation} for solutions to the semiclassical Schr\"odinger equation 
%$$\left\{ \begin{array}{ll}\imath \hbar \partial_t v(t, z) + \hbar^2 \Delta_{M} v(t, z) = 0,& t > 0\\ v(0, z) = \Psi(\hbar^2 \Delta_{M}) v_0(z),& t = 0\end{array}\right.,$$ where $\Psi \in C_c^\infty(\mathbb{R}\setminus 0)$, which extends 

%Given Theorem \ref{thm : Strichartz}, it is standard to establish the well-posedness for the nonlinear Schr\"odinger equation (NLS) on $\mathcal{X}$  with an $H^s(\mathcal{X})$ Cauchy data ,
%\begin{equation*}\label{eqn : NLS} \imath \partial_t u(t, z) + \Delta_{\mathcal{X}} u(t, z) = \beta(|u(t, z)|) u(t, z),  ,\end{equation*}$\beta(\cdot)$ is a polynomial with $\beta(0)=0$.  By the Christ-Kiselev argument, one readily obtains Strichartz estimates for inhomogeneous Schr\"odinger equations. Then one can set up  The classical contraction mapping argument (see e.g. \cite[Chapter 4]{Cazenave}) reduces the well-posedness of this class of NLS to Strichartz estimates for  linear homogeneous Schr\"odinger equations 

\subsection*{Outline} To begin with, the multi-stretched spaces $M_{b, 0}$ and $M_{sc, b, 0}$ to accommodate the resolvent kernel will be defined in Section \ref{sec : blow-up space}. The abstract microlocal machinery, including the $b$-calculus, the $sc$-calculus, and Legendrian distributions, will be reviewed Section \ref{sec : abstract b calculus} and Section \ref{sec : abstract sc calculus} and then applied to the multi-stretched spaces in Section \ref{sec : b calculus on $M_{b, 0}$} and Section \ref{sec : Legendrian distributions on conic manifolds} respectively. To alleviate technicality, we only present what will be needed in the resolvent construction. The reader is referred to \cite{Melrose-APS, Melrose-SC, Melrose-Zworski, Hassell-Vasy-JAM, Hassell-Vasy-AnnFourier, Hassell-Wunsch} for a complete and detailed presentation. With these microlocal tools, the local parametrices in the four regimes will be constructed in Section \ref{sec : b parametrix}-\ref{sec : sc-b parametrix} respectively. Subsequently, Section \ref{sec : true resolvent} will obtain the true resolvent. In the end, we establish Strichartz estimates.

\subsection*{Acknowledgement}
The author is supported by NSFC grant 11701094 and EPSRC grant EP/R001898/1. The author  would like to extend his gratitude to the following people. Nicolas Burq suggested the author prove Strichartz estimates without loss by using semiclassical estimates for the resolvent, which was one of the original motivations to construct the semiclassical resolvent in this paper.  Andrew Hassell taught the author the general philosophy of Melrosian calculus extensively used in this project, referred the author to \cite{Guillarmou-Hassell-Sikora-TAMS2013} for the model problem on non-compact metric cones, and explained to the author a few details in \cite{Hassell-Vasy-JAM, Hassell-Vasy-AnnFourier}. Peter Hintz shared the author privately with his article \cite{Hintz}. Jared Wunsch suggested the study of Schr\"odinger equations on conic manifolds and had numerous illuminative discussions with the author over this project. Part of the work was done during the author's visits to Universit\'e Paris-Sud and Northwestern University. The author would like to thank their hospitality and support. The author is also indebted to anonymous referees for their careful reading and insightful comments.

\begin{center}\begin{figure}
\includegraphics[width=0.6\textwidth]{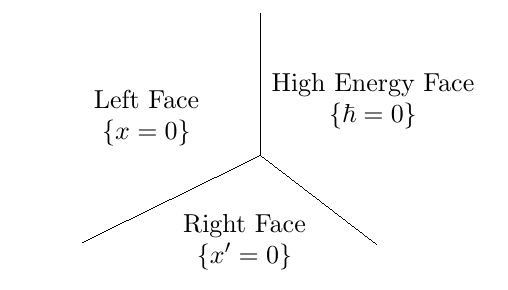}\caption{\label{fig: original}The space $M = \mathcal{C}^2 \times [0, 1)$}\end{figure}

 \end{center}

\section{The stretched spaces} \label{sec : blow-up space}

\subsection{The real blow-up}

We wish to understand the microlocal structure of the resolvent of $P_\hbar$, expressed in \eqref{eqn : semiclassical conversion}, with a parameter $\hbar \in [0, 1)$. The resolvent kernel $P_\hbar^{-1}$ is normally viewed as living in the double space $M = \mathcal{C}^2 \times [0, 1)$, with local coordinates $(x, y, x', y', \hbar)$ near the boundaries. See Figure \ref{fig: original}.

However, the resolvent kernel is singular near the boundaries of $M$, as the $b$-calculus and $sc$-calculus suggests, even if $\mathcal{C} = \mathbb{R}^n$.  For example, the resolvent kernel on $\mathbb{R}^3$, given explicitly by $$(\Delta_{\mathbb{R}^3} - \lambda^2)^{-1}(z, z', \pm \lambda) = \frac{e^{\pm \imath \lambda |z-z'|}}{4\pi|z- z'|},$$ could be singular when $|z-z'| \rightarrow 0$ and $|\lambda| \rightarrow \infty$ simultaneously. 

To remedy this and express the asymptotic limit of the kernel explicitly, Melrose pioneered the use of real blow-up with respect to a $p$-submanifold. Given a manifold $X$ with corners, we say $Y$ is a (product type) $p$-submanifold if for every point $y\in Y$, there exist local coordinates $\{(x_1, \cdots, x_r, y', y'')\}$ of $X$ such that
\begin{itemize}
	\item $x_i \geq 0$ are boundary defining functions,
	\item $y' \in \mathbb{R}^{s'}$ and $y'' \in \mathbb{R}^{s''}$,
	\item $\dim X = r + s' + s''$,
	\item $Y$ is locally parametrized  by $\{x_1 = \cdots = x_q = 0, y'=0\}$ where $q \leq r$.
\end{itemize}
The real blow-up of $X$ with respect to $Y$ is a new manifold with an extra boundary hypersurface, $$[X; Y] = (X \setminus Y) \cup SN_+ Y,$$ where $Y$ is replaced by the inward pointing spherical normal bundle $SN_+ Y$. The manifold $[X; Y]$ is a desingularized version of $X$, in the sense that the fractions, which appear in the resolvent kernel, such as $x_i/\varrho$ for $1 \leq i \leq q$ and $y'/\varrho$  where $\varrho = (\sum_{i=1}^q x_i^2 + |y'|^2)^{1/2}$, lift to smooth functions in $[X; Y]$. 

\begin{center}\begin{figure}
		\includegraphics[width=0.6\textwidth]{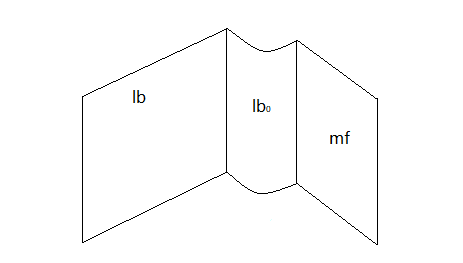}\caption{\label{fig: stretchedsingle}The stretched single space $\mathcal{C}_b$}\end{figure} \end{center}
\subsection{The stretched single space}

The differential operator $P_\hbar$ in \eqref{eqn : semiclassical conversion} is defined on the single space $\mathcal{C} \times [0,1)$. The motivation computation suggests one can convert $P_\hbar$ to a $b$-operator $P_b$ in \eqref{eqn : b-operator} and a $sc$-operator $P_{sc}$ in \eqref{eqn : scattering-operator} by changing variable $x$ to $r = x/\hbar$ and $\rho = \hbar/x$ respectively.

In terms of the real blow-up, this variable changing amounts to blowing up the corner of $\mathcal{C} \times [0,1)$. In other words, the operator $P_b$ and the operator $P_{sc}$ are defined on the stretched single space $$\mathcal{C}_b = [\mathcal{C} \times [0,1); \partial \mathcal{C} \times \{0\}].$$ We denote by $\beta_\hbar$ the blow-down map corresponding to such real blow-up.      Then the boundary of $\mathcal{C}_b$ consists of $\lb_0, \lb, \mf$, where \begin{eqnarray*}
	\lb_0 &=& \overline{\beta_\hbar^{-1}(\partial \mathcal{C} \times \{0\})}\\
	\lb &=& \overline{\beta_\hbar^{-1} (\partial \mathcal{C} \times [0, 1))}\\
	\mf &=& \overline{\beta_\hbar^{-1} (\mathcal{C} \times \{0\})}.
\end{eqnarray*}

There is a natural atlas for $\mathcal{C}_b$.
\begin{itemize}
		\item In an open neighbourhood of $\mf \cap \lb_0$ away from $\lb$, we have local coordinates 
	$(\rho, y, \hbar)$, where $\hbar$ is the total defining function of $\lb_0 \cup \mf$ and $\hbar/\rho, \rho$ locally define $\lb_0, \mf$ respectively.	
	\item In an open neighbourhood of $\lb$ away from $\mf$, we have local coordinates 
	$(r, y, \hbar)$, where $\hbar r$ is the total defining function of $\lb_0 \cup \lb$ and $\hbar, r$ locally define $\lb_0, \lb$ respectively.
	\item In an open neighbourhood away from $\lb \cup \lb_0$, we have local coordinates $(z, \hbar)$, where $z$ is in the interior of $\mathcal{C}$. 
	
\end{itemize}

To see the $b$-structure\footnote{The abstract $b$-calculus will be reviewed in Section \ref{sec : abstract b calculus}.} of $P_b$ and the $sc$-structure\footnote{The abstract $sc$-calculus will be reviewed in Section \ref{sec : abstract sc calculus}.} of $P_{sc}$, we consider the space $\mathcal{V}_b(\mathcal{C})$ of $b$-vector fields on $\mathcal{C}$ (which are tangent to $\partial \mathcal{C}$) $$\mathcal{V}_b(\mathcal{C}) = \mathrm{span} \{x \partial_x, \partial_y\} \quad \mbox{near $\partial \mathcal{C}$}$$ and the vector fields on $\mathcal{C}\times [0, 1)$, which are tangent to $\partial \mathcal{C}$ and vanish at the high energy face $\mathcal{C} \times \{0\}$,   $$\hbar \mathcal{V}_b(\mathcal{C} )   = \mathrm{span} \{\hbar x \partial_x, \hbar \partial_y\}, \quad \mathcal{V}_b([0, 1)) =   \mathrm{span} \{\hbar \partial_\hbar\} \quad \mbox{near $\partial \mathcal{C}\times [0, 1)$}.$$
We define the space $\mathcal{V}_{b, sc}(\mathcal{C}_b)$ of vector fields on $\mathcal{C}_b$ as 
$$\mathcal{V}_{b, sc}(\mathcal{C}_b) = \rho_{\lb_0}^{-1}  ( \mathcal{P}_{\mathcal{C} \times [0,1) \rightarrow \mathcal{C}} \circ \beta_\hbar)^\ast (\hbar \mathcal{V}_{b}(\mathcal{C})) + \mathcal{V}_b([0, 1)),$$ where $\rho_\bullet$ denotes the boundary defining function of the boundary hypersurface $\bullet$ and $\mathcal{P}_{\mathcal{C} \times [0,1) \rightarrow \mathcal{C}}$ is the canonical projection from $\mathcal{C} \times [0,1)$ to $\mathcal{C}$. 
With such a space of vector fields on $\mathcal{C}_b$, we further introduce the compressed tangent bundle ${}^{b, sc}T \mathcal{C}_b$ as a bundle  over  $\mathcal{C}_b$ the smooth sections of which are $\mathcal{V}_{b, sc}(\mathcal{C}_b)$, the compressed cotangent bundle ${}^{b, sc}T^\ast \mathcal{C}_b$ as the dual bundle of ${}^{b, sc}T \mathcal{C}_b$, the compressed density bundle ${}^{b, sc}\Omega(\mathcal{C}_b)$ as the line bundle the smooth sections of which are the wedge products of a basis of sections for ${}^{b, sc}T^\ast \mathcal{C}_b$.

We see the following local properties: 
\begin{itemize}
	\item near $\lb \subset \mathcal{C}_b$ away from $\mf$,

\begin{itemize}
	\item   we have the local coordinates $(r, y, \hbar)$ and $\rho_{\lb_0} = \hbar$,
	\item   $\mathcal{V}_{b, sc}(\mathcal{C}_b)$ is locally spanned by  
$$r\partial_r, \quad  \partial_y,  \quad \hbar \partial_\hbar,$$ 
\item 
the operator $P_b$ is then an elliptic combination of the sections of ${}^{b, sc}T \mathcal{C}_b$,
\item a basis of sections of ${}^{b, sc}T^\ast \mathcal{C}_b$ is locally given by the following singular $1$-forms, $$\frac{dr}{r}, \quad dy, \quad \frac{d\hbar}{\hbar},$$

\item the compressed density takes the form $$\bigg|\frac{drdyd\hbar}{r\hbar}\bigg|;$$

\end{itemize}

\item near $\mf \subset \mathcal{C}_b$ away from $\mathrm{lb}$,

\begin{itemize}
	\item  we have the local coordinates $(\rho, y, \hbar)$ and $\rho_{\lb_0} = \hbar/\rho$
	\item    $\mathcal{V}_{b, sc}(\mathcal{C}_b)$ is locally spanned by  
$$ \rho^2 \partial_\rho, \quad \rho \partial_y,  \quad \hbar \partial_\hbar,$$  
\item the operator $P_{sc}$ is then an elliptic combination of the sections of ${}^{b, sc}T \mathcal{C}_b$,
\item a basis of sections of ${}^{b, sc}T^\ast \mathcal{C}_b$ is locally given by the following singular $1$-forms, $$\frac{d\rho}{\rho^2}, \quad \frac{dy}{\rho}, \quad \frac{d\hbar}{\hbar},$$
\item the compressed density takes the form $$\bigg|\frac{d\rho dy d\hbar}{\rho^{n+1 }\hbar}\bigg|;$$
\end{itemize}

\item away from $\lb \cup \lb_0$, i.e. a neighbourhood of $(\mathcal{C}\setminus\mathcal{K}) \times [0, 1)$,

\begin{itemize}
	\item  we have the local coordinates $(z, \hbar)$
	\item    $\mathcal{V}_{b, sc}(\mathcal{C}_b)$ is locally spanned by  
	$$ \partial_z,  \quad \hbar \partial_\hbar,$$   
	\item a basis of sections of ${}^{b, sc}T^\ast \mathcal{C}_b$ is locally given by the following singular $1$-forms, $$dz, \quad \frac{d\hbar}{\hbar},$$
	\item the compressed density takes the form $$\bigg|\frac{dz d\hbar}{ \hbar}\bigg|,$$
	\item this region is viewed as a part of $\{\rho < 1\} \subset \mathcal{C}_b$ though there is no $\rho = \hbar/x$ coordinate. 
\end{itemize}

\end{itemize}

\begin{center}
	\begin{figure}
		\includegraphics[width= \textwidth]{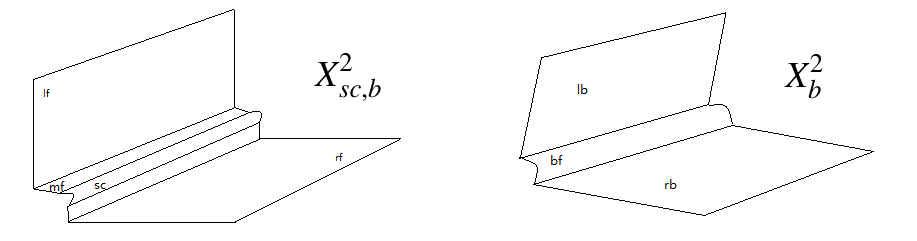}\caption{\label{Xb2}The $sc$-stretched and $b$-stretched double spaces}\end{figure}
\end{center}

\subsection{The multi-stretched double space}

However, the stretched single space $\mathcal{C}_b$ is insufficient to address all singularities of the resolvent, e.g. the diagonal singularity as in the expression of $(\Delta_{\mathbb{R}^3} - \lambda^2)^{-1}$, as well as its interactions with boundary singularities generally in Melrosian calculus on manifolds with ends.

Via the canonical  'left' ('right') projection $\mathcal{P}_L (\mathcal{P}_R) : M \rightarrow \mathcal{C}$, we can pull back $P_\hbar$ to the double space $M$ and view it as a distribution acting on the 'left half' ('right half') of $M$ when we construct the 'right-sided' ('left-sided') parametrix. To avoid confusion,  we shall use the following notations for a differential operator $P$ on the single space and distributions $G$ on the double space  
\begin{itemize}
	\item by $P(f)$ we mean $P$ acting on a function $f$ on the single space, 
	\item by $P \circ G$ we mean $P$ acting on the left half of $G$ on the double space,
	\item  the usual composition of $G_1$ and $G_2$ and the power of $G$ on the double space is denoted by  $G_1 G_2$ and $G G = G^2$. 
\end{itemize}

As we have seen on the single space, the operator $P_\hbar$ turns out to be a $b$-operator $P_b$ and a $sc$-operator $P_{sc}$ near the two ends $\{r < 1\}$ and $\{\rho < 1\}$ respectively. This implies we need to perform blow-ups to accommodate both the $b$-structure and the $sc$-structure.

In an abstract setting, a $b$-differential operator $\mathcal{O}_b$ and a $sc$-differential operator $\mathcal{O}_{sc}$ individually act as a distribution on a manifold $X$ with local coordinates $(x, y)$ near the boundary where $x$ is a boundary defining function. We shall follow the convention that the coordinates for the right copy of the double space are primed. Then the kernels of $\mathcal{O}_{sc}$ and $\mathcal{O}_{b}$ are defined in $X^2$ with two boundary faces $B_l = \{(0, y, z') \in \partial X \times X\}$ and $B_r = \{(z, 0, y') \in X \times \partial X\}$ as well as a corner $C_X = \{(0, y, 0, y') \in \partial X \times \partial X\}$. The resolvent kernels of $\mathcal{O}_{sc}$ and $\mathcal{O}_{b}$, in terms of the $sc$-calculus and $b$-calculus, live in   $X^2_b = [X^2; C_X]$ and $X_{sc, b}^2 = [X^2_b; \{x / x' = 1, y = y'\}]$ respectively. See Figure \ref{Xb2}.

Hence we aim to describe the resolvent kernel of $P_\hbar$ by consolidating such two distinct microlocal structures in $M$. The resolvent exhibits separate singularities as we approach the left face, the right face, and the high energy face respectively. We thus need to perform a series of blow-ups to compare these singularities.

The order of these blow-ups matters on $M$. First, the reduction to \eqref{eqn : blownup-semiclassical-operator} by coordinate changing $r = x/\hbar$ ($r' = x'/\hbar$)  means we blow up the corner between the left (right) face and the high energy face. Second, as $P_b$ for small $r = x/\hbar$ is a $b$-differential operator, we need to blow up the intersection of the left and right faces to employ the $b$-calculus. Third, a blow-up of the diagonal at the high energy faces is required to apply the $sc$-calculus for $P_{sc}$. These blown up faces will affect one another and meet at the codimension three corner $\{x=x'=\hbar=0\} \subset M$.  

It is natural to firstly blow up the corner with the highest codimension,  $$\beta_0 : M_0 = [M; \{x=x'=\hbar=0\}] \longrightarrow M , $$ and then perform blow-ups, \begin{equation}\label{eqn : M_{b, 0}}\beta_{b} : M_{b, 0} = [M_0; \{x=x'=0\}, \{x=\hbar=0\}, \{x'=\hbar=0\}] \longrightarrow M_0.\end{equation}
This two-fold blow-up $\beta_{b, 0} =   \beta_0\circ\beta_b$ is illustrated in Figure \ref{fig: M_{b, 0}}.  We introduce the following notations for the boundary hypersurfaces of $M_{b, 0}$,
\begin{eqnarray*}
\bfz  &=& \overline{\beta_{b, 0}^{-1}(\partial \mathcal{C} \times \partial \mathcal{C} \times \{0\}) }  \\
\mathrm{bf} &=&	\overline{\beta_{b, 0}^{-1}( \partial \mathcal{C} \times \partial \mathcal{C} \times [0, 1)   \setminus  \partial \mathcal{C} \times \partial \mathcal{C} \times \{0\} ) } 	 \\      \lb_0 &=&	\overline{\beta_{b, 0}^{-1}( \partial \mathcal{C} \times  \mathcal{C} \times \{0\}   \setminus  \partial \mathcal{C} \times \partial \mathcal{C} \times \{0\} ) } \\
 \rb_0 &=&	\overline{\beta_{b, 0}^{-1}( \mathcal{C} \times  \partial \mathcal{C} \times \{0\}   \setminus  \partial \mathcal{C} \times \partial \mathcal{C} \times \{0\} ) }  \\
  \mf  &=&	\overline{\beta_{b, 0}^{-1}(   \mathcal{C} \times  \mathcal{C} \times \{0\}    ) }  \\
 \lb &=&	\overline{\beta_{b, 0}^{-1}( \partial  \mathcal{C} \times  \mathcal{C} \times [0, 1)   ) } \\
	  \rb  &=&	\overline{\beta_{b, 0}^{-1}( \mathcal{C} \times  \partial  \mathcal{C} \times [0, 1)   ) }. 
\end{eqnarray*}

Let $\diag_\bullet$ be the diagonal of a double space $\bullet$. The diagonal of the interior of $M$ is lifted via $\beta_{b, 0}^{-1}$ to the diagonal of $M_{b, 0}$ $$\diag_{M_{b, 0}} = \overline{\beta^{-1}_{b, 0}(\{(z, z, \hbar) : z \in \mathcal{C} \setminus \partial \mathcal{C}, \hbar \in [0, 1)\})}.$$ We denote by $\partial_b \diag_{M_{b, 0}}$ and $\partial_{sc} \diag_{M_{b, 0}}$ the intersections of $\diag_{M_{b, 0}}$ with $\mathrm{bf}$ and $\mf$ respectively.
The dashed line in $\mf$ in Figure \ref{fig: M_{b, 0}} stands for $\partial_{sc} \diag_{M_{b, 0}}$.

The multi-stretched double space $M_{b, 0}$ is the suitable double space for the resolvent kernel of the operator $\tilde{P}_\hbar$ defined on the stretched single space. 
On the one hand, the above blow-up scheme $(\beta_0 \circ \beta_{b})^{-1}$ to create the multi-stretched spaces $M_{b, 0}$ from $\mathcal{C}^2 \times [0, 1)$ is designed in accordance with the variable changing performed from $P_\hbar$ to $\tilde{P}_\hbar$. On the other hand, the original differential operator $P_\hbar$, defined on the single space $\mathcal{C} \times [0, 1)$ with a parameter $\hbar$, lifts to $P_b$ and $P_{sc}$ locally defined on the stretched single space $\mathcal{C}_b$. Although there is no canonical projection from the multi-stretched double space $M_{b, 0}$ to the stretched single space $\mathcal{C}_b$, one can prove there exists a left (right) stretched projection $\tilde{\mathcal{P}}_{L} (\tilde{\mathcal{P}}_{R}) : M_{b, 0} \rightarrow \mathcal{C}_b$ by invoking \cite[Lemma 2.7]{Hassell-Mazzeo-Melrose} and the fact \[M_{b, 0} = [[\mathcal{C}_b \times \mathcal{C} ;  \lb_0 \times \partial \mathcal{C}] ; \lb \times \partial \mathcal{C}, \mf \times  \partial \mathcal{C}].\]

In summary, the operator $P_\hbar$ defined on $M$ is lifted, via the left (right) stretched projection $\tilde{\mathcal{P}}_{L} (\tilde{\mathcal{P}}_{R})$, to $\tilde{P}_\hbar$ defined on $M_{b, 0}$, which acts on the left (right) 'half' of $M_{b, 0}$. The operator $\tilde{P}_\hbar$ is further converted to a $b$-operator $P_b$ in \eqref{eqn : b-operator} near $\mathrm{bf}$ for small $r$ $(r')$, whilst it is a $sc$-operator $P_{sc}$ in \eqref{eqn : scattering-operator} near $\mf$ for small $\rho$ $(\rho')$.

In addition,   the $sc$-calculus requires a further blow-up on $M_{b, 0}$ as we shall see. Then we create the face $\mathrm{sc}$ by blowing up $\partial_{sc} \diag_{M_{b, 0}}$ \begin{equation}\label{fig : M_{sc, b, 0}}\beta_{sc} : M_{sc, b, 0} = [M_{b, 0}; \partial_{sc} \diag_{M_{b, 0}}] \longrightarrow M_{b, 0},\end{equation} as in Figure \ref{blown-upspace(new)}. 
 
 We shall mostly live in $M_{b, 0}$ and only visit $M_{sc, b, 0}$ briefly while constructing the $sc$-pseudodifferential parametrix of $P_{sc}$.

 \begin{center}\begin{figure}
 		\includegraphics[width=0.7\textwidth]{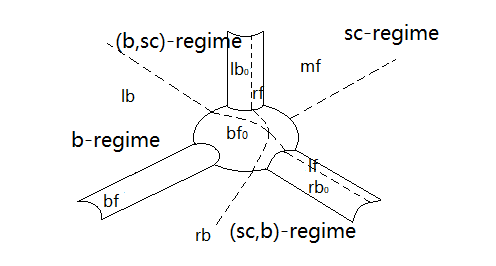}\caption{\label{fig: M_{b, 0}}The space $M_{b, 0} = \beta_b^{-1} \beta_0^{-1} M$}\end{figure} 
 	\begin{figure}
 		\includegraphics[width=0.7\textwidth]{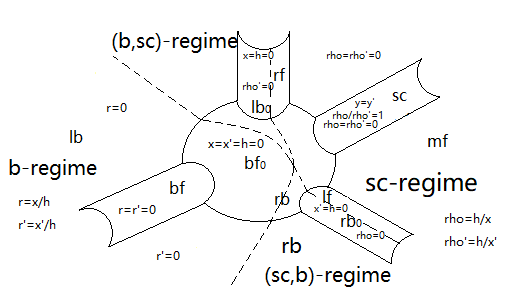}\caption{\label{blown-upspace(new)}The space $M_{sc, b, 0} = \beta_{sc}^{-1} \beta_b^{-1} \beta_0^{-1} M$}\end{figure} \end{center}

\subsection{An atlas of $M_{b, 0}$}

 To separate the distinct behaviours of $\tilde{P}_\hbar$ as the left (right) operator on $M_{b, 0}$, we  decompose the stretched space $M_{b, 0}$ (and $M_{sc, b, 0}$ similarly), by the defining functions $$r = x/\hbar, \quad r' = x'/\hbar,  \quad \rho = \hbar/x,  \quad \rho' = \hbar/x',$$ into the following regimes.  The two intersecting dashed curves in Figure \ref{fig: M_{b, 0}} and Figure \ref{blown-upspace(new)} stand for the borders of the regimes over the boundary faces.

\begin{itemize}

	\item The $b$-regime is defined by $\{r < 1\} \cap \{r' < 1\} \subset M_{b, 0}$, furnished with local coordinates $(r, y, r', y', \hbar)$, and has boundary hypersurfaces:
	\begin{itemize}

		\item $\lb = \overline{\{r= 0, 0 < r' < 1\}}$ is the left boundary in the $b$-calculus for $P_b$,
		
		\item $\rb = \overline{\{0 < r < 1, r' = 0\}}$ is the right boundary in the $b$-calculus for $P_b$,
		
		\item $\mathrm{bf} =  \overline{\{r=0, r'=0\}}$ is the front face in the $b$-calculus for $P_b$.
		
	\end{itemize}

	\item The $(b, sc)$-regime is defined by $\{r < 1\} \cap \{\rho' < 1\} \subset M_{b, 0}$, furnished with local coordinates $(r, y, \rho', y', \hbar)$, and has boundary hypersurfaces:
\begin{itemize}
	\item   $\lb_0 = \overline{\{0 < r < 1, \rho' = 0\}}$,	
	\item   $\lb = \overline{\{0 < \rho' < 1, r=0\}}$.
\end{itemize}

	\item The $sc$-regime is defined by $\{\rho < 1\} \cap \{\rho' < 1\} \subset M_{b, 0}$, furnished with local coordinates $(\rho, y, \rho', y', \hbar)$, and has boundary hypersurfaces:
	
	\begin{itemize}
		\item $\lf  = \overline{\{\rho = 0, 0 < \rho' < 1\}} \subset \mathrm{rb}_0$, defined by $\rho$, acts as the left face in the $sc$-calculus for $P_{sc}$, hence the name $\lf$,
		\item $\rf = \overline{\{\rho' = 0, 0 < \rho < 1\}} \subset \mathrm{lb}_0$, defined by $\rho'$, acts as the right face in the $sc$-calculus for $P_{sc}$, hence the name $\rf$,
		\item $\mf =  \overline{\{\rho=0, \rho'=0\}}$ is the main face, which plays a central role in the discussion of Legendrian structures in the $sc$-calculus.
		
	\end{itemize}

	\item The $(sc, b)$-regime is defined by $\{\rho < 1\} \cap \{r' < 1\} \subset M_{b, 0}$, furnished with local coordinates $(\rho, y, r', y', \hbar)$, and has boundary hypersurfaces:
	
	\begin{itemize}
		\item   $\rb_0 = \overline{\{0 < r' < 1, \rho =0\}}$,	
		\item   $\rb = \overline{\{0 < \rho < 1, r' = 0\}}$.

	\end{itemize}

	\item The transitional face $\bfz$ is defined by  $\overline{\{x^2 + x'^2 + \hbar^2 = 0\}} \subset M_{b, 0}$ and likewise decomposed into four pieces by the defining functions $r, r', \rho, \rho'$. Restricted to $\bfz$, the operator $\tilde{P}_\hbar$ reduces to $\Delta_{\mathfrak{C}} - (1 + \imath 0)$. Therefore, $\bfz$ can be viewed as a stretched double metric cone $\mathfrak{C}^2$, whilst the resolvent kernel $(\Delta_{\mathfrak{C}} - (1 + \imath 0))^{-1}$ on a stretched version of $\mathfrak{C}^2$ is equivalent to the resolvent kernel $\tilde{P}_\hbar^{-1}$ on $\bfz$. This will be discussed in detail in Section \ref{Sec: metric cone resolvent}.
	
	\item The interior of $M_{b, 0}$ is defined by a neighbourhood of $(\mathcal{C}\setminus\mathcal{K})^2 \times [0, 1)$ and furnished with local coordinates $(z, z', \hbar)$. We remark that the interior does have a boundary face in $\{\hbar = 0\}$ but it has nothing to do with the conic singularities of $\mathcal{C}$. We view the interior as a subset of the $sc$-regime as we did for the stretched single space.
	
\end{itemize}

% We denote by $\bfz$ the front face of $M_0$ and by $\mathrm{bf}$, $\mathrm{lb}_0$, $\mathrm{rb}_0$ the additional front faces of $M_{b, 0}$, and by  $\diag_b, \diag_{M_{b, 0}}$ the diagonals in the $b$-regime and $sc$-regime. In particular, the high energy face lifts to the main face ($\mf$) of the $sc$-regime.
% whilst the  diagonal near $\mf$ is $$\diag_{sc} = \overline{\{\rho = \rho', y=y', \rho>0\}},$$ where $\rho=1/r$ and $\rho'=1/r'$.
%The two dashed curves lying in $\bfz \cup \mathrm{lb} \cup \mathrm{rb}_0$ and $\bfz \cup \mathrm{rb} \cup \mathrm{lb}_0$ divide the boundaries into the four regimes, whilst the dashed line in $\mf$ stands for the boundary of the $sc$-diagonal in the $sc$-regime, $$\partial_{sc} \diag_{M_{b, 0}} = \{\rho = \rho' = 0, y = y', \rho/\rho'=1\}.$$  Moreover, the half faces $\{\rho \leq 1\} \cap \mathrm{lb}_0$ and $\{\rho' \leq 1\} \cap \mathrm{rb}_0$ in the $sc$-regime, somehow confusingly, act as the right face (where $\rho'=0$) and the left face (where $\rho=0$) in the $sc$-parametrix construction. So we relabel them, only in the $sc$-regime, as $\rf = \{\rho \leq 1\} \cap \mathrm{lb}_0$ and $\lf = \{\rho \leq 1\} \cap \mathrm{rb}_0$.

	\subsection{The compressed cotangent bundle and density of $M_{b, 0}$}
	
Understanding the microlocal structure of the resolvent kernel on $M_{b, 0}$, especially the singularities near the boundary, requires a suitable compressed cotangent bundle.

By the left and right stretched projections $\tilde{\mathcal{P}}_{L}$ and $\tilde{\mathcal{P}}_{R}$, we pull back the space $\mathcal{V}_{b, sc}(\mathcal{C}_b)$ of vector fields on the stretched single space $\mathcal{C}_b$ to the space $\mathcal{V}_{b, sc}(M_{b, 0})$ of vector fields on the stretched double space $M_{b, 0}$,
$$\mathcal{V}_{b, sc}(M_{b, 0}) = \tilde{\mathcal{P}}_{L}^\ast (\mathcal{V}_{b, sc}(\mathcal{C}_b)) + \tilde{\mathcal{P}}_{R}^\ast (\mathcal{V}_{b, sc}(\mathcal{C}_b)).$$

This Lie algebra determines the smooth sections of the compressed tangent bundle ${}^{b, sc}TM_{b, 0}$ and then defines the desired compressed cotangent bundle  ${}^{b, sc}T^\ast M_{b, 0}$ as well as the compressed density bundle  ${}^{b, sc}\Omega(M_{b, 0})$.
	
	%The operator $\tilde{P}_\hbar$ is a $b$-operator in the $b$-regime and a $sc$-operator in the $sc$-regime respectively

%It is useful to keep in mind several obvious facts on $M_{sc, b, 0}$ to fit into Figure \ref{fig: M_{b, 0}} for the $b$-calculus and $sc$-calculus in the parametrix construction.

With the atlas of $M_{b, 0}$, we have the following local expressions.

\begin{itemize}

\item In the $b$-regime $\{r < 1\} \cup \{r' < 1\}$, under coordinates $(r, y, r', y', \hbar)$,
\begin{itemize}
\item  the compressed cotangent bundle ${}^{b, sc}T^\ast M_{b, 0}$ is spanned by singular $1$-forms \[\frac{dr}{r}, \quad  \frac{dr'}{r'},  \quad dy, \quad  dy',  \quad \frac{d\hbar}{\hbar},\] and essentially a $b$-cotangent bundle;

\item  the compressed density ${}^{b, sc}\Omega(M_{b, 0})$  thus takes the form, 
  \[ 
 \bigg|\frac{dr dy dr' dy' d\hbar}{rr'\hbar}\bigg|,
   \] and is essentially a $b$-density.

\end{itemize}
\item In the $sc$-regime $\{\rho < 1\} \cap \{\rho' < 1\}$, under coordinates $(\rho, y, \rho', y', \hbar)$,

\begin{itemize}
\item  the compressed cotangent bundle ${}^{b, sc}T^\ast M_{b, 0}$ is spanned by singular $1$-forms \[\frac{d\rho}{\rho^2}, \quad  \frac{d\rho'}{\rho'^2},  \quad \frac{dy}{\rho},  \quad \frac{dy'}{\rho'},  \quad \frac{d\hbar}{\hbar},\] and essentially a $sc$-cotangent bundle;

\item  the compressed density ${}^{b, sc}\Omega(M_{b, 0})$  thus takes the form, 
\[ 
\bigg|\frac{d\rho dy d\rho' dy' d\hbar}{\rho^{n+1}\rho'^{n+1}\hbar}\bigg|,
\] and is essentially a $sc$-density.
\end{itemize}

\item In the $(b, sc)$-regime $\{r < 1\} \cap \{\rho' < 1\}$, under coordinates $(r, y, \rho', y', \hbar)$,

\begin{itemize}
 \item  the compressed cotangent bundle ${}^{b, sc}T^\ast M_{b, 0}$ is spanned by singular $1$-forms \[\frac{dr}{r},  \quad \frac{d\rho'}{\rho'^2}, \quad  dy, \quad  \frac{dy'}{\rho'},  \quad \frac{d\hbar}{\hbar},\]
\item  the compressed density ${}^{b, sc}\Omega(M_{b, 0})$  thus takes the form, \[ 
	\bigg|\frac{dr dy d\rho'dy'  d\hbar}{r \rho'^{n+1} \hbar}\bigg|.
  \]

\end{itemize}

\item In the $(sc, b)$-regime $\{\rho < 1\} \cap \{r' < 1\}$, under coordinates $(\rho, y, r', y', \hbar)$,

\begin{itemize}
	\item    the compressed cotangent bundle ${}^{b, sc}T^\ast M_{b, 0}$ is spanned by singular $1$-forms \[\frac{d\rho}{\rho^2}, \quad \frac{dr'}{r'},  \quad \frac{dy}{\rho},  \quad dy',  \quad \frac{d\hbar}{\hbar},\]
	\item 	 the compressed density ${}^{b, sc}\Omega(M_{b, 0})$  thus takes the form,  \[
 \bigg|\frac{ d\rho dy dr' dy' d\hbar}{\rho^{n+1}r' \hbar}\bigg|.
  \]

\end{itemize}

\item The transitional region, $\bfz$, can be viewed as the space where the resolvent kernel of $\Delta_{\mathfrak{C}}$ lives, which will be discussed in detail in Section \ref{Sec: metric cone resolvent}.

\end{itemize}

\section{The abstract $b$-calculus}\label{sec : abstract b calculus}

In this section, we shall review the abstract $b$-calculus. The notations in this section will be self-contained and independent of the conic manifold $\mathcal{C}$.

\subsection{The small $b$-calculus}

Suppose $X$ is an $n$-manifold which near any boundary point has a coordinate chart $\{(x, y) : x \in [0, 1), y \in \mathbb{R}^{n-1}\}$. The space $\mathcal{V}_b(X)$ of $b$-vector fields is the collection of smooth sections of $T X$ tangent to the boundary $\partial X$. Then the $b$-vector fields are smooth sections of the $b$-tangent bundle ${}^b T X$, which has a local frame $\{x\partial_x, \partial_y\}$ near the boundary. The $b$-cotangent bundle ${}^b T^\ast X$ is the dual bundle of ${}^b T X$.

The space $\mathrm{Diff}_b^m (X)$ of $m$th order $b$-differential operators consists of finite sums of $m$-fold products of $b$-vector fields. Near $\partial X$, a $b$-differential operator $A \in \mathrm{Diff}_b^m (X)$ acts on $b$-half densities  $$f \bigg|\frac{dx dy}{x}\bigg|^{1/2} \in C^\infty(X; {}^b\Omega^{1/2}),$$ and takes the explicit form, $$A = \sum_{\alpha + |\beta| \leq m} a_{\alpha\beta}(x, y) (x\partial_x)^\alpha \partial_y^\beta, \quad \mbox{$\alpha \in \mathbb{N}_0$ and $\beta \in \mathbb{N}_0^{n-1}$}.$$

The $b$-differential operators can be generalized to $b$-pseudodifferential operators. The Schwartz kernel of a $b$-pseudodifferential operator is a distribution living in the $b$-stretched space  $$X^2_b = \beta_b^{-1} (X^2) = [X^2; \partial X \times \partial X]$$ with three boundary faces
\begin{align*}
&\mathrm{bf} = \overline{\beta^{-1}_b (\partial X \times \partial X)}\\
&\mathrm{lb} = \overline{\beta^{-1}_b (\partial X \times X^\circ)}\\
&\mathrm{rb} = \overline{\beta^{-1}_b (X^\circ \times \partial X)}.\end{align*} and  conormal to the $b$-diagonal $\diag_b = \overline{\beta_b^{-1} (\diag^\circ)}$. To be concrete,
\begin{definition}
The space $\Psi^m_b(X; {}^b\Omega^{1/2})$ is the space of operators with a kernel $\kappa |d\mu_b|^{1/2}$ living in $X_b^2$ such that
\begin{itemize}
\item $\kappa$ is a distribution, conormal of degree $m$ to $\diag_b$, smooth up to $\mathrm{bf}$, and rapidly vanishing at $\mathrm{lb}$ and $\mathrm{rb}$.
\item $|d\mu_b|^{1/2}$ is the lift of $\pi_l^\ast {}^b\Omega^{1/2} \otimes \pi_r^\ast {}^b\Omega^{1/2}$ to $X_b^2$, where $\pi_l$ and $\pi_r$ are the projections from $X_b^2$ down to a single space $X$.
\end{itemize}
\end{definition}

In analogy with classical pseudodifferential operators,  $b$-pseudodifferential operators have a well-defined symbol map $$\Psi_b^m(X; {}^b\Omega^{1/2}) \longrightarrow S^m({}^bT^\ast X)/S^{m-1}({}^bT^\ast X),$$ where $S^m(\cdot)$ is the classical symbol space of degree $m$. This is extended to a short exact sequence
$$0 \rightarrow \Psi_b^{m-1}(X; {}^b\Omega^{1/2}) \hookrightarrow \Psi_b^{m}(X; {}^b\Omega^{1/2}) \rightarrow S^m({}^bT^\ast X) \rightarrow 0.$$

\subsection{The full $b$-calculus}

Since the small $b$-calculus only addresses the singularities of the resolvent near $\diag_b$, the full $b$-calculus is introduced to describe the boundary behaviours of the true resolvent.

First of all, the boundary singularities are measured in terms of $b$-Sobolev spaces and polyhomogeneous conormal distributions to boundaries. The $b$-Sobolev space of order $m$ is defined by $b$-differential operators as \begin{equation}\label{eqn : b-Sobolev}H_b^m(X; {}^b\Omega^{1/2}) = \{u \in L^2(X; {}^b\Omega^{1/2}) : \mathrm{Diff}_b^{m}(X; {}^b\Omega^{1/2}) u \subset L^2(X; {}^b\Omega^{1/2})\}.\end{equation}

The space of conormal distributions of degree $(\alpha_{\mathrm{bf}}, \alpha_\lb, \alpha_\rb)$ to $(\mathrm{bf}, \lb, \rb)$  on $X^2_b$, denoted by $\mathcal{A}^{(\alpha_{\mathrm{bf}}, \alpha_\lb, \alpha_\rb)}(X_b^2; {}^b\Omega^{1/2})$, consists of distributions such that $$\{u : \mathcal{V}_b u \subset \rho_{\mathrm{bf}}^{\alpha_{\mathrm{bf}}} \rho_\lb^{\alpha_\lb} \rho_\rb^{\alpha_\rb} L^\infty(X_b^2; {}^b\Omega^{1/2})\} \quad \mbox{with $(\alpha_{\mathrm{bf}}, \alpha_\lb, \alpha_\rb)\in \mathbb{R}^3$},$$ where $\rho_{\mathrm{bf}}$, $\rho_\lb$ and $\rho_\rb$ are defining functions for $\mathrm{bf}$, $\lb$ and $\rb$ respectively.

The polyhomogeneity for a boundary $\bullet$ is measured by the index set $\mathcal{E}_\bullet$, which is a discrete subset of $\mathbb{C} \times \mathbb{N}_0$ such that  \begin{itemize}
\item $(z_j, k_j) \in \mathcal{E}_\bullet, |(z_j, k_j)| \rightarrow \infty \Rightarrow \Re z_j \rightarrow \infty$,
\item $(z, k) \in \mathcal{E}_\bullet \Rightarrow (z, l) \in \mathcal{E}_\bullet, l \in \mathbb{N}_0, 0 \leq l \leq k$,
\item $(z, k) \in \mathcal{E}_\bullet \Rightarrow (z + j, k) \in \mathcal{E}_\bullet, j \in \mathbb{N}_0$.
\end{itemize}
We denote $\Re \mathcal{E}_\bullet = \min\{\Re z : \exists (z, k) \in \mathcal{E}_\bullet\}$ for $\mathcal{E}_\bullet$.

A collection $\mathcal{E}$ of index sets for each boundary hypersurface is called the index family.
The space $\mathcal{A}_{phg}^{(\mathcal{E}_{\mathrm{bf}}, \mathcal{E}_\lb, \mathcal{E}_\rb)}(X_b^2; {}^b\Omega^{1/2})$ of polyhomogeneous conormal distributions to $\mathrm{bf}$, $\lb$ and $\rb$ consists of distributions having the following local expressions
\begin{eqnarray*}  u \sim \sum_{(z_{\mathrm{bf}}, k_{\mathrm{bf}}) \in \mathcal{E}_{\mathrm{bf}}} \rho_{\mathrm{bf}}^{z_{\mathrm{bf}}} (\log \rho_{\mathrm{bf}})^{k_{\mathrm{bf}}}  a_{(z_{\mathrm{bf}}, k_{\mathrm{bf}})}\quad \mbox{with $ a_{(z_{\mathrm{bf}}, k_{\mathrm{bf}})} \in C^\infty(X^2_b; {}^b\Omega^{1/2})$ }, \quad \mbox{near $\mathrm{bf}$};
	\\   u \sim \sum_{(z_\lb, k_\lb) \in \mathcal{E}_\lb} \rho_\lb^{z_\lb} (\log \rho_\lb)^{k_\lb} a_{(z_\lb, k_\lb)}\quad \mbox{with $a_{(z_\lb, k_\lb)} \in C^\infty(X^2_b; {}^b\Omega^{1/2})$ }, \quad \mbox{near $\lb$};\\
  u \sim \sum_{(z_\rb, k_\rb) \in \mathcal{E}_\rb} \rho_\rb^{z_\rb} (\log \rho_\rb)^{k_\rb} a_{(z_\rb, k_\rb)}\quad \mbox{with $a_{(z_\rb, k_\rb)} \in C^\infty(X^2_b; {}^b\Omega^{1/2})$ }, \quad \mbox{near $\rb$}.\end{eqnarray*}
To measure the polyhomogeneity of the product and composition of two conormal distributions, the operations of addition and extended union for two index sets $\mathcal{E}_1$ and $\mathcal{E}_2$ are defined to be
\begin{align*}
\mathcal{E}_1+\mathcal{E}_2 &= \{(z_1 + z_2, k_1 + k_2) | (z_1, k_1) \in \mathcal{E}_1, (z_2, k_2) \in \mathcal{E}_2\}\\
\mathcal{E}_1 \, \overline{\cup} \, \mathcal{E}_2 &= \mathcal{E}_1 \cup \mathcal{E}_2 \cup \{(z, k) : \mbox{$\exists$ $(z, k_1) \in \mathcal{E}_1$ and $(z, k_2) \in \mathcal{E}_2$ with $k = k_1 + k_2 + 1$}\}.
\end{align*} To measure the least power which can occur, the infimum of an index set $\mathcal{E}_\cdot$ is defined by \[\inf \mathcal{E}_\cdot = \min \{\Re z : (z, 0) \in \mathcal{E}_\cdot\}.\]

The space $\Psi_b^{m, \mathcal{E}}(X_b^2; {}^b\Omega^{1/2})$ of full $b$-pseudodifferential operators with $\mathcal{E} = (\mathcal{E}_{\lb}, \mathcal{E}_{\rb}, \mathcal{E}_{\mathrm{bf}})$ is the collection of the sums of small $b$-pseudodifferential operators $\Psi_b^{m}(X_b^2; {}^b\Omega^{1/2})$ and the operators with a polyhomogeneous conormal kernel in $\mathcal{A}_{phg}^{(\mathcal{E}_{\lb}, \mathcal{E}_{\rb}, \mathcal{E}_{\mathrm{bf}} )}(X_b^2; {}^b\Omega^{1/2})$. %Specifically,
%\begin{definition}
%We say a distribution $A \in \Psi_b^{m, \mathcal{E}}(X_b^2; {}^b\Omega^{1/2})$ with $\mathcal{E} = (\mathcal{E}_{\mathrm{bf}}, \mathcal{E}_{\lb}, \mathcal{E}_{\rb})$,  if 
 %$A$ is a $b$-pseudodifferential operator in $\Psi_b^{m}(X_b^2; {}^b\Omega^{1/2})$, whilst the kernel of $A$ is also a polyhomogeneous conormal distribution in $ \mathcal{A}_{phg}^{(\mathcal{E}_b, \mathcal{E}_l, \mathcal{E}_r)}(X_b^2; {}^b\Omega^{1/2})$.
%\end{definition}

Second, the boundary behaviour of the resolvent of $A$ involves the indicial operator  $\mathcal{I}(A)$,\footnote{It is also called normal operator. In general, the normal operator $\mathcal{N}_\bullet(\cdot)$ is the restriction of $\cdot$ to the face $\bullet$.  When the indicial family is further involved, we call it indicial operator, to be consistent with the convention in \cite{Melrose-APS}. Otherwise, it is named normal operator. } which is the restriction of $A$ to $\partial X$. It has a local expression, $$\mathcal{I}(A) = \sum_{\alpha + |\beta| = m} a_{\alpha\beta}(0, y) (x\partial_x)^\alpha \partial_y^\beta.$$

In the same spirit with the Fourier transform approach to evolution equations, the inversion of $\mathcal{I}(A)$, via the Mellin transform $\mathcal{M}$,
\[\mathcal{M} (u)(\lambda)  = \int_0^\infty x^{-i\lambda} u(x, y) \frac{dx}{x},\] reduces to the indicial family $\mathcal{I}(A, \lambda)$,  
\begin{align*}  
  \mathcal{I}(A, \lambda) &= \mathcal{M}(\mathcal{I}(A))(\lambda) = \sum_{\alpha + |\beta| = m} a_{\alpha\beta}(0, y) (\imath\lambda)^\alpha \partial_y^\beta. \end{align*} 
The Mellin transform gives an isomorphism 
\[x^\alpha L^2(\mathbb{R}_+, \frac{dx}{x}) \longrightarrow L^2(\{\Im \lambda = \alpha\}, d\lambda),\]  for any $\alpha \in \mathbb{R}$. The inverse Mellin transform corresponds to integration along $\{\Im \lambda = \alpha\}$, \[\mathcal{M}^{-1}(v)(x)= \frac{1}{2\pi} \int_{\{\Im \lambda = \alpha\}}  x^{\imath \lambda} v(\lambda, y)d\lambda.\]

To address the invertibility of the indicial operator, we define the boundary spectrum,
\begin{align*}
\mathrm{spec}_b (\mathcal{I}(A)) &= \Big\{\lambda \in \mathbb{C} : \mbox{$\mathcal{I}(A, \lambda)$ is not invertible on $C^\infty(\partial X)$}\Big\}, \\
\mathrm{Spec}_b (\mathcal{I}(A)) &= \Big\{(\lambda, k) : \mbox{the inverse of $\mathcal{I}(A, \cdot)$ has a pole of order $k+1$ at $\lambda$} \Big\}.
\end{align*}

%We employ polyhomogeneous conormal distributions to the boundaries of $X^2_b$ to describe the boundary behaviour of the resolvent kernel of a $b$-differential operator on $X$. The double space $X^2$ is a manifold with corner, which is locally homeomorphic  to \begin{itemize}\item $\{(x, y, x', y') : x, x' \in [0, \infty), y, y' \in \mathbb{R}^{n-1}\}$ near the corner; \item $\{(x, y, z') : x \in [0, \infty), y \in \mathbb{R}^{n-1}, z' \in \mathbb{R}^n\}$ near the left boundary $\lb$ of $X^2$; \item $\{(z, x', y') :  z \in \mathbb{R}^n, x' \in [0, \infty), y' \in \mathbb{R}^{n-1}\}$ near the right boundary $\rb$ of $X^2$.\end{itemize} Then $x$ and $x'$ define $\lb$ and $\rb$ respectively. The space $X^2_b$ is obtained by blowing up the corner $\partial X \times \partial X = \{(0, y, 0, y')\}$. Then the new left and right boundaries, still denoted by $\lb$ and $\rb$, have new defining functions, say $\rho_l$ and $\rho_r$.

Through the Mellin transform, the invertibility of $\mathcal{I}(A, \lambda)$ amounts to the invertibility of $\mathcal{I}(A)$ on weighted $b$-Sobolev spaces.
\begin{proposition}The indicial operator $\mathcal{I}(A)$ of $A \in \mathrm{Diff}_b^m(X)$ is an isomorphism between $x^\alpha H_b^j(X; {}^b\Omega^{1/2})$ and $x^\alpha H_b^{j - m}(X; {}^b\Omega^{1/2})$ for all $j$ if and only if $\alpha \in \mathbb{R}$, $\alpha \neq - \Im \lambda$ for any $\lambda \in \mathrm{spec}_b (\mathcal{I}(A))$.\end{proposition}
%Furthermore, the inverse of the indicial operator is a full $b$-pseudodifferential operator.
%\begin{proposition}If  $\alpha  \not\in \Im\mathrm{spec}_b (\mathcal{I}(A))$, the inverse of $\mathcal{I}(A)$ is a polyhomogeneous conormal distribution to $\lb$ and $\rb$ on $X_b^2$ lying in $$\mathcal{A}_{phg}^{(\mathcal{E}_{\lb}, \mathcal{E}_{\rb})}(X_b^2) \quad \mbox{modulo $\Psi_b^\ast(X_b^2)$},$$ where the index sets read  \end{proposition}

Finally, the composition of the full $b$-pseudodifferential operators will be useful in the construction of the parametrix. See \cite[Theorem 4.20]{Albin}. \begin{proposition} 
For two full $b$-pseudodifferential operators on $X_b^2$, \begin{eqnarray*}
	A \in \Psi_b^{m_1, \mathcal{E}}(X^2_b, {}^b\Omega^{1/2}) &&\mbox{with $\mathcal{E} = (\mathcal{E}_\lb, \mathcal{E}_\rb, \mathcal{E}_{\mathrm{bf}})$}, \\
	B \in \Psi_b^{m_2, \mathcal{F}}(X^2_b, {}^b\Omega^{1/2}) &&\mbox{with $\mathcal{F} = (\mathcal{F}_\lb, \mathcal{F}_\rb, \mathcal{F}_{\mathrm{bf}})$}, 
\end{eqnarray*}    satisfying   $\Re (\mathcal{E}_{\rb} + \mathcal{F}_{\lb}) > 0$, the composition operator $A \circ B$ lies in the space $$\Psi_b^{m_1 + m_2, \mathcal{G}}(X^2_b, {}^b\Omega^{1/2})  \quad \mbox{with $\mathcal{G} = (\mathcal{G}_\lb, \mathcal{G}_\rb, \mathcal{G}_{\mathrm{bf}})$}$$ where the index family $\mathcal{G}$ is given by \begin{align*}
	 \mathcal{G}_{\lb} &= (\mathcal{E}_{\mathrm{bf}} + \mathcal{F}_{\lb})\overline{\cup} \mathcal{E}_{\lb}, \\ \mathcal{G}_{\rb} &= (\mathcal{E}_{\rb} + \mathcal{F}_{\mathrm{bf}}) \overline{\cup} \mathcal{F}_{\rb}, \\ \mathcal{G}_{\mathrm{bf}}  &= (\mathcal{E}_{\mathrm{bf}} + \mathcal{F}_{\mathrm{bf}}) \overline{\cup} (\mathcal{E}_{\lb} + \mathcal{F}_{\rb} ).
\end{align*} 

\end{proposition}

\section{The $b$-calculus on $M_{b, 0}$}\label{sec : b calculus on $M_{b, 0}$}

\subsection{The $b$-pseudodifferential operator on $M_{b, 0}$}

We now  apply the abstract $b$-calculus to the $b$-regime of $M_{b, 0}$, where $r < 1$ and $r' < 1$. The relevant boundary hypersurfaces here are $\mathrm{bf}$, $\lb$, $\rb$ and $\bfz$.

Suppose $m \in \mathbb{R}$ and $\mathcal{E} = (\mathcal{E}_{\lb}, \mathcal{E}_{\rb}, \mathcal{E}_{\mathrm{bf}}, \mathcal{E}_{\bfz})$ is the index family for the boundary. The space $\Psi_b^{m, \mathcal{E}}(M_{b, 0}; {}^{b, sc}\Omega^{1/2})$ of $b$-pseudodifferential operators in the $b$-regime is defined to be the collection of the sums of small $b$-pseudodifferential operators in $\Psi_b^{m}(M_{b, 0}; {}^{b, sc}\Omega^{1/2})$ and operators with a polyhomogeneous conormal kernel in $\mathcal{A}_{phg}^{ \mathcal{E} }(M_{b, 0}; {}^{b, sc}\Omega^{1/2})$ supported in the $b$-regime, where \begin{itemize}
	\item $\Psi_b^{m}(M_{b, 0}; {}^{b, sc}\Omega^{1/2})$ is the space of ${}^{b, sc}\Omega^{1/2}$-valued distributions conormal of degree $m$ to the diagonal $\diag_{M_{b, 0}}|_{b\mathrm{-regime}}$, uniformly up to $\mathrm{bf}$ and $\bfz$;
	\item $\mathcal{A}_{phg}^{ \mathcal{E} }(M_{b, 0}; {}^{b, sc}\Omega^{1/2})$  is the space of ${}^{b, sc}\Omega^{1/2}$-valued polyhomogeneous conormal distributions with respect to the index family $\mathcal{E}$ for $\lb$, $\rb$, $\mathrm{bf}$ and $\bfz$.
\end{itemize}

Adapting the abstract $b$-calculus, we have analogous composition formulas.
\begin{proposition}\label{prop : phg mapping}
	For two full $b$-pseudodifferential operators in the $b$-regime, \begin{eqnarray*}
		A \in \Psi_b^{m_1, \mathcal{E}}(M_{b, 0}, {}^{b, sc}\Omega^{1/2}) &&\mbox{with $\mathcal{E} = (\mathcal{E}_\lb, \mathcal{E}_\rb, \mathcal{E}_{\mathrm{bf}}, \mathcal{E}_{\bfz})$}, \\
		B \in \Psi_b^{m_2, \mathcal{F}}(M_{b, 0}, {}^{b, sc}\Omega^{1/2}) &&\mbox{with $\mathcal{F} = (\mathcal{F}_\lb, \mathcal{F}_\rb, \mathcal{F}_{\mathrm{bf}}, \mathcal{F}_{\bfz})$}, 
	\end{eqnarray*}    satisfying   $\Re (\mathcal{E}_{\rb} + \mathcal{F}_{\lb}) > 0$, the composition operator $A B$ lies in the space $$\Psi_b^{m_1 + m_2, \mathcal{G}}(M_{b, 0}, {}^{b, sc}\Omega^{1/2})  \quad \mbox{with $\mathcal{G} = (\mathcal{G}_{\mathrm{bf}}, \mathcal{G}_\lb, \mathcal{G}_\rb, \mathcal{G}_{\bfz})$}$$ where the index family $\mathcal{G}$ is given by \begin{align*}
		\mathcal{G}_{\lb} &= (\mathcal{E}_{\mathrm{bf}} + \mathcal{F}_{\lb})\overline{\cup} \mathcal{E}_{\lb}, \\ \mathcal{G}_{\rb} &= (\mathcal{E}_{\rb} + \mathcal{F}_{\mathrm{bf}}) \overline{\cup} \mathcal{F}_{\rb}, \\ \mathcal{G}_{\mathrm{bf}}  &= (\mathcal{E}_{\mathrm{bf}} + \mathcal{F}_{\mathrm{bf}}) \overline{\cup} (\mathcal{E}_{\lb} + \mathcal{F}_{\rb} ), \\ \mathcal{G}_{\bfz}  &= \mathcal{E}_{\bfz} + \mathcal{F}_{\bfz}.
	\end{align*} 
	
\end{proposition}

\subsection{The indicial operator at $\mathrm{bf}$}

We next consider the invertibility of the indicial operator of $P_b$ at $\mathrm{bf}$.

Recall from \eqref{eqn : b-operator} that $P_b$ is a second order $b$-differential operator and takes the form $$-(r\partial_r)^2 - r\tilde{e} (r\partial_r) + \Delta_h -  r \tilde{e}(1 - n/2) - r^2 (1 + \imath 0) +  (1 - n/2)^2.$$
Near $\lb \cap \mathrm{bf}$, we adopt coordinates $$\{(\tilde{r}, r', y, y', \hbar) : \tilde{r} = r/r'\}.$$ Here $\tilde{r}$, $r'$ and $\hbar$ are the boundary defining functions for $\lb$, $\mathrm{bf}$ and $\bfz$ respectively. Then the blowup $\beta_b^{-1}$ lifts $P_b$ to $$-(\tilde{r}\partial_{\tilde{r}})^2 - \tilde{r}r'\tilde{e} (\tilde{r}\partial_{\tilde{r}}) + \Delta_h -  \tilde{r}r' \tilde{e}(1 - n/2) - (\tilde{r}r')^2 (1 + \imath 0) +  (1 - n/2)^2.$$
The indicial operator $\mathcal{I}(P_b)$ and the indicial family $\mathcal{I}(P_b, \lambda)$  read \begin{eqnarray*}&&\mathcal{I}(P_b) = - (\tilde{r}\partial_{\tilde{r}})^2 + \Delta_{h_0} +  (1 - n/2)^2.\\&& \mathcal{I}(P_b, \lambda) = \mathcal{M} (\mathcal{I}(P_b)) = \lambda^2 + \Delta_{h_0}  +  (1 - n/2)^2,\end{eqnarray*} where $h_0$ denotes $h(0, y, dy)$. Moreover, these operators lift, through the stretched projection $\tilde{\mathcal{P}}_L$, to the $b$-regime of $M_{b, 0}$. They are $\hbar$-independent so that the face $\bfz$ is irrelevant for the inverse of $\mathcal{I}(P_b)$ and   $\mathcal{I}(P_b, \lambda)$. 

Since $\Delta_{h_0}$ is the Laplacian on the compact manifold $(\mathcal{Y}, h_0)$, the resolvent family $(\Delta_{h_0} - \tau)^{-1}$ only has simple poles $$\sigma_0^2, \sigma_1^2, \sigma_2^2, \cdots \quad \mbox{with} \quad 0 = \sigma_0^2 < \sigma_1^2 < \sigma_2^2 < \cdots.$$ 
Then the indicial family $\mathcal{I}(P_b, \lambda)$ is invertible on $C^\infty(\partial \mathcal{C})$ if and only if $-\lambda^2 - (1 - n/2)^2$ is NOT an eigenvalue of $\Delta_{h_0}$. Hence, we define the indicial roots $$\lambda_j =  \imath ((1 - n/2)^2 + \sigma^2_j)^{1/2},$$ and the boundary spectrum
\begin{eqnarray*}
\mathrm{spec}_b (\mathcal{I}(P_b)) =  \{\pm \lambda_j \in \mathbb{C} \}, &&
\mathrm{Spec}_b (\mathcal{I}(P_b)) =  \{(\pm \lambda_j, 0)  \}.
\end{eqnarray*}

By adapting the proof from \cite[(5.84)]{Melrose-APS}, we obtain the inverse of the indicial operator $\mathcal{I}(P_b)$, which is the inverse Mellin transform of the indicial family at $\mathrm{bf}$.
\begin{proposition}\label{prop : inverse indicial operators}
  The kernel of $\mathcal{I}(P_b)^{-1}$,  modulo $\Psi_b^{-2}(M_{b, 0}; {}^{b, sc}\Omega^{1/2})$,  lies in the space, \begin{multline}\label{eqn : index set pm}\mathcal{A}_{phg}^{(\mathcal{E}_\lb, \mathcal{E}_\rb, \mathcal{E}_{\mathrm{bf}}, \mathcal{E}_{\bfz})}(M_{b, 0}; {}^{b, sc}\Omega^{1/2}), \\ \mbox{with $\mathcal{E}_\lb = \mathcal{E}_\rb = \{(\Im \lambda_j, 0) \}$ and  $\mathcal{E}_{\mathrm{bf}} = \mathcal{E}_{\bfz} = \mathbb{N}_0$}.\end{multline} 
  %where the space $\mathcal{A}_{phg}^{(\mathcal{E}_\lb, \mathcal{E}_\rb, \mathcal{E}_{\bfz})}(\mathrm{bf})$ is  the space of distributions on the boundary hypersurface $\mathrm{bf}$ polyhomogeneous conormal to $\lb$ and $\rb$ with respect to the index family $(\mathcal{E}_\lb, \mathcal{E}_\rb, \mathcal{E}_{\bfz})$.
  \end{proposition}

  %\begin{eqnarray*}\tilde{E}(n/2 - 1) &=& \{(z, 0) \in \mathbb{C} \times \mathbb{N}_0; \mbox{$(z - q, 0) \in E(n/2 - 1)$ for some $q \in \mathbb{N}_0$}\},\\ E(n/2 - 1) &=& \{(- \imath ((n/2 - 1)^2 + \sigma_j^2)^{1/2}, 0) \}.\end{eqnarray*}

%We remark that the operator $P_b$ in the $b$-regime has an extra parameter $\hbar$, which locally defines an extra boundary hypersurface $\bfz \subset M_{b, 0}$. But $\hbar$ acts as a smooth parameter of the first order term in $P_b$. On the other hand, the interior of $\bfz$ is contained in the 'interior' in the $b$-regime since $1 > r=x/\hbar > 0$ and $1 > r'=x'/\hbar > 0$ in the inferior of $\bfz$. So the $b$-calculus does not include the boundary behaviours near $\bfz$.  

\section{The abstract $sc$-calculus} \label{sec : abstract sc calculus}

In this section, we shall review the abstract scattering calculus and Legendrian distributions. The notations in this section will be self-contained and independent of the conic manifold $\mathcal{C}$. 
 
\subsection{The $sc$-pseudodifferential calculus}

Let $X$ be an $n$-manifold with boundary. It is furnished with local coordinates $z$ in the interior and $(x, y)$ near the boundary with a boundary defining function $x$. 

%We endow $X$ with a (scattering) $sc$-metric, which takes the form, near the boundary, $$g =  \frac{dx^2}{x^4} + \frac{h(x, y, dy)}{x^2},$$ where $h$ is a smooth tensor on $X$ and $h|_{x=0}$ is a smooth metric on $\partial X$.

Consider the Lie algebra,
$$\mathcal{V}_{sc}(X) = \{V | V = xW, \mbox{$W$ is a $C^\infty$ vector field tangent to $\partial X$}\}.$$ Near $\partial X$, this Lie algebra, smoothly spanned by the vector fields $x^2 \partial_x$ and $x\partial_{y}$, gives smooth sections of the $sc$-tangent bundle ${}^{sc}TX$. Corresponding to this, the $sc$-cotangent bundle ${}^{sc}T^\ast X$, near $\partial X$, is generated  by $dx/x^2$ and $dy/x$. In other words, a $sc$-cotangent vector, near $\partial X$, takes the form $$d(f/x) = \tau dx/x^2 + \mu dy/x$$ for some smooth function $f$.

The $sc$-differential operators of order $k$ are the sums of products of at most $k$ $sc$-vector fields. They act on $sc$-half densities, which, near the boundary, take the form $$f \Big|\frac{dxdy}{x^{n+1}}\Big|^{1/2} \in C^\infty(X; {}^{sc}\Omega^{1/2}).$$ The space of such operators is denoted by $\mathrm{Diff}_{sc}^k(X; {}^{sc}\Omega^{1/2})$.

We extend this to $sc$-pseudodifferential operators by viewing their Schwartz kernels as conormal distributions on certain double spaces. The appropriate double space to accommodate the $sc$-pseudodifferential operators is the two-fold blown-up space $$X_{sc, b} = [X_b^2; \partial  \diag_{X^2_b}] = [[X^2; (\partial X)^2]; \partial  \diag_{X^2_b}],$$ where $\diag_{X^2_b}$ is the lift of the diagonal of $X^2$ to $X_b^2$. The boundary hypersurfaces and the diagonal $\diag_{X^2}$ of $X^2$ lift via the two-fold blow-up   as follows,
\begin{eqnarray*}\partial X \times X &\longrightarrow& \lf, \\
X \times  \partial X &\longrightarrow& \rf, \\
( \partial X \times  \partial X ) \setminus \partial \diag_{X^2} &\longrightarrow& \mf, \\
\partial \diag_{X^2} &\longrightarrow& \mathrm{sc}, \\
\diag_{X^2} &\longrightarrow& \diag_{X^2_{sc, b}},\end{eqnarray*} where $\diag_{X^2_{sc, b}}$ denotes the diagonal of $X^2_{sc, b}$.

If we view the kernel of a $sc$-differential operator on $X$ as a distribution on $X_{sc, b}^2$, then it is microlocally supported on and conormal to $\diag_{X^2_{sc, b}}$. In the same spirit,
\begin{definition}
The space $\Psi_{sc}^{k, 0}(X; {}^{sc}\Omega^{1/2})$ of $sc$-pseudodifferential operators of degree $k$ consists of all operators with a kernel $\kappa |d\mu_{sc}|^{1/2}$, such that \begin{itemize}
\item $\kappa$ defined on $X_{sc, b}^2$ is a distribution, conormal of degree $k$ to $\diag_{X^2_{sc, b}}$, smooth up to $\mathrm{sc}$, rapidly decreasing at $\lf$, $\rf$, $\mf$,
\item $|d\mu_{sc}|^{1/2}$ is the lift of $\pi_l^\ast {}^{sc}\Omega^{1/2} (X) \otimes \pi_r^\ast {}^{sc}\Omega^{1/2} (X)$ to $X_b^2$, where $\pi_l$ is the projection from $X^2_b$ to the left space $X$ and $\pi_r$ is the projection from $X^2_b$ to the right space $X$.
\end{itemize} Furthermore, we define the space $\Psi_{sc}^{k, l}(X; {}^{sc}\Omega^{1/2}) = x^l \Psi_{sc}^{k}(X; {}^{sc}\Omega^{1/2})$.\end{definition}

There are two symbol maps defined for $\Psi^{k, 0}_{sc}(X; {}^{sc}\Omega^{1/2})$. One is the usual symbol in the interior of $X$ with respect to ${}^{sc}T^\ast X$, $$\sigma^k_{int}: \Psi_{sc}^{k, 0}(X, {}^{sc}\Omega^{1/2}) \longrightarrow S^k({}^{sc}T^\ast X) / S^{k-1}({}^{sc}T^\ast X),$$ where $S^k(\cdot)$ is the classical symbol space of degree $k$. This symbol map yields a short exact sequence,
\begin{align*}0 \longrightarrow \Psi^{k-1, 0}_{sc} (X; {}^{sc}\Omega^{1/2}) \longrightarrow \Psi^{k, 0}_{sc} (X; {}^{sc}\Omega^{1/2})  \longrightarrow S^k({}^{sc}T^\ast X) / S^{k-1}({}^{sc}T^\ast X) \longrightarrow 0.\end{align*}
In the meantime, the other symbol map, $$\sigma_{\partial}: \Psi^{k, 0}_{sc}(X, {}^{sc}\Omega^{1/2}) \longrightarrow S^k({}^{sc}T^\ast_{\partial X} X),$$ is defined as the restriction of the full symbol to $\partial X = \{x=0\}$. This is well-defined due to the fact $[\mathcal{V}_{sc}(X), \mathcal{V}_{sc}(X)] \subset x \mathcal{V}_{sc}(X)$. It likewise yields a short exact sequence,
\begin{align*}
0 \longrightarrow \Psi^{k, 1}_{sc} (X, {}^{sc}\Omega^{1/2}) \longrightarrow \Psi^{k, 0}_{sc} (X; {}^{sc}\Omega^{1/2})  \longrightarrow S^k({}^{sc}T^\ast_{\partial X}  X) / S^{k-1}({}^{sc}T^\ast_{\partial X} X) \longrightarrow 0.\end{align*}

 If one radially compactifies the fibre of ${}^{sc}T^\ast X$, then the two symbol maps combined give a joint symbol of $P \in \mathrm{Diff}_{sc}^{k}(X; {}^{sc}\Omega^{1/2}) / \mathrm{Diff}_{sc}^{k-1}(X; {}^{sc}\Omega^{1/2})$, $$j_{sc}^k(P) = \bigg(\frac{\sigma^k_{int}(P)}{|\xi|^k} \bigg|_{{}^{sc} S^\ast X}, \sigma_{\partial}(P)\bigg) \in S^k (C_{sc}(X)),$$ where $C_{sc}(X) = {}^{sc} S^\ast X \cup {}^{sc}T^\ast_{\partial X} X$. The joint symbol map $j_{sc}^k$ extends from $\mathrm{Diff}_{sc}^{k}(X; {}^{sc}\Omega^{1/2})$ to $\Psi_{sc}^{k, 0}(X; {}^{sc}\Omega^{1/2})$ multiplicatively, \[j_{sc}^{k_1}(P_1) \cdot j_{sc}^{k_2}(P_2) = j_{sc}^{k_1 + k_2}(P_1 P_2),\] but also admits a short exact sequence, 
\begin{align*}
	0 \longrightarrow \Psi^{k-1, 1}_{sc} (X, {}^{sc}\Omega^{1/2}) \longrightarrow \Psi^{k, 0}_{sc} (X; {}^{sc}\Omega^{1/2})  \longrightarrow S^k(C_{sc}(X)) / S^{k-1}(C_{sc}(X)) \longrightarrow 0.\end{align*}

As in the classical theory, the ellipticity at a point in $C_{sc}(X)$ of a $sc$-pseudodifferential operator $A$ means its joint symbol $j_{sc}^k(A)$ does not vanish, whilst the characteristic variety $\Sigma(A)$ of $A$ is the zero set of $j_{sc}^k(A)$. For example, the interior symbol of the $sc$-differential opeartor $$- (x^2 \partial_x)^2 + (n - 1) x^3 \partial_x + x^2 \Delta_h - 1$$  is obviously elliptic but the boundary symbol vanishes in a subset of ${}^{sc}T^\ast_{\partial X} X$. To understand the resolvent of such an operator,   we will need the notion of the $sc$-wavefront set of a tempered distribution $u$ on $X$, $${}^{sc}\mathrm{WF}(u) = \{q \in C_{sc}(X) \, | \, \mbox{$\chi u \notin \dot{C}^\infty(X)$ for any $\chi \in \Psi_{sc}^{0, 0}(X)$ elliptic at $q$}\}.$$

In the classical propagation of singularities of hyperbolic equations, the wavefront set of the fundamental solution consists of the bicharacteristic flow curves, which comprise a Lagrangian of the cotangent bundle. If we write the smooth sections of ${}^{sc}T^\ast X$ as \[\tau \frac{dx}{x^2} + \mu \frac{dy}{x},\] then ${}^{sc}T^\ast X$ is a symplectic manifold with a non-degenerate $2$-form, \[\omega = d\tau \wedge \frac{dx}{x^2} + d\mu \wedge \frac{dy}{x} - \mu \cdot dy \wedge \frac{dx}{x^2}.\] The symplectic form $\omega$ on ${}^{sc}T^\ast X$ induces  a contact structure on ${}^{sc}T_{\partial X}^\ast X$ with a non-degenerate $1$-form $$\chi = \iota_{x^2 \partial_x} \omega = -d\tau + \mu \cdot dy.$$ The integral curves of the Hamilton vector field determined by the contact form are called the bicharacteriscs as well. For $sc$-pseudodifferential operators, we have a variant of the theorem of microlocalization and propagation of singularities. \begin{proposition} Suppose $A \in \Psi^{k, 0}_{sc}(X, {}^{sc}\Omega^{1/2})$ has a real boundary symbol.
For a tempered distribution $u$ on $X$ but smooth in the interior, \begin{eqnarray*}
\mathrm{WF}_{sc}(Au) \subset \mathrm{WF}_{sc}(u)
\end{eqnarray*} and $\mathrm{WF}_{sc}(u) \setminus \mathrm{WF}_{sc}(Au)$ is a union of maximally extended bicharacteristics of $A$ inside $\Sigma(A)  \setminus \mathrm{WF}_{sc}(Au)$.
\end{proposition}

\subsection{The Legendrian distributions} 

Due to the non-ellipticity of the $sc$-Laplacian in a subset of $T^\ast_{\partial X} X$, it is necessary to formulate a theory of Lagrangian distributions over the boundary, i.e. Legendrian distributions. We review some details from \cite[Section 2]{Hassell-Vasy-JAM} and \cite[Section 2.2]{Hassell-Vasy-AnnFourier}.

In contact geometry, an isotropic submanifold is a submanifold of a contact manifold, on which the contact form vanishes, and a Legendrian is a maximal isotropic submanifold. Suppose $\mathscr{X}$ is a suitable  double space  on which the resolvent kernel is defined, $F$ is an isotropic submanifold of ${}^{sc}T^\ast_{\partial\mathscr{X}} \mathscr{X}$ such that the Hamiltonian vector field is nowhere tangent to $F$, the union of bicharacteristics passing through $F$ forms a Legendrian $L$. The classical theorem of propagation of singularities tells us that the resolvent has to be microlocally supported on the Legendrian $L$, i.e. it is a Legendrian distribution.

\begin{definition}On an $n$-manifold $X$ furnished with local coordinates $ (x, y) $ near $\partial X$, a Legendrian distribution $u \in I^m(X; L; {}^{sc}\Omega^{1/2})$ locally takes the form $$u = x^{m + n/4 - k/2} (2\pi)^{-k/2-n/4} \int_{\mathbb{R}^k} e^{\imath \phi(y, v) / x} a(x, y, v) dv,$$ where \begin{itemize} \item the phase function $\phi$ is non-degenerate, i.e. \[\mbox{$d(\partial_{v_i} \phi)$ are linearly independent at $\{(y, v) : d_{v} \phi=0\}$ for $i = 1, \cdots, k$,}\] \item $\phi$ locally parametrizes the Legendrian $L$, $$L = \{(y, d_{(x, y)} (\phi(y, v)/x)) : d_v \phi(y, v) = 0\},$$ \item the amplitude $a$ is a compactly supported smooth section of the $sc$-half density bundle.\end{itemize}\end{definition}

Suppose $\lambda$ is a set of functions in the $(y, v)$-space such that $( d_v \phi, \lambda)$ form local coordinates on $\{(y, v) : d_{v} \phi=0\}$. The symbol of $u$ is defined by $$\sigma^m_{(x, y), \phi}(u) = a(0, y, v)|_{\{d_v \phi = 0\}} \bigg| \frac{\partial( d_v \phi, \lambda)}{\partial (y, v)} \bigg|^{-1/2} |d\lambda|^{1/2}.$$  The symbol is neither coordinate invariant nor phase function invariant, but changes by the following transition functions
\begin{equation} \label{symbol transition} \left.\begin{array}{l}
\sigma^m_{(x', y'), \phi}(u)  = \sigma^m_{(x, y), \phi}(u) \big(\frac{x'}{x}\big)^{n/4 - m} e^{- \imath ((x'/x)  \mu \cdot \partial_{x'} y - \tau \partial_{x'}  (x'/x) )|_{x = 0} },\\
\sigma^m_{(x, y), \psi}(u) =  \sigma^m_{(x, y), \phi}(u) e^{\imath \pi (\mathrm{sign} d_{vv}^2 \psi - \mathrm{sign} d_{vv}^2 \phi)/4}.
\end{array}\right.\end{equation} These two transition functions define a line bundle $E \otimes M$. The symbol map of degree $m$ is invariantly defined $$\sigma^m : I^m(X, L; {}^{sc}\Omega^{1/2}) \longrightarrow C^\infty(L;  \Omega^{1/2}(L) \otimes S^{[m]}),$$ where $S^{[m]} = |N^\ast \partial X|^{m-n/4} \otimes E \otimes M$ and the bundle $N^\ast \partial X$ is given by differentials of functions vanishing at (each) boundary hypersurface. See \cite[Section 12-13]{Melrose-Zworski} and \cite[Section 4.6]{Hassell-Wunsch} for more information of this symbol space and generalizations to manifolds with corners.

This symbol map gives an exact sequence,
$$0 \longrightarrow I^{m + 1}(X, L; {}^{sc}\Omega^{1/2}) \longrightarrow I^{m}(X, L; {}^{sc}\Omega^{1/2})  \longrightarrow C^\infty(L;  {}^{sc}\Omega^{1/2}(L) \otimes S^{[m]}) \longrightarrow 0, $$ and the symbol calculus of $u \in I^{m}(X; L; {}^{sc}\Omega^{1/2})$ with a $sc$-pseudodifferential operator $P \in \Psi_{sc}(X; {}^{sc}\Omega^{1/2})$ with a real principal symbol vanishing on $L$,
\begin{equation}(-\imath \mathcal{L}_{H_p} - \imath (1/2 + m - n/4) \partial_\tau p + p_{\mathrm{sub}}) \sigma^m(u), \label{eqn : symbol calculus bicharacteristic}\end{equation}
where $H_p$ denotes the Hamilton vector field of the principal symbol of $P$ and $p_{\mathrm{sub}}$ is the boundary subprincipal symbol of $P$. Explicitly, if the full left symbol of a $sc$-pseudodifferential operator $P$ takes the form $$p(y, \tau, \mu) + x p_1(y, \tau, \mu) + O(x^2)$$ in the scattering cotangent bundle with local coordinates $\{(x, y, \tau, \mu)\}$,  its boundary subprincipal symbol is written as the following expansion in $x$, $$p_{\mathrm{sub}} = p_1 + \frac{\imath}{2} \bigg( \frac{\partial^2 p}{\partial y_i \partial \mu_i} - (n - 1)\frac{\partial p}{\partial \tau} + \mu_i \frac{\partial^2 p}{\partial \mu_i \partial \tau} \bigg).$$
For instance, the full left symbol of $P_{sc}$ defined in \eqref{eqn : scattering-operator} reads $$\tau^2 + h^{ij}\mu_i\mu_j + \rho (n- 1 - x e(x, y))\imath \tau - 1,$$ where this $x$ is the defining function in $\mathcal{C}$. If we adopt coordinates $\{(\rho, y, \tau, \mu)\}$, it follows that the boundary subprincipal symbol is equal to \begin{equation}\label{eqn : subprincipal symbol}  - \imath x e(x, y)\tau + \imath(\frac{ \partial h^{ij}}{\partial y_i}\mu_j).\end{equation} When $\mu = 0$, the second term vanishes and then the subprincipal symbol is just equal to the first term $- \imath x e(x, y)\tau$.

Intersecting Legendrian distributions are used to describe the resolvent near the intersection of the propagating Legendrian with other Legendrians which appear in the microlocal structure of the resolvent. The Legendrian pair accommodates the $sc$-wavefront set of the resolvent. Assume $\tilde{L} = (L_0, L_1)$ is a pair of Legendrians intersecting cleanly at $\partial L_1$. Then near a point $q \in \partial L_1$ there is a unique local phase function $\phi(y, v, s) = \phi_0(y, v) + s \phi_1(y, v, s)$ (up to a diffeomorphism) parameterizing the pair \begin{eqnarray*}
L_0 &=& \{(y, d_{(x, y)}(\phi/x)) | s=0, d_v \phi = 0\}\\
L_1 &=& \{(y, d_{(x, y)}(\phi/x)) | s \geq 0, d_s \phi = 0, d_v \phi = 0\}.
\end{eqnarray*} \begin{definition}The space $I^m(X; \tilde{L}; {}^{sc}\Omega^{1/2})$ consists of the sums of elements in $$I^m(X; L_1 \setminus \partial L_1; {}^{sc}\Omega^{1/2}) + I^{m+1/2}(X; L_0; {}^{sc}\Omega^{1/2}),$$ and distributions in a coordinate patch $\{(x, y)\}$ near $\partial X$ with an expression
$$x^{m+n/4-(k+1)/2}\int_0^\infty \,ds \int_{\mathbb{R}^k} e^{\imath \phi(y, v, s)} a(x, y, v, s) dv,$$ where $a$ is a compactly supported smooth $sc$-half density.\end{definition}

 The symbol of $u \in I^m(X; \tilde{L}; {}^{sc}\Omega^{1/2})$ is a smooth section of a bundle over $L_0 \cup L_1$. Let $\rho_1$ be a boundary defining function for $\partial L_1 = L_0 \cap L_1$ within $L_0$, and $\rho_0$ a boundary defining function for $\partial L_1$ within $L_1$.  On $L_0$ but away from $L_1$, the symbol of $u$ lies in the space
$$\rho_1^{-1} C^\infty (\Omega^{1/2}(L_0) \otimes S^{[m + 1/2]}(L_0)) = \rho_1^{-1/2} C^\infty(\Omega_b^{1/2}(L_0 \setminus \partial L_1) \otimes S^{[m + 1/2]}(L_0)), $$
and the symbol of $u$ on $L_1$ takes values in the space
$$ C^\infty(\Omega^{1/2}(L_1) \otimes S^{[m]}(L_1)) = \rho_0^{1/2} C^\infty (\Omega_b^{1/2}(L_1) \otimes S^{[m]}(L_1)).$$ Here $\Omega^{1/2}(L_{\cdot})$ is the half density bundle over the Legendrian $L_{\cdot}$ and  $\Omega_b^{1/2}(L_{\cdot})$ is the $b$-half density bundle near intersections of the Legendrian pairs.  
Then the symbol space $C^\infty(\tilde{L}; {}^b\Omega^{1/2} (\tilde{L}) \otimes S^{[m]}(\tilde{L}))$ for $I^m(X; \tilde{L}; {}^{sc}\Omega^{1/2})$  consists of smooth section pairs $(a, b)$ such that
\begin{equation}\label{eqn : symbol map R}\left.\begin{array}{ll}
   a \in \rho_1^{-1/2} C^\infty(\Omega_b^{1/2}(L_0 \setminus \partial L_1) \otimes S^{[m + 1/2]}(L_0)) & \mbox{on $L_0 \setminus \partial L_1$}\\
  b \in \rho_0^{1/2} C^\infty(\Omega_b^{1/2}(L_1) \otimes S^{[m]}(L_1)) & \mbox{on $L_1$} \\
   \rho_1^{1/2} b = e^{\imath \pi/4} (2\pi)^{1/4} \rho_0^{-1/2} a  & \mbox{at $L_0 \cap L_1$}\end{array}\right..\end{equation}
This symbol map likewise yields an exact sequence,
\begin{multline}\label{eqn : exact sequence for intersecting Legendrian}0 \longrightarrow I^{m + 1}(X; \tilde{L}; {}^{sc}\Omega^{1/2})  \longrightarrow I^{m}(X; \tilde{L}; {}^{sc}\Omega^{1/2})  \longrightarrow C^\infty(\tilde{L}; \Omega_b^{1/2}(\tilde{L}) \otimes S^{[m]}) \longrightarrow 0.\end{multline}

More care is needed to deal with a special type of Legendrian pair, which has a conic singularity at the intersection. We say that $\tilde{L} = (L, L^\sharp)$ is a conic Legendrian pair, if \begin{itemize}
\item $L^\sharp$ is a projectable Legendrian, which means that the projection map ${}^{sc}T^\ast X|_{L^\sharp} \rightarrow \partial X$ is a diffeomorphism,
\item $\tau \neq 0$ on  $L^\sharp$, which implies that $L^\sharp$ can be parametrized by the phase function $1$,
\item $L$ is an open Legendrian such that $\bar{L} \setminus L$ is contained in $L^\sharp$,
\item $\bar{L}$ has at most a conic singularity at $L^\sharp$, which means that $\hat{L} = \beta_\sharp^{-1} \bar{L}$ is transverse to $\beta^{-1}_\sharp \{x=0, \mu=0\}$, where the blowdown map $\beta_\sharp$ defines the blown-up space $[{}^{sc}T^\ast X; \{x=0, \mu=0\}].$
\end{itemize}

A local parametrization of $\tilde{L}$ near $q\in \bar{L} \cap L^\sharp$ is a function $\phi(y, v, s) = 1 + s\psi(y, v, s)$ defined in a neighbourhood of $q' = (y_0, v_0, 0)$ in $\partial X \times \mathbb{R}^k \times [0, \infty)$ satisfying

\begin{itemize}
	\item $1$ parametrizes $L^\sharp$ near $q$, 
	\item $d_v \phi = 0$ at $q'$, 
	\item $q= (y, d_{(x, y)}(\phi/x))(q')$, 
	\item $\phi$ satisfies the non-degeneracy hypothesis
$$\mbox{$ds$, $d\psi$, and $\displaystyle d\bigg(\frac{\partial \psi}{\partial v_i}\bigg)$ are linearly independent at $q'$, $1 \leq i \leq k$},$$
\item $\hat{L}$ has the expression near $q$ in terms of coordinates $\{(x/|\mu|, y, \tau, |\mu|, \hat{\mu})\}$,
$$\hat{L} = \{(0, y, -\phi, sd_y \psi, \widehat{d_y \psi})\, |\, d_s \phi = 0, d_v \psi = 0, s \geq 0\}.$$\end{itemize}

Then we define Legendrians associated with this local parametrization. 
\begin{definition}A Legendrian distribution of order $(m, p)$ associated to $(L, L^\sharp)$ is a $sc$-half density of the form $u = u_0 + (\sum_{i=1}^N u_i)\nu$, where \begin{itemize}
\item $\nu$ is a smooth $sc$-half density,
\item $u_0$ is a Legendrian distribution lying in $$I^m_c(X; L; {}^{sc}\Omega^{1/2}) + I^p(X; L^\sharp; {}^{sc}\Omega^{1/2}),$$ where the subscript $c$ indicates that the wavefront set is contained in the interior of the Legendrian $L$,
\item $u_j$, near the boundary, has an expression
$$\int_0^\infty ds \int e^{i\phi_j(y,v,s)/x} a_j(y, v, x/s, s)\bigg(\frac{x}{s}\bigg)^{m+n/4-(k+1)/2}s^{p+n/4-1}dv,$$
where $\phi_j$ locally parametrizes $(L, L^\sharp)$ and $a_j \in C^\infty(X \times \mathbb{R}^k \times [0,\infty) \times [0,\infty)),$ with compact support in $v$, $x/s$ and $s$.
\end{itemize}
The space of such Legendrian distributions is denoted by $I^{m, p}(X; \tilde{L}; {}^{sc}\Omega^{1/2})$.\end{definition}
The symbol map is defined on $\hat{L}$ and yields an exact sequence
\begin{multline*}
0 \longrightarrow I^{m+1, p}(X; \tilde{L}; {}^{sc}\Omega^{1/2}) \longrightarrow I^m(X; \tilde{L}; {}^{sc}\Omega^{1/2}) \\ \longrightarrow s^{p-m} C^\infty(\hat{L}, \Omega^{1/2}_b(\hat{L}) \otimes S^{[m]}(\hat{L}) ) \longrightarrow 0.\end{multline*}

\section{The fibred $sc$-calculus on $M_{b, 0}$} \label{sec : Legendrian distributions on conic manifolds}

We now come back to the $sc$-regime of the space $M_{b, 0}$ where the resolvent kernel of $P_{sc}$ lives. Similar to the resolvent of the $sc$-Laplacian, the resolvent kernel of $P_{sc}$ consists of $sc$-pseudodifferential operators and Legendrian distributions. To understand this microlocal structure, we apply to $M_{b, 0}$ the $sc$-calculus and Legendrian machinery.

\subsection{The $sc$-pseudodifferential operators on $M_{sc, b, 0}$}

The appropriate stretched double space for the $sc$-pseudodifferential calculus is $M_{sc, b, 0}$ instead of $M_{b, 0}$ as required in the abstract $sc$-calculus.

 The space $\Psi_{sc}^k(M_{sc, b, 0}; {}^{b, sc}\Omega^{1/2})$  is defined to be the collection of ${}^{b, sc}\Omega^{1/2}$-valued distributions in the $sc$-regime of $M_{sc, b, 0}$, which are conormal of degree $k$ to $\diag_{M_{sc, b, 0}}$, smooth up to $\mathrm{sc}$ and $\bfz$, rapidly decreasing at $\lf, \rf, \mf$.

\subsection{The fibred $sc$-bundles on $M_{b, 0}$} 

Dealing with the Legendrian structure near the corners in the $sc$-regime requires working on the finer structure of boundary fibrations. We shall see that these fibrations are trivial at $\bfz$ but non-trivial at $\mf$, $\lf$ and $\rf$ such that the regions near the corners $\mf \cup \lf$ and $\mf \cup \rf$ are of interest. In the meantime, since $\lf$ and $\rf$ are away from the diagonal conormal bundle and the fibration is trivial at $\bfz$, we can blow down the $\mathrm{sc}$ face and return to $M_{b, 0}$. 

Now we are concerned about four boundary faces $\mf, \lf, \rf, \bfz$ of $M_{b, 0}$.  But it suffices to discuss the boundary fibration and Legendrian structures away from $\rf$ as the structure away from $\lf$ is analogous. In addition, we shall denote by $\phi_\bullet$ the fibration at the face $\bullet$, by $Z_\bullet$ the base of $\phi_\bullet$, by $F_\bullet$ the fibre of $\bullet$, i.e. \[\phi_\bullet : \bullet \longrightarrow Z_\bullet.\]

Recall the local expressions of $\mathcal{V}_{b, sc }(M_{b, 0})$ in the $sc$-regime away from $\rf$. 
\begin{itemize}
	\item The region near $\lf \cap \mf \cap \bfz$ admits local coordinates $$(\tilde{\rho}, \rho', y, y', x),$$ where $x, \rho', \tilde{\rho}=\rho/\rho'$ are the boundary defining functions of $\bfz, \mf, \lf$ respectively. The Lie algebra $\mathcal{V}_{b, sc }(M_{b, 0})$ thus lifts to
 \[\mathrm{span} \{\tilde{\rho}^2\rho' \partial_{\tilde{\rho}}, \rho'^2 \partial_{\rho'}, \rho \partial_y, \rho' \partial_{y'}, x \partial_x \} \quad \mbox{near $\lf \cap \mf \cap \bfz$}.\]  

\item The region near the interior of $\lf$ has local coordinates $$(\rho, y, z', x)\quad \mbox{where $z'$ is in the interior of $\mathcal{C}$}.$$ The Lie algebra $\mathcal{V}_{b, sc }(M_{b, 0})$ is then locally expressed by
\[\mathrm{span} \{ \rho^2  \partial_{\rho},  \rho \partial_y,  \partial_{z'}, x \partial_x \} \quad \mbox{near the interior of $\lf$}.\]  
\end{itemize}

We are interested in the restrictions of $\mathcal{V}_{b, sc }(M_{b, 0})$ to the boundary hypersurfaces and the induced fibrations.
\begin{itemize}
	\item The Lie algebra $\mathcal{V}_{b, sc }(M_{b, 0})$ restricted to $\mf$ is simply spanned by $x \partial_x$. This implies the fibration over $\mf$ reads \[\phi_\mf : \mf \longrightarrow Z_\mf, \,  (\tilde{\rho}, y, y', x) \longmapsto (\tilde{\rho}, y, y'),\] and the base $Z_{\mf}$ can be identified with $\mf \cap \bfz$.

\item The Lie algebra $\mathcal{V}_{b, sc }(M_{b, 0})$ restricted to $\lf$ takes the form \[\left\{\begin{array}{ll}\mathrm{span} \{ \rho'^2 \partial_{\rho'},   \rho' \partial_{y'}, x \partial_x \} & \mbox{near $\lf \cap \mf \cap \bfz$}\\
	\mathrm{span} \{  \partial_{z'}, x \partial_x \} & \mbox{near the interior of $\lf$}\end{array} \right..\]
This yields the fibration at $\lf$, which \begin{eqnarray*}
		\phi_\lf : \lf \longrightarrow Z_\lf, \,
\left\{ \begin{array}{ll} (\rho', y, y', x) \longmapsto y & \mbox{near $\lf \cap \mf \cap \bfz$}\\
 (y, z', x) \longmapsto y & \mbox{near the interior of $\lf$}\end{array}	\right..
\end{eqnarray*}
The base $Z_\lf$ of the fibration $\phi_\lf$
 is a compact manifold without boundary,  and the fibres of $\phi_\lf$ intersect $\mf$ transversally. 
 
\item The Lie algebra $\mathcal{V}_{b, sc }(M_{b, 0})$ restricted to $\bfz$ reduces to \[\mathrm{span} \{\tilde{\rho}^2\rho' \partial_{\tilde{\rho}}, \rho'^2 \partial_{\rho'}, \rho \partial_y, \rho' \partial_{y'}\} \quad \mbox{near $\lf \cap \mf \cap \bfz$}.\]  This gives a trivial fibration at $\bfz$,  which does not really matter in the fibred $sc$-calculus.
\end{itemize}

Denote $\Phi = \{\phi_\mf, \phi_\lf, \phi_\rf\}$ and $\tilde{M}_{b, 0} = M_{b, 0} \setminus (\lf \cup \rf)$. The Lie algebra $\mathcal{V}_{s\Phi}(M_{b, 0})$ of the fibred $sc$-vector fields, with respect to the stratification $\Phi$ and total boundary defining function $\rho$ of $\lf \cup \mf$, consists of smooth vector fields $V$ such that
\begin{itemize}
\item $V \in \mathcal{V}_b(M_{b, 0})$,
\item $V\rho \in \rho^2 C^\infty(M_{b, 0})$,
\item $V$ is tangent to the fibres of $\Phi$.
\end{itemize}

The Lie algebra $\mathcal{V}_{s\Phi}(M_{b, 0})$ generates the differential operator algebra $\mathrm{Diff}_{s\Phi}(M_{b, 0})$, which are smooth sections of the fibred $sc$-bundle ${}^{s\Phi}TM_{b, 0}$. Its dual bundle is denoted by ${}^{s\Phi}T^\ast M_{b, 0}$. The wedge product of the basis of ${}^{s\Phi}T^\ast M_{b, 0}$ determines the fibred $sc$-density bundle ${}^{s\Phi}\Omega(M_{b, 0})$. One can write down the expression of ${}^{s\Phi}T^\ast M_{b, 0}$ and ${}^{s\Phi}\Omega(M_{b, 0})$ in local coordinates.
\begin{itemize}
\item Near $\mf \cap \lf \cap \bfz$, we have local coordinates $(\rho', \tilde{\rho}, y, y', x)$. The basis of ${}^{s\Phi}T^\ast M_{b, 0}$ is given by $$\bigg\{\frac{d\rho}{\rho^2}, \frac{d\tilde{\rho}}{\rho}, \frac{dy}{\rho}, \frac{dy'}{\rho'}, \frac{dx}{x}\bigg\} \quad \mbox{or} \quad \bigg\{\frac{d\rho}{\rho^2},  \frac{d\rho'}{\rho'^2},  \frac{dy}{\rho}, \frac{dy'}{\rho'}, \frac{dx}{x}\bigg\}.$$
      For $q \in {}^{s\Phi}T^\ast M_{b, 0}$, we write \begin{multline*}q = \tau \frac{d\rho}{\rho^2} + \eta \frac{d\tilde{\rho}}{\rho} + \mu \cdot \frac{dy}{\rho} + \mu' \cdot \frac{dy'}{\rho'} + \xi \frac{dx}{x}\\ \mbox{or} \quad q = \tilde{\tau} \frac{d\rho}{\rho^2} + \tau' \frac{d\rho'}{\rho'^2} + \mu \cdot \frac{dy}{\rho} + \mu' \cdot \frac{dy'}{\rho'} + \xi \frac{dx}{x},\end{multline*}
               where $\tau = \tilde{\tau} + \tilde{\rho}\tau'$ and $\eta = - \tau'$. The fibred $sc$-density ${}^{s\Phi}\Omega(M_{b, 0})$ locally reads \[  \bigg|\frac{d\rho d\tilde{\rho} dy dy' dx}{\rho^{n+2}\rho'^{n-1}x} \bigg|    \quad \mbox{or} \quad \bigg| \frac{d\rho dy d\rho' dy'dx}{\rho^{n+1} \rho'^{n+1}x}     \bigg|.\]
\item Near the interior of $\lf$, we have local coordinates $\{(\rho, y, z')\}$. The basis of ${}^{s\Phi}T^\ast M_{b, 0}$ is given by
         $$\bigg\{\frac{d\rho}{\rho^2}, \frac{dy}{\rho}, dz', \frac{dx}{x}\bigg\}.$$ For $q \in {}^{s\Phi}T^\ast M_{b, 0}$, we write $$q = \tau \frac{d\rho}{\rho^2} +  \mu \cdot \frac{dy}{\rho} + \zeta' \cdot dz'+ \xi \frac{dx}{x}.$$The fibred $sc$-density ${}^{s\Phi}\Omega(M_{b, 0})$ locally reads \[    \bigg| \frac{d\rho dy dz'  dx}{\rho^{n+1}  x}     \bigg|.\]
\end{itemize}

We are interested in two natural subbundles of ${}^{s\Phi}T_\lf^\ast M_{b, 0}$.
\begin{itemize}
\item The restriction of ${}^{s\Phi}T_\lf^\ast M_{b, 0}$ to each fibre $F_\lf$ of $\lf$ is isomorphic to the $sc$-cotangent bundle of $F_\lf$, denoted by ${}^{b, sc }T^\ast(F_\lf; \lf)$. Recall that the fibre $F_\lf$ is furnished with local coordinates \[\left\{ \begin{array}{ll}(\rho', y', x) & \mbox{near $\lf \cap \mf \cap \bfz$}\\(z', x) & \mbox{near the interior of $\lf$} \end{array} \right..\] The basis of ${}^{b, sc }T^\ast(F_\lf; \lf)$ is given by \[\left\{ \begin{array}{ll}\mathrm{span}\Big\{ \frac{d\rho'}{\rho'^2}, \frac{dy'}{\rho'}, \frac{dx}{x}  \Big\}& \mbox{near $\lf \cap \mf \cap \bfz$}\\\mathrm{span}\Big\{dz', \frac{dx}{x}  \Big\} &\mbox{near the interior of $\lf$}\end{array}\right..\]
\item The quotient bundle, ${}^{s\Phi} T^\ast_\lf M_{b, 0} / {}^{b, sc } T^\ast (F_\lf; \lf)$, is the pull back of a bundle ${}^{s\Phi} N^\ast Z_\lf$ from the base $Z_\lf$ of the fibration $\phi_\lf$, and $${}^{s\Phi}N^\ast Z_\lf = {}^{b, sc} T^\ast_{Z_\lf \times \{0\}} ((Z_\lf)_{y} \times [0, \epsilon)_\rho).$$ The induced map, \begin{multline*}
 \tilde{\phi}_\lf : {}^{s\Phi} T_\lf^\ast M_{b, 0} \rightarrow {}^{s\Phi}N^\ast Z_\lf, \\ \left\{ \begin{array}{ll} (\rho', y, y', x, \tau, \eta, \mu, \mu', \xi) \longmapsto (y, \tau,  \mu)& \mbox{near $\lf \cap \mf \cap \bfz$}\\ (z', y, x, \tau, \mu, \zeta', \xi) \longmapsto (y, \tau,  \mu)&  \mbox{near the interior of $\lf$}\end{array}\right., \end{multline*} has fibres isomorphic to ${}^{b, sc }T^\ast F_\lf$.
\end{itemize}

\subsection{The contact structures on fibred $sc$-bundles}\label{subsec : contact structures}

We next introduce the contact structures and Legendrians in ${}^{s\Phi}T^\ast M_{b, 0}$. The sympletic form of ${}^{s\Phi}T^\ast M_{b, 0}$, denoted by $\omega$, induces the following contact forms on ${}^{s\Phi} T^\ast_\mf M_{b, 0}$, ${}^{s\Phi} N^\ast Z_\lf$ and ${}^{b, sc }T^\ast_{\partial F_\lf}(F_\lf; \lf)$ respectively.   

\begin{itemize}

\item At $\mf$, we use the following local expression  for ${}^{s\Phi}T^\ast M_{b, 0}$  \[\tau \frac{d\rho}{\rho^2} + \eta \frac{d\tilde{\rho}}{\rho} + \mu \cdot \frac{dy}{\rho} + \mu' \cdot \frac{ dy'}{\rho'} + \xi \frac{dx}{x}\]  It follows that the symplectic form $\omega$ takes the form \begin{multline*} d\tau \wedge \frac{d\rho}{\rho^2} + d\eta \wedge \frac{d\tilde{\rho}}{\rho} + d \mu \wedge \frac{dy}{\rho} + d \mu' \wedge \frac{dy'}{\rho'} + d \xi \wedge \frac{dx}{x} \\ - \eta d\tilde{\rho} \wedge \frac{d\rho}{\rho^2} - \mu \cdot dy \wedge \frac{d\rho}{\rho^2} + \mu' \cdot dy' \wedge \frac{d\tilde{\rho}}{\rho} - \frac{\tilde{\rho}\mu'\cdot dy' \wedge d\rho}{\rho^2}.\end{multline*} 
Then $\omega$ 
 induces a contact form $\chi$ on ${}^{s\Phi} T^\ast_\mf M_{b, 0}$, $$\chi = \iota_{\rho^2\partial_\rho} (\omega) = -d\tau + \eta d\tilde{\rho} + \mu \cdot dy + \tilde{\rho} \mu' \cdot dy'.$$

 \item At the corner $\mf \cap \lf$,  the one-form $\chi$ vanishes identically on ${}^{b, sc }T^\ast (F_\lf; \lf)$ but we have a contact form on  ${}^{s\Phi} N^\ast Z_\lf$. In local coordinates, this is $$-d\tau + \mu \cdot dy.$$
\item
At $\partial F_\lf$, we use the following local expression  for ${}^{s\Phi}T^\ast M_{b, 0}$ 
$$\tilde{\tau} \frac{d\rho}{\rho^2} + \tau' \frac{d\rho'}{\rho'^2} + \mu \cdot \frac{dy}{\rho} + \mu' \cdot \frac{dy'}{\rho'} + \xi \frac{dx}{x}.$$
Then ${}^{b, sc }T^\ast_{\partial F_\lf}(F_\lf; \lf)$ admits a contact form of the form $$-d\tau' + \mu' \cdot dy'.$$

\end{itemize}

By incorporating these contact structures, we define Legendrians in ${}^{s\Phi} T^\ast M_{b, 0}$.
\begin{definition}A Legendrian  $L$ of ${}^{s\Phi} T^\ast M_{b, 0}$ is a Legendrian   of ${}^{s\Phi} T^\ast_{\mf} M_{b, 0}$ such that \begin{itemize}
 \item $L$ is transversal to ${}^{s\Phi} T^\ast_{\lf \cap \mf} M_{b, 0}$ and ${}^{s\Phi} T^\ast_{\rf \cap \mf} M_{b, 0}$;
      \item the maps $\tilde{\phi}_\lf$ and $\tilde{\phi}_\rf$ induce fibrations \begin{eqnarray*}\breve{\phi}_\lf : L \cap {}^{s\Phi} T^\ast_{\lf \cap \mf} M_{b, 0} \rightarrow L_1 \\ \breve{\phi}_\rf : L \cap {}^{s\Phi} T^\ast_{\rf \cap \mf} M_{b, 0} \rightarrow L_1',\end{eqnarray*}where $L_1$ and $L_1'$ are Legendrians of ${}^{s\Phi} N^\ast Z_\lf$ and ${}^{s\Phi} N^\ast Z_\rf$;
           \item the fibres of $\breve{\phi}_\lf$ and $\breve{\phi}_\rf$ are Legendrians of ${}^{b, sc} T^\ast_{\partial F_\lf} F_\lf$ and ${}^{b, sc} T^\ast_{\partial F_\rf} F_\rf$.\end{itemize}
\end{definition}

\subsection{The Legendrian distributions on $M_{b, 0}$}\label{subsec : sc-Legendrian distributions}

We shall apply to $M_{b, 0}$ the parametrization of the Legendrians and local expressions of Legendrian distributions. We only focus on the part away from $\rf$, as the part away from $\lf$ is defined analogously.

We will say that a function $\phi(\tilde{\rho}, y, y', v, w)/\rho$ with $v \in \mathbb{R}^k$ and $w \in \mathbb{R}^l$ is a non-degenerate parametrization of $L$ locally near $q \in {}^{s\Phi}T^\ast_{\mf \cap \lf} M_{b, 0} \cap L$ which is given in local coordinates as $(\tilde{\rho}_0, y_0, y_0', x_0, \tau_0, \eta_0, \mu_0, \mu', \xi)$ with $\tilde{\rho}_0 = 0$, if
\begin{itemize}
\item $\phi$ has the form
$$\phi(\tilde{\rho}, y, y', v, w) = \phi_1(y, v) + \tilde{\rho} \phi_2 (\tilde{\rho}, y, y', v, w),$$ such that  $\phi_1, \phi_2$ are smooth in neighbourhoods of $(y_0, v_0) \in \partial \mathcal{C} \times \mathbb{R}^k$ and $q' = (0, y_0, y_0', v_0, w_0) \in ([0, 1)_{\tilde{\rho}} \times (\partial \mathcal{C})^2) \times \mathbb{R}^{k + k'}$  respectively with
$$(0, y_0, y_0', x_0, d_{(\rho, \tilde{\rho}, y, y')}(\phi/\rho)(q'), 0) = q, \quad d_v\phi(q') = 0, \quad d_w \phi_2(q') = 0;$$
\item $\phi$ is non-degenerate in the sense that
$$d_{(y, v)} \frac{\partial \phi_1}{\partial v}, \quad d_{(y', w)} \frac{\partial \phi_2}{\partial w} $$ are independent at $(y_0', v_0)$ and $q'$ respectively;
\item locally near $q$, $L$ is given by
$$L = \{(\tilde{\rho}, y, y', x, d_{(\rho, \tilde{\rho}, y, y')}(\phi/\rho), 0) : d_v \phi = 0,\quad d_w \phi_2 = 0\}.$$
\end{itemize}
Such a phase function does parametrize the three Legendrians affiliated with $L$,
\begin{itemize}
\item $\phi/\rho$ is a non-degenerate phase function for the Legendrian  $L$ in ${}^{s\Phi} T^\ast_{\mathrm{int}(\mf)} M_{b, 0}$;
\item $\phi_1$ parametrizes $L_1$;
\item $\phi_2(0, y, \cdot, v, \cdot)$, for fixed $(y, v) \in \{(y, v) : d_{v}\phi_1 = 0\}$, parametrizes the fibre of $\tilde{\phi}_\lf$.
\end{itemize}
 \begin{remark}  The phase function $\phi$ is indeed independent of $x$ (the boundary defining function of $\bfz$) as well as its dual cotangent variable $\xi$, since the three contact forms introduced above are all independent of $x$ and $\xi$. Consequently, we see $x$ freely varies in $[0, 1)$ and $\xi \equiv 0$ on the Legendrian $L$, i.e. $L$ extends in $x$ smoothly to $\bfz \cap (\mf  \cup \lf)$.  
 \end{remark}

Away from $\lf \cup \rf$, we only have the contact structure in ${}^{b, sc}T^\ast_\mf \tilde{M}_{b, 0}$. The compressed bundle ${}^{b, sc}T^\ast_\mf \tilde{M}_{b, 0}$ is a $sc$-cotangent bundle with respect to $\mf$ which is irrelevant to the fibred $sc$-bundle. The Legendrian distributions are thus defined as in the abstract calculus. Specifically,
\begin{definition}The space $I^m_{sc, c}(\tilde{M}_{b, 0}; L; {}^{b, sc}\Omega^{1/2})$ of Legendrian distribution associated with a Legendrian $L \subset {}^{b, sc}T^\ast_\mf \tilde{M}_{b, 0}$ consists of distributions the form $$u = \rho'^{m + n/2 - k/2} (2\pi)^{-k/2-n/2} \int_{\mathbb{R}^k} e^{\imath \phi(\tilde{\rho}, y, y', v) / \rho'} a(\tilde{\rho}, \rho', y, y', x, v) dv,$$ where the phase function $\phi$ is non-degenerate and locally parametrizes the Legendrian $L$ and  $a \in C^\infty_c(\tilde{M}_{b, 0} \times \mathbb{R}^k; {}^{b, sc}\Omega^{1/2})$.\end{definition}

\begin{remark}Despite $\dim M_{b, 0} = 2n +1$, the contact structures are independent of the extra variable $x$ such that the powers in the oscillatory integral don't take the extra dimension into account, hence the exponent $n/2$ as on $2n$-manifolds.\footnote{In the abstract setting, we live in $n$-manifolds, hence the exponent $n/4$.} But the amplitude $a$ does depend on $x$.\end{remark}
 
The Legendrian distributions associated with the Legendrian $L$ near $\lf$  are defined on the fibred $sc$-bundle ${}^{s\Phi}T^\ast \tilde{M}_{b, 0}$ as follows.
\begin{definition}
A Legendrian distribution $u \in I_{s\Phi}^{m; (r_\lf, \infty, r_{\bfz})}(M_{b, 0}; L; {}^{s\Phi}\Omega^{1/2})$ associated to a Legendrian  $L$ of ${}^{s\Phi}T^\ast M_{b, 0}$ is a distribution in $C^{-\infty} (M_{b, 0}; {}^{s\Phi}\Omega^{1/2})$ taking  the form
$$  u = x^{r_{\bfz}} u_0 + x^{r_{\bfz}} u_1 + \sum_i x^{r_{\bfz}} \omega_i \cdot \nu + \sum_i x^{r_{\bfz}} \omega_i' \cdot \nu ,$$
away from $\rf$, where \begin{itemize}
\item $u_0 \in C^\infty(M_{b, 0}; {}^{b, sc}\Omega^{1/2})$ is a ${}^{b, sc}\Omega^{1/2}$-valued smooth function,
\item $u_1 \in I^m_{sc, c}(\tilde{M}_{b, 0}, L; {}^{b, sc}\Omega^{1/2})$ is a Legendrian distribution associated with $L$ and supported in $\tilde{M}_{b, 0}$,
\item $\nu  \in C^\infty(M_{b, 0}; {}^{s\Phi}\Omega^{1/2})$ is a smooth section of the $s\Phi$-half density bundle,
\item $\omega_i$ is an oscillatory integral supported near the corner $\lf \cap \mf$ and takes the form $$\omega_i(\rho', \tilde{\rho}, y, y', x) = \rho'^{m-(k + k')/2+n/2}\tilde{\rho}^{r_\lf - k/2}   \int_{\mathbb{R}^{k+k'}} e^{\imath \phi_i(\tilde{\rho}, y, y', v, w)/\rho} a_i(\rho', \tilde{\rho}, y, y', x, v, w)\, dvdw,$$ where the amplitude $a_i \in C_c^\infty([0, \epsilon) \times U \times \mathbb{R}^{k+k'})$ with open $U \subset \lf$  and the phase function $\phi_i$ parametrizes the Legendrian $L$ on $U$,
\item  $\omega_i'$ is an oscillatory integral supported in the interior of $\lf$  and takes the form
$$\omega'_i(\rho, y, z', x) = \rho^{r_\lf - k/2 }    \int_{\mathbb{R}^k} e^{\imath \psi_i(y, w) / \rho} a_i(\rho, y, z', x, w) dw,$$
where the amplitude $a_i \in C_c^\infty([0, \epsilon) \times U \times \mathbb{R}^k)$ with open $U \subset \lf$   and the phase function $\psi_i = \phi_i |_{\tilde{\rho}=0}$ parametrizes the Legendrian $L_1$ on $\phi_\lf(U)$.
\end{itemize}The space $I_{s\Phi}^{m; (\infty, r_\rf, r_{\bfz})}(M_{b, 0}; L; {}^{s\Phi}\Omega^{1/2})$ away from $\lf$ is defined in the same manner.%, whilst the full space $I_{s\Phi}^{m; (r_\lf, r_\rf, r_{\bfz})}(M_{b, 0}; L; {}^{s\Phi}\Omega^{1/2})$ is simply their sum, \[ I_{s\Phi}^{m; (r_\lf, \infty, r_{\bfz})}(M_{b, 0}; L; {}^{s\Phi}\Omega^{1/2}) + I_{s\Phi}^{m; (\infty, r_\rf, r_{\bfz})}(M_{b, 0}; L; {}^{s\Phi}\Omega^{1/2}).\]
\end{definition}
  
%\begin{remark}
%	The Legendrian distributions do NOT address the boundary behaviours near $\bfz$, though the symbols (amplitudes) are dependent of $\rho_{\bfz} = x$ (away from $\rf$). This is because $\rho_{\bfz}$ does not really affect the Legendrians, unlike $\lf$ and $\rf$. When we remove the errors at $\bfz$, we shall use the notation $x^\varrho I_{s\Phi}^{m; r_\lf} = \{x^\varrho u : u \in I_{s\Phi}^{m; r_\lf}\}$.
%\end{remark}

A Legendrian pair with conic points, $(L, L^\sharp)$, in ${}^{s\Phi} T^\ast M_{b, 0}$ consists of two Legendrians $L$ and $L^\sharp$ of ${}^{s\Phi} T^\ast M_{b, 0}$ which form an intersecting pair with conic points in ${}^{b, sc} T^\ast_{\tilde{M}_{b, 0}} \tilde{M}_{b, 0}$ such that the fibrations $\breve{\phi}_\lf$ and $\breve{\phi}^\sharp_\lf$ of $L$ and $L^\sharp$ have the same Legendrian   $L_1$ of ${}^{s\Phi}N^\ast Z_\lf$ as base and  the fibres of $\breve{\phi}_\lf$ and $\breve{\phi}^\sharp_\lf$ are intersecting pairs of Legendrians with conic points of ${}^{b, sc} T^\ast_{\partial F} F$.

One can introduce a parametrization for $(L, L^\sharp)$. First of all, we note that $L^\sharp$ is parametrized by a function on $\mf$ of the form $$\phi_0/\rho = (\phi_1(y) + \tilde{\rho} \phi_2(\tilde{\rho}, y, y'))/\rho$$ near the boundary with $\lf$, where $\phi_1 = 1$ on $L^\sharp$ near $L \cap L^\sharp$. Then $L^\sharp_1$ is parametrized by $\phi_1(y)/\rho$.  The fibres of $\breve{\phi}_\lf^\sharp$ are pairs of Legendrians with conic points $(L(y), L^\sharp(y))$. For fixed $y$, $L^\sharp(y)$ is parametrized by $\phi_2(0, y, y')/x'$. Second,  we perform the blow-up $$[{}^{s\Phi}T^\ast_{\mf}(M_{b, 0}); \{\mu' = 0, \tau' = 0, \mu = 0\}],$$ to desingularize the conic singularity of $(L, L^\sharp)$.  Then we can introduce a non-degenerate parametrization $\phi/\rho$ of $(L, L^\sharp)$, near a point $q \in \partial \hat{L} \cap {}^{s\Phi}T^\ast_{\mf \cap \lf} M_{b, 0}$, where $$\phi(\tilde{\rho}, y, y', s, w) = \phi_1(y) + \tilde{\rho} \phi_2(\tilde{\rho}, y, y') + s\tilde{\rho} \psi(\tilde{\rho}, y, y', s, w)$$
is a smooth function defined on a neighbourhood of $q' = (0, y_0, y_0', 0, w_0)$ with
$$q = (0, y_0, y_0', x, - \phi_1, - \frac{\psi}{|d_{y'} \psi|}, \frac{d_{y} \psi}{|d_{y'} \psi|}, 0, \widehat{d_{y'}\psi}(q'), 0), \quad d_s \phi(q') = 0, \quad d_w \psi(q') = 0,$$
and
the lift $\hat{L}$ of $L$ under the blow-up is given by $$\hat{L} = \{(\tilde{\rho}, y, y', x, -\phi_1, - \frac{\psi}{|d_{y'}\psi|}, \frac{d_{y} \psi}{|d_{y'} \psi|}, s |d_{y'} \psi|, \widehat{d_{y'} \psi}, 0) : d_s \phi = 0, d_w \psi = 0, s \geq 0\}.$$
Note that the phase function for $(L, L^\sharp)$ is again independent of $x$ and $\xi$ due to the contact forms being independent of these variables. Therefore, $(L, L^\sharp)$ extends in $x$ smoothly to $\bfz \cap (\mf  \cup \lf)$.

  Away from $\lf \cup \rf$, there are no conic points. The intersecting Legendrian distributions are defined on ${}^{b, sc} T^\ast \tilde{M}_{b, 0}$  as in the abstract settings. Specifically, \begin{definition}The space $I^{m, p_\sharp}_{sc, c}(\tilde{M}_{b, 0}; \tilde{L}; {}^{b, sc}\Omega^{1/2})$ of intersecting Legendrian distributions, associated with $\tilde{L} = (L, L^\sharp)$ in ${}^{b, sc}T^\ast_\mf \tilde{M}_{b, 0}$, consists of  distributions of the form $$u = u_0 + \sum_{i=1}^N u_i\nu_i,$$ where \begin{itemize}
		\item $\nu_i \in C^\infty(\tilde{M}_{b, 0}; {}^{b, sc}\Omega^{1/2})$ is a ${}^{b, sc}\Omega^{1/2}$-valued smooth function on $\tilde{M}_{b, 0}$,
		\item $u_0$ is a Legendrian distribution lying in $$I^m_{sc, c}(\tilde{M}_{b, 0}; L; {}^{b, sc}\Omega^{1/2}) + I^{p_\sharp}_{sc, c}(\tilde{M}_{b, 0}; L^\sharp; {}^{b, sc}\Omega^{1/2}),$$
		\item $u_i$  has a local expression
		$$\int_0^\infty ds \int_{\mathbb{R}^k} e^{i\phi_i(\tilde{\rho}, y, y',v,s)/\rho'} a_i(\tilde{\rho}, y, y', x, v, \rho'/s, s)\bigg(\frac{\rho'}{s}\bigg)^{m+n/2-(k+1)/2}s^{p_\sharp+n/2-1}dv,$$
		where  $a_i \in C^\infty_c(\tilde{M}_{b, 0} \times \mathbb{R}^k \times [0,\infty)\times [0,\infty))$, $\phi_i$ is non-degenerate and locally parametrizes $(L, L^\sharp)$.
	\end{itemize}
\end{definition}
The intersecting Legendrian distributions with conic points are defined on ${}^{s\Phi} T^\ast M_{b, 0}$ as follows. 
\begin{definition} The space $I_{s\Phi}^{m, {p_\sharp}; (r_\lf, \infty, r_{\bfz})}(M_{b, 0}, \tilde{L}; {}^{s\Phi}\Omega^{1/2})$ associated to a Legendrian pair with conic points $\tilde{L} = (L, L^\sharp)$ of ${}^{s\Phi} T^\ast M_{b, 0}$ consists of distributions $u \in C^{-\infty}(M_{b, 0}; {}^{s\Phi}\Omega^{1/2})$ that takes the form $$u =   u_0 +   u_1 +   u_2 + \sum_j   w_j \cdot \nu,$$  away from $\rf$, where \begin{itemize}
\item $u_0 \in I_{s\Phi}^{{p_\sharp}; (r_\lf, \infty, r_{\bfz})}(M_{b, 0}, L^\sharp; {}^{s\Phi}\Omega^{1/2})$ is a Legendrian distribution associated with $L^\sharp$,
\item $u_1 \in I_{s\Phi}^{m; (r_\lf, \infty, r_{\bfz})}(M_{b, 0}, L; {}^{s\Phi}\Omega^{1/2})$ is a Legendrian distribution associated with $L$,
\item $u_2 \in x^{r_{\bfz}} I_{sc, c}^{m, {p_\sharp}}(\tilde{M}_{b, 0}, \tilde{L}; {}^{b, sc} \Omega^{1/2})$ is an intersecting Legendrian distribution associated with $\tilde{L}$ and compactly supported in $\tilde{M}_{b, 0}$,
\item $\nu \in C^\infty(M_{b, 0}; {}^{s\Phi}\Omega^{1/2})$ is a smooth section of the $s\Phi$-half density bundle,
\item $\omega_j$ is an oscillatory integral of the form, \begin{multline*}\omega_j(\rho', \tilde{\rho}, y, y', x) = \int_{[0, \infty) \times \mathbb{R}^{k'}} e^{\imath \phi_j(\tilde{\rho}, y, y', s, w)/\rho} a_j(\rho', \tilde{\rho}, y, y', x, \rho'/s, s, w) \\ \bigg(\frac{\rho'}{s}\bigg)^{m - (k' + 1)/2 + n/2} s^{{p_\sharp}-1+n/2} \tilde{\rho}^{r_\lf} x^{r_{\bfz}} \, ds dw,\end{multline*}
where the amplitude $a_j \in C_c^\infty([0, \epsilon) \times U \times [0, \infty) \times [0, \infty) \times \mathbb{R}^{k'})$ with $U$  open in $\mf$ and the phase function $\phi_j$ parametrizes the Legendrian  pair $(L, L^\sharp)$ on $U$.\end{itemize}
The space $I_{s\Phi}^{m, {p_\sharp}; (\infty, r_\rf, r_{\bfz})}(M_{b, 0}, \tilde{L}; {}^{s\Phi}\Omega^{1/2})$ away from $\lf$ is defined in the same manner.
 \end{definition}

In the end, we define the 'global' Legendrian distributions, to incorporate the boundary behaviours with respect to both $\lf$ and $\rf$, by \begin{eqnarray*}\lefteqn{I_{s\Phi}^{m; (r_\lf, r_\rf, r_{\bfz})}(M_{b, 0}; L; {}^{s\Phi}\Omega^{1/2})}\\&=&I_{s\Phi}^{m; (r_\lf, \infty, r_{\bfz})}(M_{b, 0}; L; {}^{s\Phi}\Omega^{1/2}) + I_{s\Phi}^{m; (\infty, r_\rf,  r_{\bfz})}(M_{b, 0}; L; {}^{s\Phi}\Omega^{1/2})\\ \lefteqn{I_{s\Phi}^{m, {p_\sharp}; (r_\lf, r_\rf,   r_{\bfz})}(M_{b, 0}, \tilde{L}; {}^{s\Phi}\Omega^{1/2})} \\ &=& I_{s\Phi}^{m, {p_\sharp}; (r_\lf, \infty, r_{\bfz})}(M_{b, 0}, \tilde{L}; {}^{s\Phi}\Omega^{1/2}) + I_{s\Phi}^{m, {p_\sharp}; (\infty, r_\rf,  r_{\bfz})}(M_{b, 0}, \tilde{L}; {}^{s\Phi}\Omega^{1/2}).\end{eqnarray*}
The space of Legendrian distributions on $M_{b, 0}$ admits a short exact sequence \begin{multline} \label{eqn : exact seq for fibred intersecting Legendrian distributions}0 \longrightarrow I_{s\Phi}^{m+1, {p_\sharp}; (r_\lf, r_\rf, r_{\bfz})}(M_{b, 0}, \tilde{L}; {}^{s\Phi}\Omega^{1/2}) \longrightarrow I_{s\Phi}^{m, {p_\sharp}; (r_\lf, r_\rf, r_{\bfz})}(M_{b, 0}, \tilde{L}; {}^{s\Phi}\Omega^{1/2})\\ \longrightarrow \rho_\lf^{r_\lf-m} \rho_\rf^{r_\rf-m} s^{{p_\sharp}-m}   \rho_{\bfz}^{r_{\bfz}} C^\infty(\hat{L}, \Omega_b^{1/2}(\hat{L}) \otimes S^{[m]}(\hat{L})) \longrightarrow 0.  
\end{multline}

\begin{center}
	\begin{figure}
		\includegraphics[width= 0.6\textwidth]{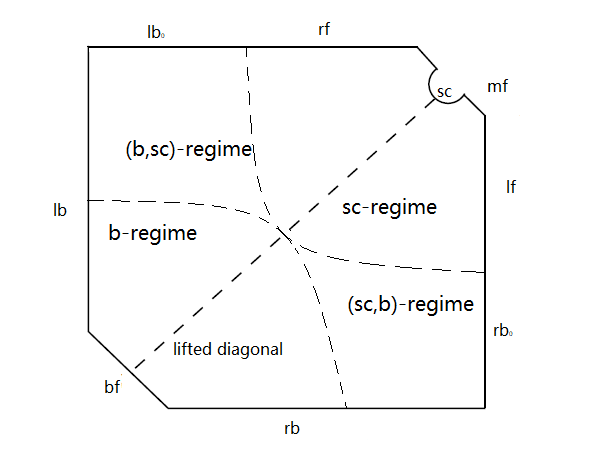}\caption{\label{fig : bf0}The transitional face $\bfz$ and the stretched double metric cones $\mathfrak{C}^2_{sc, b}$}\end{figure} \end{center}
\section{The normal operator at the transitional face} \label{Sec: metric cone resolvent}

If we pull back the differential operators $\tilde{P}_\hbar$, $P_b$ and $P_{sc}$ on $\mathcal{C}$ in \eqref{eqn : blownup-semiclassical-operator}, \eqref{eqn : b-operator} and \eqref{eqn : scattering-operator} to the stretched double spaces $M_{b, 0}$ and $M_{sc, b, 0}$ through  the canonical left (right) projection, the normal operator $\mathcal{N}_{\bfz}(\tilde{P}_\hbar)$ of $\tilde{P}_\hbar$ at $\bfz$ turns out to be $\Delta_{\mathfrak{C}} - (1 + \imath 0)$ where $\Delta_{\mathfrak{C}}$ is the Laplacian on the metric cone $(\mathfrak{C}, \mathfrak{g})$.

In fact, $\Delta_{\mathfrak{C}}$ on $(\mathfrak{C}, \mathfrak{g})$ has the global expression on $\mathfrak{C}$, $$  - \partial_{\mathfrak{x}}^2 - \frac{(n - 1)}{\mathfrak{x}} \partial_{\mathfrak{x}} + \frac{\Delta_{\mathfrak{h}}}{\mathfrak{x}^2}, \quad \mbox{for $0 < \mathfrak{x} < \infty$}.$$  
 By making $\mathfrak{h} = h_0 = h(0, y, dy)$, $r = \mathfrak{x}$ near the small end $\{\mathfrak{x} = 0\}$ and $\rho = 1/\mathfrak{x}$ near the big end $\{\mathfrak{x} = \infty\}$,  we find
\begin{equation}\label{eqn : normal operator at transition}\left.\begin{array}{ll}
		\Delta_{\mathfrak{C}} - (1 + \imath 0) = r^{-1} \mathcal{N}_{\bfz}(P_b) r^{-1} & \mbox{near $\{r=0\}$}\\ 
		\Delta_{\mathfrak{C}} - (1 + \imath 0) = \mathcal{N}_{\bfz}(P_{sc})  & \mbox{near $\{\rho=0\}$}
	\end{array}\right. .\end{equation}

%\subsection{The resolvent of $\Delta_{\mathfrak{C}}$ and the stretched space $\mathcal{C}_{sc, b}^2$}

 Therefore, we want to understand the resolvent kernel on the metric cone $(\mathfrak{C}, \mathfrak{g})$ to see the transitional behaviours of the resolvent kernel $R$ at $\bfz$ between the regimes of $M_{b, 0}$ and $M_{sc, b, 0}$.

First of all, we need to work out the appropriate double space to accommodate the resolvent kernel of $\Delta_{\mathfrak{C}} - (1 + \imath 0)$. The equations \eqref{eqn : normal operator at transition} suggest $\Delta_{\mathfrak{C}}$ is a $b$-differential operator near the small end $\{r=0\}$ and a $sc$-differential operator near the big end $\{r=\infty\}$. Hence the resolvent kernel lives in a stretched version of $\mathfrak{C}^2$ to accommodate these features. For this, we simultaneously blow up the two separate corners $C_{\mathfrak{C}^2, 0} = \{r = 0, r' = 0\}$ and $C_{\mathfrak{C}^2, \infty} = \{\rho = \rho' = 0\}$, $$\mathfrak{C}^2_b = [\mathfrak{C}^2; C_{\mathfrak{C}^2, \infty}, C_{\mathfrak{C}^2, 0}].$$ Subsequently, we perform the blow-up at the boundary at $\infty$ of the lifted diagonal $\diag_{\mathfrak{C}^2_b}$ of $\mathfrak{C}^2_b$,   $$\mathfrak{C}^2_{sc, b} = [\mathfrak{C}^2_b; \partial_\infty \diag_{\mathfrak{C}^2_b}]. $$ Figure \ref{fig : bf0} depicts the space $\mathfrak{C}^2_{sc, b}$. The dashed line is the lifted diagonal $\diag_{\mathfrak{C}^2_{sc, b}}$ of $\mathfrak{C}^2_{sc, b}$. The boundary at $\infty$ of the diagonal $\diag_{\mathfrak{C}_b^2}$ lifts to the face $\mathrm{sc}$.

We assert that the space $\mathfrak{C}^2_{sc, b}$ ($\mathfrak{C}^2_{b}$) is isometric to the transitional face $\bfz$ in $M_{sc, b, 0}$ ($M_{b, 0}$). On the one hand, we may view $\mathfrak{C}^2_{sc, b}$ ($\mathfrak{C}^2_{b}$) as the stretched version of the double space of the front face of the stretched single space $\mathcal{C}_b$. Recall that $\mathcal{C}_b = [\mathcal{C} \times [0 ,1); \partial \mathcal{C} \times \{0\} ]$. The lift of $\partial \mathcal{C} \times \{0\}$, i.e. $\lb_0$ or $\rb_0$, is isometric to $\mathfrak{C}$. This face can be parametrized by \[ \overline{\{(r, y) : 0 < r = x/\hbar \leq \infty\}}\] and the constraint $x=0$ reduces the metric $g = dx^2 + x^2 h(x, y, dy)$ multiplied by the rescaling factor $1/\hbar^2$ to a product type metric $\mathfrak{g} = dr^2 + r^2\mathfrak{h}(y, dy)$. It follows $\lb_0 \times \rb_0  \cong \mathfrak{C}^2$. On the other hand, the face $\bfz$ in $M_{sc, b, 0}$ ($M_{b, 0}$) can obtained by blowing up $\lb_0 \times \rb_0$, where the blow-down map is just given by the left and right stretched projections $\tilde{\mathcal{P}}_L$ and $\tilde{\mathcal{P}}_R$ in Section \ref{sec : blow-up space}.

 The three faces created by the two fold blow-ups are viewed as the three boundaries of $\bfz$ with $\mathrm{bf}, \mf, \mathrm{sc}$ respectively, whilst the four side edges of the square $\mathfrak{C}^2$ are linked to the four boundaries of $\bfz$ with $\lb, \rb, \lb_0, \rb_0$ respectively. The dashed curves separate the $b$-regime, the $sc$-regime, the $(b, sc)$-regime and the $(sc, b)$-regime. We name these boundaries of $\mathfrak{C}^2_{sc, b}$ ($\mathfrak{C}^2_{b}$) $$\{\mathfrak{lb}, \mathfrak{rb}, \mathfrak{lb}_0, \mathfrak{rb}_0, \mathfrak{bf}, \mathfrak{mf}, \mathfrak{sc}\},$$ which are isometric to the corners of $M_{sc, b, 0}$ ($M_{b, 0}$), $$\{\lb \cap \bfz, \rb\cap \bfz, \lb_0 \cap \bfz, \rb_0\cap \bfz, \mathrm{bf}\cap \bfz, \mf\cap \bfz, \mathrm{sc}\cap \bfz\}.$$ In addition, we denote, only in the $sc$-regime, \begin{align*}\mathfrak{lf} &= \mathfrak{rb}_0 \cap \text{$sc$-regime} = \overline{\{\rho/\rho' = 0, 1 > \rho' > 0\}}\\  \mathfrak{rf} &= \mathfrak{lb}_0 \cap \text{$sc$-regime} = \overline{\{\rho'/\rho = 0, 1 > \rho > 0\}}\end{align*} to be consistent with the stretched double spaces  $M_{sc, b, 0}$ and $M_{b, 0}$. 

Through this isomorphism, the compressed cotangent bundles ${}^{b, sc} T^\ast \mathfrak{C}_b^2$ and ${}^{b, sc} T^\ast \mathfrak{C}_{sc, b}^2$ are induced from ${}^{b, sc} T^\ast(\bfz)$. These bundles further generate the compressed densities ${}^{b, sc} \Omega(\mathfrak{C}_b^2)$ and ${}^{b, sc}\Omega (\mathfrak{C}_{sc, b}^2)$.

Moreover, there exist induced contact structures and Legendrian submanifolds $\mathfrak{L}_\pm$ and $\mathfrak{L}^\sharp$ over $\mathfrak{lb}_0 \cup \mathfrak{rb}_0 \cup \mathfrak{mf} \subset \mathfrak{C}^2_{b}$. This is inherited through the isometry from $L_\pm$ and $L^\sharp$ over $\lb_0 \cup \rb_0 \cup \mf \subset M_{b, 0}$, since the contact structure on $M_{b, 0}$ is independent of the boundary defining function of $\bfz$ and extends smoothly to $\bfz$. Then Legendrian distributions are analogously defined over $\mathfrak{lb}_0 \cup \mathfrak{rb}_0 \cup \mathfrak{mf} \subset \mathfrak{C}^2_{b}$ as we present in Section \ref{sec : Legendrian distributions on conic manifolds}.

Guillarmou-Hassell-Sikora \cite[Section 5]{Guillarmou-Hassell-Sikora-TAMS2013} proved that the resolvent kernel $( \Delta_{\mathfrak{C}} - (1 + \imath 0) )^{-1}$ is the sum of full $b$-pseudodifferential operators, $sc$-pseudodifferential operators, and Legendrian distributions. Specifically, the resolvent, in terms of our notations, reads  
\begin{theorem}\label{thm : resovlent on metric cones}
	Let $\chi_b$ and $\chi_{sc}$ be smooth bump functions on $\mathfrak{C}$ in a neighbourhood of  $\{r < 1\}$  and $\{\rho < 1\}$  respectively, and also a partition of identity $\chi_b + \chi_{sc} = 1$. The resolvent kernel $R_{\mathfrak{C}}$ of $(\Delta_{\mathfrak{C}} - (1 + \imath 0) )^{-1}$ is a distribution on $\mathfrak{C}^2_{b}$ consisting of, 
	\begin{itemize}
		\item $R_{\mathfrak{C}, b} = \chi_b R_{\mathfrak{C}} \chi_b$ is a distribution supported in the $b$-regime of $\mathfrak{C}^2_{b}$ such that
		\begin{itemize}
			\item $R_{\mathfrak{C}, b}$ is a ${}^{b, sc}\Omega^{1/2}(\mathfrak{C}_b^2)$-valued full $b$-pseudodifferential operator conormal at order $-2$ to the diagonal $\diag_{\mathfrak{C}^2_{b}}$ and polyhomogeneous conormal to $\mathfrak{lb}$ and $\mathfrak{rb}$,
			\item $R_{\mathfrak{C}, b}$ vanishes to order $2$ at $\mathfrak{bf}$, $n/2$ at $\mathfrak{lb}$ and $\mathfrak{rb}$;
		\end{itemize}
		
		\item $R_{\mathfrak{C}, sc} = \chi_{sc} R_{\mathfrak{C}} \chi_{sc}$ is a distribution supported in the $sc$-regime of $\mathfrak{C}^2_{b}$ such that
		
		\begin{itemize}
			\item $R_{\mathfrak{C}, sc}$ is conormal to order $-2$ to the diagonal $\diag_{\mathfrak{C}^2_{sc, b}}$ (with the diagonal at infinity blown up),
			\item $R_{\mathfrak{C}, sc}$, near the big end of $\diag_{\mathfrak{C}^2_{b}}$ in $\mathfrak{C}^2_b$, is an intersecting Legendrian distribution lying in 
			$$I^{-1/2}(\mathfrak{C}^2_{b}; (N^\ast \diag_{\mathfrak{C}^2_b}, \mathfrak{L}^\circ); {}^{b, sc}\Omega^{1/2}(\mathfrak{C}_b^2)),$$ where $\cdot^\circ$ denotes the interior of $\cdot$, 
			\item $R_{\mathfrak{C}, sc}$, near the intersection of $\mathfrak{mf} \cap \mathfrak{lf} $ and $\mathfrak{mf} \cap \mathfrak{rf}$, is intersecting Legendrian distributions lying in 
			$$I^{-1/2, (n - 2)/2; (n-1)/2, (n-1)/2}(\mathfrak{C}^2_{b}; (\mathfrak{L}, \mathfrak{L}^\sharp); {}^{s\Phi}\Omega^{1/2}),$$ 
			\item $R_{\mathfrak{C}, sc}$ vanishes to order $(n-1)/2$ at both $\mathfrak{lf}$ and $\mathfrak{rf}$. 
		\end{itemize}
	\item $R_{\mathfrak{C}, (b, sc)} = \chi_b R_{\mathfrak{C}} \chi_{sc}$ is a distribution supported in the $(b, sc)$-regime of $\mathfrak{C}^2_{sc, b}$ such that
	\begin{itemize}
	\item $R_{\mathfrak{C}, (b, sc)}$ is given explicitly by \begin{equation} \label{eqn : resolvent lb_0}\frac{\pi \imath r}{2\rho'} \sum_{j=0}^\infty \Pi_{\mathrm{E}_j}(y, y')  \mathrm{J}_{\nu_j}(r) \mathrm{Ha}_{\nu_j}^{(1)}(1/\rho')   \bigg|\frac{drdyd\rho'dy' }{r\rho' }\bigg|^{1/2},\end{equation} where $\Pi_{\mathrm{E}_j}$ is the projection on the $j$th eigenspace $\mathrm{E}_j$ of the operator $\Delta_{\mathfrak{h}} + (n/2-1)^2$, $\nu_j^2$ is the corresponding eigenvalue, and $\mathrm{J}_\nu(\cdot), \mathrm{Ha}_\nu^{(1)}(\cdot)$ are standard Bessel and Hankel functions.

		\item $R_{\mathfrak{C}, (b, sc)}$ vanishes to order $n/2$ at $\mathfrak{lb}$, $(n - 1)/2$ at  $\mathfrak{lb}_0$;		\end{itemize}

		\item $R_{\mathfrak{C}, (sc, b)} = \chi_{sc} R_{\mathfrak{C}} \chi_b$  is a distribution supported in the $(sc, b)$-regime of $\mathfrak{C}^2_{sc, b}$ such that	\begin{itemize}
				\item $R_{\mathfrak{C}, (sc, b)}$ is given explicitly by 
				\begin{equation} \label{eqn : resolvent rb_0}	\frac{\pi \imath r'}{2\rho} \sum_{j=0}^\infty \Pi_{\mathrm{E}_j}(y, y')  \mathrm{J}_{\nu_j}(r') \mathrm{Ha}_{\nu_j}^{(1)}(1/\rho)   \bigg|\frac{d\rho dydr'dy' }{\rho r' }\bigg|^{1/2},\end{equation}

			\item $R_{\mathfrak{C}, (sc, b)}$ vanishes to order $n/2$ at $\mathfrak{rb}$, $(n - 1)/2$ at  $\mathfrak{rb}_0$.	\end{itemize}	
	\end{itemize}
Consequently, the resolvent kernel of $\mathcal{N}_{\bfz}(\tilde{P}_\hbar)$ is described as above, if we replace $R_{\mathfrak{C}}$ by $\mathcal{N}_{\bfz}(\tilde{P}_\hbar)$, $\mathfrak{C}^2_{sc, b}$ by $\bfz \subset M_{sc, b, 0}$, and $\mathfrak{C}^2_{b}$ by $\bfz \subset M_{b, 0}$.
\end{theorem}

%\subsection{The transitional face $\bfz$}

\section{The $b$-regime parametrix}\label{sec : b parametrix}

In this section, we shall construct a parametrix $G_b$ for $P_b$ with a compact remainder $E_b$  in the $b$-regime of $M_{b, 0}$, where the relevant boundary faces include $\lb, \rb, \mathrm{bf}, \bfz$.  More precisely,
\begin{proposition}\label{thm : b-parametrix}There exists a full $b$-pseudodifferential operator $$G_b \in   \Psi_b^{-2, (\hat{\mathcal{E}}_b + 1, \check{\mathcal{E}}_b + 1, \tilde{\mathcal{E}}_b + 2, \mathbb{N}_0)}(M_{b, 0}; {}^b\Omega^{1/2}),$$ in the $b$-regime  such that \begin{equation*}r^{-1} P_b r^{-1}  \circ G_b - \chi_b \Id \chi_b = - E_b \in   \Psi_b^{-\infty, (\infty, \hat{\mathcal{E}}_b + 1, \infty, \infty)}(M_{b, 0}; {}^b\Omega^{1/2}),\end{equation*} where the index sets are defined to be \begin{align}\label{eqn : index set lb}\hat{\mathcal{E}}_b  &= \overline{\bigcup}_{i \in \mathbb{N}_0} (\mathcal{E}_b + i) \\ \label{eqn : index set rb}\check{\mathcal{E}}_b  &= \hat{\mathcal{E}}_b \, \overline{\cup} \, \hat{\mathcal{E}}_b, \\  
\label{eqn : index set bf}\tilde{\mathcal{E}}_b &= \{(0, 0)\} \cup  \Big((\mathbb{N}_0 + 1)\overline{\cup} (\hat{\mathcal{E}}_b  + \check{\mathcal{E}}_b  )\Big) , 	
 \end{align} provided the index set $\mathcal{E}_b = \{(\Im \lambda_j, 0)\}$ as in \eqref{eqn : index set pm}. \end{proposition}

\begin{proof}

The operator $P_b$ is lifted through the stretched left projection $\tilde{\mathcal{P}}_L$ to the $b$-regime in the stretched double $M_{b, 0}$.
The construction away from $\bfz$ is a direct application of the abstract full $b$-calculus. We follow  the $b$-parametrix construction in \cite[Chapter 4-6]{Melrose-APS}. See also Albin's lecture notes \cite[Section 4-5]{Albin} for more details on the index set of the face $\mathrm{bf}$. Moreover, we need additional correction terms near $\bfz$.

\subsection{Construction near the diagonal}

First of all, we invoke the small $b$-calculus to remove the diagonal singularity. Since $P_b \in \mathrm{Diff}_b^2(M_{b, 0}; {}^{b, sc}\Omega^{1/2})$ is an elliptic $b$-differential operator in the $b$-regime, the routine elliptic construction shows that there exists a small $b$-pseudodifferential operator $G_{b, 1}  \in \Psi_{b}^{-2}(M_{b, 0}; {}^{b, sc}\Omega^{1/2})$ in the $b$-regime such that
$$P_b \circ G_{b, 1}  - \chi_b \Id \chi_b   = -  E_{b, 1}  \in \Psi_b^{- \infty}(M_{b, 0}; {}^{b, sc}\Omega^{1/2}).$$

\subsection{Error at $\mathrm{lb}$}\label{subsec: lb error}

We next wish to find a $G_{b, 2} \in \Psi_b^{-\infty, (\hat{\mathcal{E}}_b, \infty, \mathbb{N}_0, \mathbb{N}_0)}(M_{b, 0}; {}^{b, sc}\Omega^{1/2})$ such that $$P_b \circ G_{b, 2} -  E_{b, 1}   = -  E_{b, 2}  \in     \Psi_b^{-\infty, (\infty, \hat{\mathcal{E}}_b, \mathbb{N}_0 + 1, \mathbb{N}_0)} (M_{b, 0}; {}^{b, sc}\Omega^{1/2}).$$  Through this step, we reduce the remainder to a distribution vanishing at $\mathrm{bf}$ and to infinite order at $\lb$. 

Having the remainder vanishing at $\mathrm{bf}$ amounts to solving the indicial equation, $$\mathcal{I}(E_{b, 1}) = \mathcal{I}(P_b) \circ \mathcal{I}(G_{b, 2}).$$ We solve this equation through the inverse Mellin transform of the inverse of the indicial family $\mathcal{I}(P_b, \lambda)$, which is known to be a polyhomogeneous conormal distribution in Proposition \ref{prop : inverse indicial operators}. 
Then part of $G_2$ is taken to be $$G_{b, 2}^{(1)} = -\frac{1}{2\pi} \int_{-\infty}^\infty s^{\imath \lambda} \mathcal{I}^{-1}(P_b, \lambda) \circ \mathcal{M}(E_{b, 1})(\lambda) \,d\lambda,$$ which satisfies $$P_b \circ G_{b, 2}^{(1)} -  E_{b, 1} = -   E_{b, 2}^{(1)}   \in \Psi_b^{-\infty, (\mathcal{E}_b, \mathcal{E}_b, \mathbb{N}_0 + 1, \mathbb{N}_0)}(M_{b, 0}; {}^{b, sc}\Omega^{1/2}).$$

In addition, the error at $\lb$ can be removed completely by \cite[Lemma 5.44]{Melrose-APS}. In fact, this lemma implies that  there exists $G_{b, 2}^{(2)} \in \mathcal{A}_{phg}^{(\hat{\mathcal{E}}_b, \infty, \mathbb{N}_0 + 1, \mathbb{N}_0)}(M_{b, 0}; {}^{b, sc}\Omega^{1/2})$ such that $$P \circ G_{b, 2}^{(2)} -  E_{b, 2}^{(1)}  = -   E_{b, 2} \in \mathcal{A}_{phg}^{(\infty, \mathcal{E}_b, \mathbb{N}_0 + 1, \mathbb{N}_0)}(M_{b, 0}; {}^{b, sc}\Omega^{1/2}).$$ 
If we write $G_{b, 2} = G_{b, 2}^{(1)} + G_{b, 2}^{(2)}$, the parametrix $G_{b, 1} + G_{b, 2}$ obeys that 
$$P_b \circ (G_{b, 1} + G_{b, 2}) - \chi_b \Id \chi_b = -  E_{b, 2}  \in    \Psi_b^{-\infty, (\infty, \mathcal{E}_b, \mathbb{N}_0 + 1, \mathbb{N}_0)}(M_{b, 0}; {}^{b, sc}\Omega^{1/2}).$$ 

\subsection{Error at $\mathrm{bf}$}\label{subsec: bf error}

We  can iterate this construction to remove the error at $\mathrm{bf}$. By the calculus of full $b$-pseudodifferential operators, the $j$th-power of the remainder is $$  E_{b, 2}^j  \in   \Psi_b^{-\infty, (\infty, \mathcal{E}_{b, j}, \mathbb{N}_0 + j, \mathbb{N}_0)}(M_{b, 0}; {}^{b, sc}\Omega^{1/2}),$$ where  the index sets obey
\begin{eqnarray*} \mathcal{E}_{b, 1} &=& \mathcal{E}_{b}, \\
\mathcal{E}_{b, j+1} - 1 &=& \mathcal{E}_{b, j}\overline{\cup}(\mathcal{E}_{b}-1),\\ 
\lim_{j\rightarrow \infty}\mathcal{E}_{b, j} &=& \hat{\mathcal{E}}_b.\end{eqnarray*}
If it is sufficiently close to $\mathrm{bf}$, say $\chi_b = 1$,  we have $$  \chi_b(\Id -    E_{b, 2} )( \Id  + \sum_{j = 1}^\infty   E_{b, 2}^j)\chi_b   - \chi_b \Id \chi_b \in \bigcap_{j > 0}  \Psi_b^{-\infty, (\infty, \hat{\mathcal{E}}_b, \mathbb{N}_0 + j, \mathbb{N}_0)}(M_{b, 0}; {}^{b, sc}\Omega^{1/2}).$$
By the composition of index sets given Proposition \ref{prop : phg mapping}, the improved parametrix $$G_{b, 1, 2} = (G_{b, 1} + G_{b, 2}) (\chi_b \Id \chi_b +  \sum_{j = 1}^\infty E_{b, 2}^j )$$ lies in a slightly larger space $\Psi_b^{-2, (\hat{\mathcal{E}}_b, \check{\mathcal{E}}_b, \tilde{\mathcal{E}}_b, \mathbb{N}_0)}(M_{b, 0}; {}^{b, sc}\Omega^{1/2})$, whilst the new remainder $E_{b, 1, 2} \in \Psi_b^{-\infty, (\infty, \hat{\mathcal{E}}_b, \infty, \mathbb{N}_0)}(M_{b, 0}; {}^{b, sc}\Omega^{1/2})$ vanishes to infinite order at $\mathrm{bf}$.

\subsection{Error at $\bfz$}

We next construct a correction term $G_{b, 3}$ supported around $\bfz$ in the $b$-regime and lying in $$G_{b, 3} \in \Psi_b^{-2, (\hat{\mathcal{E}}_b, \check{\mathcal{E}}_b, \tilde{\mathcal{E}}_b, \mathbb{N}_0)}(M_{b, 0}; {}^{b, sc}\Omega),$$
 which leads to an improved remainder vanishing  at $\bfz$, \begin{equation*}P_b \circ G_{b, 3} -   E_{b, 1, 2} = - E_{b, 3} \in    \Psi_b^{-\infty, (\infty, \hat{\mathcal{E}}_b, \infty, \infty)}(M_{b, 0}; {}^b\Omega^{1/2}).\end{equation*}

This again amounts to solving the normal operator equation $$\mathcal{N}_{\bfz}(P_b) \circ \mathcal{N}_{\bfz}(G_{b, 3}) = \mathcal{N}_{\bfz}(E_{b, 1, 2}).$$ 
Due to $$\mathcal{N}_{\bfz}(P_b) = \mathcal{N}_{\bfz}(\chi_b r\tilde{P}_\hbar r \chi_b) = \chi_br(\Delta_{\mathfrak{C}} - 1 \pm \imath 0)r\chi_b ,$$ this can be done by using Theorem \ref{thm : resovlent on metric cones}.

If we take $$ G_{b, 3}^{(1)} = \big(r^{-1} R_{\mathfrak{C}, b} r'^{-1}\big)   \mathcal{N}_{\bfz}(E_{b, 1, 2}),$$ it results in a better remainder $E_{b, 3}^{(1)}$, which vanishes at $\bfz$ as well as to infinite order at $\lb$ and $\mathrm{bf}$ and lies in $$E_{b, 3}^{(1)} \in    \Psi_b^{-\infty, (\infty, \hat{\mathcal{E}}_b, \infty, \mathbb{N}_0 + 1)}(M_{b, 0}; {}^{b, sc}\Omega^{1/2}).$$ The calculus of full $b$-pseudodifferential operators, Proposition \ref{prop : phg mapping}, yields that $\mathcal{N}_{\bfz}(G_{b, 3}^{(1)})$ is polyhomogeneous conormal to $\lb$, $\rb$, $\mathrm{bf}$ with respect to index sets $(\hat{\mathcal{E}}_b, \check{\mathcal{E}}_b, \tilde{\mathcal{E}}_b)$ as desired. Repeating this procedure $j$ times yields  parametrices $$ G_{b, 3}^{(j)} \in \Psi_b^{-2, (\hat{\mathcal{E}}_b, \check{\mathcal{E}}_b, \tilde{\mathcal{E}}_b, \mathbb{N}_0 + j - 1)}(M_{b, 0}; {}^{b, sc}\Omega^{1/2})$$ and remainders $$E_{b, 3}^{(j)} \in    \Psi_b^{-\infty, (\infty, \hat{\mathcal{E}}_b, \infty, \mathbb{N}_0 + j)}(M_{b, 0}; {}^{b, sc}\Omega^{1/2}).$$ Then the desired parametrix $G_{b, 3}$ and remainder $E_{b, 3}$ are constructed by $$G_{b, 3} = \sum_{j=1}^\infty G_{b, 3}^{(j)}\quad \mbox{and} \quad E_{b, 3} = \lim_{j \rightarrow \infty} E_{b, 3}^{(j)}.$$

In the end, we obtain a parametrix $G_b$ with a remainder $E_b$ such that \begin{eqnarray*}
	&r^{-1} G_br'^{-1} =  G_{b, 1, 2}  +  G_{b, 3},& 	r^{-1}G_br'^{-1} \in \Psi_b^{-2, (\hat{\mathcal{E}}_b, \check{\mathcal{E}}_b, \tilde{\mathcal{E}}_b, \mathbb{N}_0)}(M_{b, 0}; {}^{b, sc}\Omega^{1/2}),  \\ &r^{-1}  E_b r'^{-1} = E_{b, 3},&  r^{-1} E_b r'^{-1} \in  \Psi_b^{-\infty, (\infty, \hat{\mathcal{E}}_b, \infty, \infty)}(M_{b, 0}; {}^{b, sc}\Omega^{1/2}).
\end{eqnarray*} Then $G_b$ and $E_b$ satisfy the conditions in Proposition \ref{thm : b-parametrix}, since $rr'$ lifts to $\rho_{\mathrm{bf}}^2 \rho_{\lb} \rho_{\rb}$ in the $b$-regime of $M_{b, 0}$.
\end{proof}

\section{The $(b, sc)$-regime  parametrix}
\label{sec : b-sc parametrix}

Recall that the $(b, sc)$-regime is a neighbourhood of $\{r < 1\} \cup \{\rho' < 1\}$ in $M_{b, 0}$ and furnished with local coordinates $(r, y, \rho', y', \hbar)$. It has three boundary hypersurfaces $\lb, \lb_0, \bfz$ with boundary defining functions $r, \rho', \hbar/\rho'$ respectively. The compressed density ${}^{b, sc}\Omega(M_{b, 0})$ locally takes the form $$\bigg|\frac{dr dy d\rho dy'd\hbar}{r\rho'^{n+1}\hbar}\bigg|.$$ 

Since $\tilde{P}_\hbar$ is viewed as a left operator on $M_{b, 0}$ via the stretched left projection $\tilde{\mathcal{P}}_L$, we can convert it, as in the $b$-regime, to a $b$-operator $P_b$ in the $(b, sc)$-regime. The $b$-calculus in the $b$-regime extends trivially to the $(b, sc)$-regime, though only the polyhomogeneous conormal distributions in the full calculus are relevant. We shall construct the following parametrix for $\tilde{P}_\hbar$ in this regime.  

\begin{proposition}\label{prop : b-sc parametrix}
	In the $(b, sc)$-regime of $M_{b, 0}$, there is a parametrix $G_{(b, sc)}$ with a kernel $$G_{(b, sc)} \in \rho_{\lb_0}^{(n-1)/2} \sum_{(z, k) \in \hat{\mathcal{E}}_b + 1} \rho_{\lb}^z (\log \rho_{\lb})^k C^\infty(M_{b, 0}; {}^{b, sc}\Omega^{1/2}),$$ provided the index set $\hat{\mathcal{E}}_b$ in \eqref{eqn : index set lb}, such that \[\tilde{P}_\hbar \circ G_{(b, sc)} - \chi_b \Id \chi_{sc} = - E_{(b, sc)}   \in  \rho_{\lb}^\infty \rho_{\lb_0}^\infty \rho_{\bfz}^\infty   C^\infty(M_{b, 0}; {}^{b, sc}\Omega^{1/2}).\]
\end{proposition}

\begin{proof}
	Constructing the parametrix in the $(b, sc)$-regime  involves removing errors at $\lb$, $\lb_0$ and $\bfz$ and obtaining the polyhomogeneity or vanishing order at $\lb$ and $\lb_0$.

	\subsection{Boundary error removal}
	
	The parametrices near $\lb$ and $\lb_0$ are similar to the $b$-regime. Since $r < 1$ in this regime, $P_b = r \tilde{P}_\hbar r$ is still a $b$-differential operator but supported away from the diagonal $\diag_{M_{b, 0}}$ and the face $\mathrm{bf}$.
	
	Since the operator $P_b$ in the $(b, sc)$-regime takes the same form with the $b$-regime. The error at $\lb$ can be removed through the indicial family which is similar to Section \ref{subsec: lb error} and Section \ref{subsec: bf error}. But the difference is that the indicial operator/family at $\mathrm{bf}$ is replaced by that at $\lb_0\cup\bfz=\overline{\{\hbar=0\}}$, since $\lb_0$ plays the same role as $\mathrm{bf}$ in the $b$-regime.  Specifically, we find a parametrix $G_{(b, sc), 1}$ in the $(b, sc)$-regime wihich is polyhomogeneous conormal to $\lb$ with an index set $\hat{\mathcal{E}}_b$ such that \[P_b \circ G_{(b, sc), 1} - \chi_{b} \Id \chi_{sc} =  - E_{(b, sc), 1}   \in \rho_{\lb}^\infty \hbar C^\infty(M_{b, 0}; {}^{b, sc}\Omega).\]
	
	Then the standard power series argument removes the error at $\lb_0 \cup \bfz$ completely. It yields that there a parametrix $G_{(b, sc), 2}$   polyhomogeneous conormal to $\lb_0 \cup \bfz$ with respect to $\tilde{\mathcal{E}}_b$ and to $\lb$ with respect to $\hat{\mathcal{E}}_b$ such that \[P_b \circ G_{(b, sc), 2} - E_{(b, sc), 1} =  - E_{(b, sc), 2}  \in \rho_{\lb}^\infty \hbar^\infty C^\infty(M_{b, 0}; {}^{b, sc}\Omega).\]

	From this, the corresponding parametrix for $\tilde{P}_\hbar$ is  \[ G_{(b, sc)} = r  (G_{(b, sc), 1} + G_{(b, sc), 2})  \rho'^{-1}.\]

	\subsection{Vanishing at $\lb_0$}
	The given polyhomogeneity of the parametrix $G_{b, sc}$ in $\hbar$ is crude, since $\hbar$ is the total defining function for $\lb_0 \cup \bfz$. The index set $\tilde{\mathcal{E}}_b$ in \eqref{eqn : index set bf} for $\hbar$, by definition, does not address any vanishing, though $\inf\hat{\mathcal{E}}_b = n/2-1$ does give an $n/2$-vanishing of $G_{b, sc}$ at $\lb$. It remains to work out the exact vanishing order of $G_{(b, sc)}$ at $\lb_0$. This is determined by the resolvent of the normal operator.\footnote{If $\tilde{P}_\hbar$ is written near $\lb_0$ in the form of an asymptotic expansion in $\rho_{\lb_0}$, the normal operator is just the leading terms of such expansion. On the other hand, the resolvent of $\tilde{P}_\hbar$ also has an asymptotic expansion in $\rho_{\lb_0}$, whilst the resolvent of the normal operator is just the leading terms of the resolvent of $\tilde{P}_\hbar$.}

	Since $x =\hbar=0$ at $\lb_0 \cup \bfz$, we see the normal operator of $P_\hbar$ at $\lb_0 \cup \bfz$ is $$\mathcal{N}_{\lb_0 \cup \bfz}(\tilde{P}_\hbar) = - \partial_{r}^2 - \frac{n-1}{r}\partial_r + \frac{\Delta_{h_0}}{r^2}  -  (1 + \imath 0),$$where $h_0 = h(0, y, dy)$ is a metric on $\mathcal{Y}$ and thus the Laplacian $\Delta_{h_0}$ on $(\mathcal{Y}, h_0)$ is independent of $r$. This is equal to $\Delta_{\mathfrak{C}} - (1 + \imath 0)$ and then converted to a $b$-operator
	$$P_{(b, sc)} = r \mathcal{N}_{\lb_0 \cup \bfz}(\tilde{P}_\hbar) r = - (r\partial_{r})^2 + \Delta_{h_0} + (n/2 -1)^2 - r^2(1\pm \imath 0).$$ 
	
	One can apply separation of variables to compute  $(P_{(b, sc)})^{-1}$. This is done in \cite[(5.5) (5.6)]{Guillarmou-Hassell-Sikora-TAMS2013} for the resolvent of $\Delta_{\mathfrak{C}}- (1 + \imath 0)$ given in Theorem \ref{thm : resovlent on metric cones}. But the resolvent kernel of $P_{(b, sc)}$ lives in a $(2n+1)$-manifold. Using this result, we have 
	$$(\mathcal{N}_{\lb_0 \cup \bfz}(\tilde{P}_\hbar))^{-1} = \frac{\pi \imath r}{2\rho'} \sum_{j=0}^\infty \Pi_{\mathrm{E}_j}(y, y')  \mathrm{J}_{\nu_j}(r) \mathrm{Ha}_{\nu_j}^{(1)}(1/\rho')   \bigg|\frac{drdyd\rho'dy'd\hbar}{r\rho'\hbar}\bigg|^{1/2},$$ where $\Pi_{\mathrm{E}_j}$ is the projection on the $j$th eigenspace $\mathrm{E}_j$ of the operator $\Delta_{h_0} + (n/2-1)^2$, $\nu_j^2$ is the corresponding eigenvalue, and $\mathrm{J}_\nu(\cdot), \mathrm{Ha}_\nu^{(1)}(\cdot)$ are standard Bessel and Hankel functions as in \eqref{eqn : resolvent lb_0} and \eqref{eqn : resolvent rb_0}. 
	
	%It follows that we can find a parametrix $G_{(b, sc), 2}^{(0)}$ in the $(b, sc)$-regime satisfying  $$\mathcal{N}_{\lb_0}(G_{(b, sc), 2}^{(0)}) = \mathcal{N}_{\lb_0}(P_b)^{-1},$$ and making the remainder vanish at $\lb_0$ $$P_b \circ G_{(b, sc), 2}^{(0)} - E_{(b, sc), 1} = - E_{(b, sc), 2}^{(0)} \in   \rho_{\lb_0} \rho_{\lb}^\infty  C^\infty(M_{b, 0}; {}^{b, sc}\Omega^{1/2}).$$ 

	Invoking the asymptotic behaviours of Bessel and Hankel functions, \begin{eqnarray}\label{eqn : Bessel}\mathrm{J}_\nu(r)  &\sim& \frac{1}{\Gamma(\nu + 1)} \bigg(\frac{r}{2}\bigg)^\nu + O(r^{\nu + 1}), \quad \mbox{for $r \rightarrow 0$,}\\ \label{eqn : Hankel}\mathrm{Ha}_\nu^{(1)}(r') &\sim& r'^{-1/2} e^{\imath r' - \imath \nu \pi / 2 + \imath \pi / 4} \sum_{k = 0}^\infty a_k(\nu) r'^{-k}, \quad \mbox{for $r' \rightarrow \infty$,}\end{eqnarray} where the coefficients are given by \begin{eqnarray*}a_0(\nu) = 1, && a_k(\nu) = \frac{\prod_{j = 1}^k (4\nu^2 - (2j - 1)^2)}{k! 8^k} \quad \mbox{for $k \in \mathbb{Z}_+$},\end{eqnarray*} we see that $$  (\mathcal{N}_{\lb_0 \cup \bfz}(\tilde{P}_\hbar))^{-1} \sim r^{n/2}\rho'^{(n-1)/2} e^{\imath /\rho'} \bigg|\frac{drdyd\rho'dy'd\hbar}{r\rho'^{n+1}\hbar}\bigg|^{1/2}.$$ This implies the kernel of $G_{(b, sc)}$ vanishes to order $(n-1)/2$ at $\lb_0$ and to order $n/2$ at $\lb$. \end{proof}

\section{The $sc$-regime parametrix}\label{sec : sc parametrix}

Let us move to the $sc$-regime of $M_{b, 0}$. While using the $sc$-pseudodifferential calculus, we need perform an additional blow-up and stay briefly in $M_{sc, b, 0}$. We construct a parametrix $G_{sc}$ for $P_{sc}$ there. Specifically, \begin{proposition}[$sc$-parametrix]\label{thm : sc-parametrix}There is a distribution $G_{sc}$ in the $sc$-regime such that \begin{align*}P_{sc} \circ G_{sc} - \chi_{sc} \Id \chi_{sc}  = - E_{sc}   \in e^{\imath / \rho'}  \rho_{\rf}^{(n-1)/2 - \varepsilon_r/2} \rho_{\bfz} \rho^\infty_{\mf}\rho^\infty_{\lf} \ C^\infty(M_{b, 0}; {}^{b, sc}\Omega^{1/2}).\end{align*} The kernel of $G_{sc}$, modulo $sc$-pseudodifferential operators in $\Psi^{-2}_{sc} (M_{sc, b, 0}; {}^{b, sc}\Omega^{1/2})$, are Legendrian distributions lying in the class 
\begin{multline*} I_{sc, c}^{-1/2} (M_{b, 0}; (N^\ast \diag_{M_{b, 0}}, L_\pm^\circ); {}^{b, sc}\Omega^{1/2}) \\ +   I_{s\Phi}^{-1/2, (n-2)/2-\mathbf{e}; ((n-1-\mathbf{e})/2, (n-1-\mathbf{e})/2, 0)}(M_{b, 0}; (L_\pm, L^\sharp); {}^{s\Phi}\Omega^{1/2}).\end{multline*} Here we employ the following notations \begin{itemize}\item the diagonals in the $sc$-regime of $M_{b, 0}$ and $M_{sc, b, 0}$ are locally expressed by \begin{align*}\diag_{M_{b, 0}} &=  \{(\tilde{\rho}, \rho', y, y') \in M_{b, 0} : \tilde{\rho} = 1, y = y'\}   \\ \diag_{M_{sc, b, 0}} &= \{ (\mathbf{s}, \rho', \mathbf{y}, y') \in M_{sc, b, 0} : \mathbf{s}  = 0, \mathbf{y} = 0\}    \end{align*} 
	where $ \tilde{\rho} = \rho/\rho'=x'/x$, $\mathbf{s} = (\sigma - 1)/\rho'$ and $\mathbf{y} = (y-y')/\rho'$ are the projective coordinates on $\mf$ and $\mathrm{sc}$,
\item $L= L_\pm$ is the forward/backward geometric Legendrian, $L^\sharp$ is a diffractive Legendrian, and $L^\circ_\pm$ is the interior of $L_\pm$, which are explicitly given in \eqref{eqn : geometric Legendrian} and \eqref{eqn : diffractive Legendrian},
\item $\{ \rho_{\lf}, \rho_{\rf}, \rho_{\mf}\}$ denote the boundary defining functions of $\{\lf, \rf, \mf\}$ respectively, 
%\item in particular, $\rho_\sharp$ is a function on the fibred $sc$-cotangent bundle ${}^{s\Phi} T^\ast M_{b, 0}$, so $\rho_\sharp^\cdot C^\infty M_{b, 0}$ or $\rho_\sharp^\cdot I(M_{b, 0})$ means
\item the error function $\varepsilon = xe(x, y)$, with $e(x,y)$ defined in \eqref{def : e} and $\mathbf{e} = \sup_{(x, y) \in \mathcal{K}}\varepsilon$, produces additional oscillations near $\lf, \rf, L^\sharp$ for the principal symbol of $G_{sc}$, hence the loss $\mathbf{e}$ of the orders of $G_{sc}$,
\item $\bullet_l$ and $\bullet_r$ denote the quantity $\bullet$ in the left and right variables respectively,
\item $\rho'$ on $M_{b, 0}$ is understood as the total boundary defining function of $\mf\cup\rf$. 
\end{itemize}
 \end{proposition}

\begin{proof} 
	
	Due to the scattering nature of $P_{sc}$, the proof exploits the strategy of the resolvent construction for the Laplacian on asymptotically conic manifolds (near the infinity)  by Hassell-Vasy \cite[Section 4]{Hassell-Vasy-AnnFourier}.  However, we must work out  the explicit oscillations near $\lf, \rf, L^\sharp$, caused by $e$, which is the essential feature of non-product conic manifolds as mentioned above.
\subsection{Construction near the diagonal}

We seek $G_{sc, 1}  \in \Psi_{sc}^{-2} (M_{sc, b, 0}; {}^{b, sc}\Omega^{1/2})$ such that
$$P_{sc} \circ G_{sc, 1}  = \chi_{sc} \Id \chi_{sc} - E_{sc, 1} , \quad E_{sc, 1} \in \Psi^{-\infty}_{sc}(M_{sc, b, 0}; {}^{b, sc}\Omega^{1/2}).$$
This means that the remainder term $E_{sc, 1}$ is smooth on $M_{sc, b, 0}$ so that we have solved away the singularity along the diagonal $\diag_{M_{sc, b, 0}}$.

Similar to the $b$-regime, the standard elliptic argument applies here since the interior symbol is non-vanishing. Thus, we first choose a $\tilde{Q}_1 \in \Psi^{-2}_{sc} (M_{sc, b, 0}; {}^{b, sc}\Omega^{1/2})$ with symbol $(\sigma^2(P_{sc}))^{-1}$. Then
$$P_{sc} \circ \tilde{Q}_1 = \chi_{sc}\Id\chi_{sc} - \tilde{E}_1, \quad \tilde{E}_1  \in \Psi_{sc}^{-1}(M_{sc, b, 0}; {}^{b, sc}\Omega^{1/2}).$$ Multiplying $\tilde{Q}_1$ by a finite Neumann series $(\chi_{sc}\Id\chi_{sc} + \tilde{E}_1  + \cdots + \tilde{E}_1^{k - 1})$ yields a remainder in $\Psi_{sc}^{-k}(M_{sc, b, 0}; {}^{b, sc}\Omega^{1/2}).$ Taking an asymptotic limit gives us the desired $G_{sc, 1}$ and $E_{sc, 1}$.

\subsection{Error near the boundary of the diagonal} 

The $\mathrm{sc}$ face will be irrelevant from this stage onwards, and we thus blow down $\mathrm{sc}$ and live in $M_{b, 0}$ hereafter. We then adopt local coordinates $(\tilde{\rho}, \rho', y, y', x)$ in this region of $M_{b, 0}$. The remainder $E_{sc, 1}$ has a singularity at $\partial_{sc} \diag_{M_{b, 0}}$, $$\partial_{sc} \diag_{M_{b, 0}} = \{\rho' = 0, \tilde{\rho}  = 1, y = y'\},$$ and its kernel in a neighbourhood of $\partial_{sc} \diag_{M_{b, 0}}$ within $\mf$ is a smooth function of $\rho$, $(\tilde{\rho} - 1)/\rho$, $(y - y')/\rho$, and $y'$. Therefore it can be expressed, in terms of the Fourier transform, as
$$E_{sc, 1} = \bigg( \int_{\mathbb{R}^n} e^{\imath (  (y-y')\cdot\eta + (\tilde{\rho} - 1)t )/\rho} a_1(\rho', y, x, \eta, t) \,d\eta dt \bigg) \nu, $$ where $\nu \in C^\infty(M_{b, 0}; {}^{b, sc}\Omega^{1/2})$ and $a_1$ is smooth in all variables and Schwartz in $(\eta, t)$. The phase function $(y - y') \cdot \eta + (\tilde{\rho} - 1)t$ parametrizes the Legendrian
$$N^\ast \diag_{M_{b, 0}} = \{y = y', \tilde{\rho} = 1, \mu = - \mu', \tau = - \tau'\}.$$  Therefore, $E_{sc, 1}$ is a Legendrian distribution of order $0$ associated to $N^\ast \diag_{M_{b, 0}}$. To solve away this error inductively, we further write $$E_{sc, 1} = \sum_{k=0}^\infty E_{sc, 1}^{(k)},\quad E_{sc, 1}^{(k)} = \bigg( \int_{\mathbb{R}^n} e^{\imath (  (y-y')\cdot\eta + (\tilde{\rho} - 1)t )/\rho} a_1^{(k)}(\rho', y, x, \eta, t) \,d\eta dt \bigg) \nu, $$ where $E_{sc, 1}^{(k)}$ is a Legendrian distribution of order $k$, and $a_1^{(k)}$ is smooth in all variables and Schwartz in $(\eta, t)$.

Since $\sigma_\partial(P_{sc}) = \tau^2 + h^{ij}\mu_i\mu_j - 1$ vanishes on a codimension one submanifolds of $N^\ast \diag_{M_{b, 0}}$, we consider its (left) Hamilton vector field $V_l$,
$$V_l = 2\tau \tilde{\rho} \frac{\partial}{\partial \tilde{\rho}} + 2 \tau \mu_i \frac{\partial}{\partial \mu_i} - h^{ij}\mu_i\mu_j \frac{\partial}{\partial \tau} + \bigg( \frac{\partial (h^{ij}\mu_i\mu_j)}{\partial \mu_i} \frac{\partial}{\partial y_j} - \frac{\partial (h^{ij}\mu_i\mu_j)}{\partial y_i} \frac{\partial}{\partial \mu_j}\bigg).$$
 The fact that $\tau^2 + h^{ij}\mu_i\mu_j = 1$ on the characteristic variety $\Sigma (P_{sc})$ implies that either $\tau \tilde{\rho} \partial_{\tilde{\rho}}$ or $- h^{ij}\mu_i\mu_j \partial_{\tau}$ in $V_l$ is non-vanishing. Thus $V_l$ is transverse to $N^\ast \diag_{M_{b, 0}}$ at the intersection with $\Sigma(P_{sc})$. We define $L^\circ$ to be the flowout Legendrian emanating from $N^\ast \diag_{M_{b, 0}} \cap \Sigma(P_{sc})$ with respect to $V_l$, and $L_\pm^\circ$ to be the flowout in the positive and negative directions with respect to $V_l$.

We wish to find an intersecting Legendrian distribution $$Q_2^{(0)} \in I_{sc, c}^{-1/2} (\tilde{M}_{b, 0}; (N^\ast \diag_{M_{b, 0}}, L_+^\circ); {}^{b, sc}\Omega^{1/2}),$$ such that \begin{multline}\label{eqn : intersecting Legendrian construction}P_{sc} \circ Q_2^{(0)} - E_{sc, 1}^{(0)} \\ \in I_{sc, c}^1(\tilde{M}_{b, 0}; N^\ast \diag_{M_{b, 0}}; {}^{b, sc}\Omega^{1/2}) +  I_{sc, c}^{3/2} (\tilde{M}_{b, 0}; (N^\ast \diag_{M_{b, 0}}, L_+^\circ); {}^{b, sc}\Omega^{1/2}),\end{multline}
microlocally near $N^\ast \diag_{M_{b, 0}}$. Using the ellipticity of $P_{sc}$ away from $\Sigma(P_{sc})$, we choose $Q_2^{(0)}$ with a symbol on $N^\ast \diag_{M_{b, 0}} \setminus L_+^\circ$ equal to $\sigma_\partial^0(P_{sc})^{-1} \sigma^0(E_{sc, 1}^{(0)} )$. Then the value of the symbol on $L_+^\circ \cap N^\ast \diag_{M_{b, 0}}$ is given by \eqref{eqn : symbol map R}. The symbol of $Q_2^{(0)}$ extends to $L_+$ through the symbol calculus \eqref{eqn : symbol calculus bicharacteristic}, which gives the transport equation
\begin{equation*}\Big(-\imath \mathcal{L}_{H_p}   - \imath (1/2 + m - n/4) \partial_\tau p    + p_{\mathrm{sub}}\Big) \sigma(Q_2^{0}) = 0, \quad \mbox{for $p = \sigma^0_\partial(P_{sc})$}.\label{eqn : transport eqn}\end{equation*}
This ODE  has a unique solution in a sufficiently small neighbourhood of $N^\ast \diag_{M_{b, 0}}$. Namely, the remainder $E_{sc, 1}^{(0)}$ on $L_+^\circ$ has been solved away, and the symbol of $F_1^{(1)} = P_{sc} \circ Q_2^{(0)} - E_{sc, 1}^{(0)}$ on $L_+^\circ$ is of order $-3/2$. Then \eqref{eqn : intersecting Legendrian construction} follows from the exact sequence \eqref{eqn : exact sequence for intersecting Legendrian}.

We can inductively solve away the remainder $$E_{sc, 1}^{(k)} + F_1^{(k)} \in I_{sc, c}^k(\tilde{M}_{b, 0}; N^\ast \diag_{M_{b, 0}}; {}^{b, sc}\Omega^{1/2}) + I_{sc, c}^{k + 1/2}(\tilde{M}_{b, 0}; (N^\ast \diag_{M_{b, 0}}, L_+^\circ); {}^{b, sc}\Omega^{1/2})$$ with an intersecting Legendrian distribution $$Q_2^{(k)} \in I_{sc, c}^{k + 1/2}(\tilde{M}_{b, 0}; (N^\ast \diag_{M_{b, 0}}, L_+^\circ); {}^{b, sc}\Omega^{1/2}),$$ up to a remainder $$F_1^{(k + 1)} \in I_{sc, c}^{k + 1}(\tilde{M}_{b, 0}; N^\ast \diag_{M_{b, 0}}; {}^{b, sc}\Omega^{1/2}) + I_{sc, c}^{k + 3/2} (\tilde{M}_{b, 0}; (N^\ast \diag_{M_{b, 0}}, L_+^\circ); {}^{b, sc}\Omega^{1/2}).$$ Taking the sum  $G_{sc, 2} = \sum_{k=0}^\infty Q_2^{(k)}$ yields a remainder, smooth near $N^\ast \diag_{M_{b, 0}}$. But away from $\partial\diag_{M_{b, 0}}$, we end up with a  remainder, $$E_{sc, 2} = - P_{sc} \circ G_{sc, 2} + E_{sc, 1} \in I_{sc, c}^{1/2}  (\tilde{M}_{b, 0}; L_+^\circ \setminus N^\ast \diag_{M_{b, 0}}; {}^{b, sc}\Omega^{1/2}).$$

\begin{center}\begin{figure}
		\includegraphics[width= \textwidth]{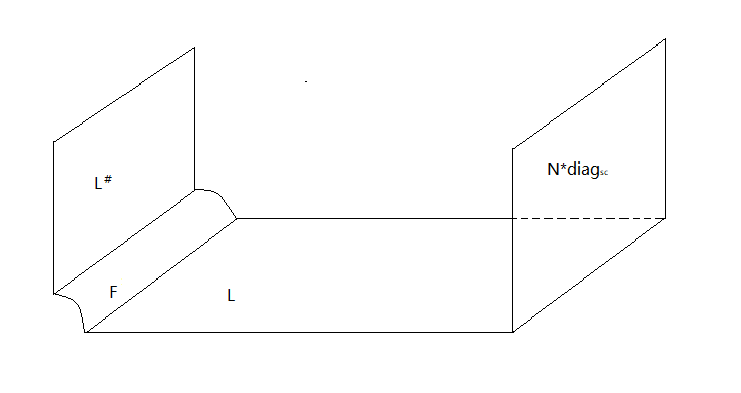}\caption{\label{fig: Legendrian-intersection}The intersecting pair of Legendrians with conic points }\end{figure} \end{center}
\subsection{The bicharacteristic flow}	
\label{subsec : bicharacteristic}

Before solving away the error on $L^\circ_+$, we describe the global structure of $L^\circ_+$ to understand the conic singularities popping up in the closure of $L^\circ_+$.

Recall that $L^\circ$ was defined to be the flowout from $N^\ast\diag_{M_{b, 0}} \cap \Sigma(P_{sc})$ by the Hamilton vector field $V_l$ in left variables $(\rho, y)$. This is the same with the flowout generated by the Hamilton vector field $V_r$ in right variables $(x', y')$.

Melrose-Zworski's calculations (with Hassell-Vasy's modifications \cite[(4.7) (4.9)]{Hassell-Vasy-AnnFourier})  give the following explicit expressions of $L^\circ$   and its closure $L$ under the local coordinates $(\tilde{\rho}, y, y', x, \tau, \tau', \mu, \mu', \xi)$ of ${}^{b, sc}T^\ast_{\mf} M_{b, 0}$.
$$L^\circ =\left\{\begin{array}{l|l}
&  s,s' \in (0, \pi),  \tilde{\rho} = \sin (s)/\sin (s'), \\& \tau = - \cos (s), \tau' =   \cos (s'),
\\ & \mu = \eta(s)\sin(s), \mu'=-\eta(s')\sin(s'),\\ (\tilde{\rho}, y, y', x, \tau, \tau', \mu, \mu', 0)  & h^{ij}\eta_i(s)\eta_j(s) = 1,
\\ & \mbox{$(y(s), \eta(s))$ and $(y'(s'), \eta'(s'))$}\\ &\mbox{ are geodesics in $T^\ast \partial \mathcal{C}$}\\& \mbox{ emanating from $(y_0, \hat{\mu}_0)$ } 
   \end{array}\right\}\cup T_\pm   \cup F_\pm$$
where the vertex sets of the square $[0, \pi]^2$ are defined for $s, s' \in \{0, \pi\}$,\begin{eqnarray*}T_\pm &=& \{(\tilde{\rho}, y_0, y_0, x, \pm 1, \mp 1, 0, 0, 0) \, : \, \sigma \in (0, \infty), y_0 \in \partial \mathcal{C}\},\\ F_\pm &=& \{(\tilde{\rho}, y, y', x, \mp 1, \mp 1, 0, 0, 0) \, : \, \exists\, \mbox{geodesics of length $\pi$ connecting $y, y'$}\}.\end{eqnarray*}
The closure of $L^\circ$ in ${}^{s\Phi}T^\ast M_{b, 0}$  is the set,
\begin{equation}\label{eqn : geometric Legendrian}
L = \left\{\begin{array}{l|l}
 & \mbox{$s,s' \in (0, \pi)$, $\sin^2(s) + \sin^2(s') > 0$, }\\ & \tilde{\rho} = \sin (s)/\sin (s'),
  \tau = - \cos (s), 
  \\&  \tau' =   \cos (s'), 
  \mu = \eta(s)\sin(s), 
  \\ (\tilde{\rho}, y, y', x, \tau, \tau', \mu, \mu', 0) & \mu'=-\eta(s')\sin(s'),  h^{ij}\eta_i(s)\eta_j(s) = 1,
\\  & \mbox{$(y(s), \eta(s))$ and $(y'(s'), \eta'(s'))$} \\&  \mbox{are geodesics in $T^\ast \partial \mathcal{C}$}\\ & \mbox{emanating from $(y_0, \hat{\mu}_0)$ } 
   \end{array}\right\}\cup T_\pm  \cup F_\pm.  \end{equation}The Legendrian $L$ is smooth near $T_\pm$ but has conic singularities around $F_\pm$. See \cite[Section 4.3]{Hassell-Vasy-AnnFourier}

In the coordinates $(y_0, x, \hat{\mu}_0, s, s')$, the vector field $V_l$ is given by $\sin (s') \partial_{s'}$ and $V_r$ is given by $- \sin (s) \partial_s$. The intersection of $L$ and $N^\ast \diag_{M_{b, 0}}$ is given by $\{s = s'\}$. Thus $L_+$ is given by $\{s \leq s'\}$.

There exist Legendrian submanifolds $L^\sharp_\pm$ of ${}^{b, sc} T^\ast_{\mf} \tilde{M}_{b, 0}$, parametrized by \begin{equation}\label{eqn : diffractive Legendrian}L^\sharp_\pm = \{(\tilde{\rho}, y, y', x, \mp 1, \mp 1, 0, 0, 0) \,:\, y, y' \in \partial X, \tilde{\rho} \in [0, \infty] \}.\end{equation} By \cite[Proposition 4.1]{Hassell-Vasy-AnnFourier}, the Legendrian pair $\tilde{L} = (L, L^\sharp_\pm)$, intersecting at $$L \cap L^\sharp_\pm  = F_{\pm},$$ forms a pair of intersecting Legendrians with conic points. See Figure \ref{fig: Legendrian-intersection}.

\begin{remark}
 In contrast to the double scattering $n$-manifolds, $M_{b, 0}$ is a $(2n+1)$-manifold such that ${}^{s\Phi}T^\ast_{\mf}M_{b, 0}$ is a $(4n + 1)$-contact manifold. However, the additional variables $(x, \xi)$ are irrelevant to the contact forms in Section \ref{subsec : contact structures} such that $\xi\equiv0$ on the Legendrians and $x$ complicates nothing but acts as a smooth parameter of Legendrians. It follows that the Legendrian structures described in \cite{Hassell-Vasy-AnnFourier} can be employed verbatim in the present case.
\end{remark}
\subsection{Error at the bicharacteristic}

We next solve away the error  in the interior of $L_+$. This involves solving the transport equation globally on $L$.
What concerns us is the solution of the transport equation near the boundaries of $L$, as the transport equation is smooth and always solvable in the interior of $L$.

To begin with we examine the form of the transport equation at the boundary of $L_+$. It suffices to discuss the situation near $\lf$ as it is analogous for $\rf$.

In a neighbourhood $\{\tilde{\rho} < 2\}$ of $\lf \cap \mf$, we have local coordinates $(y', x, \mu', \tilde{\rho})$ near $T_+$ and $(y', x, \hat{\mu}', \tilde{\rho}, s)$ away from $T_+$.
In either set of coordinates, the Hamilton vector field in left variables $(\tilde{\rho}, y)$ thus takes the form
$$V_l = 2 \tau \tilde{\rho} \partial_{\tilde{\rho}} + Z, \quad \mbox{for some non-vanishing vector fields $Z$}.$$
Recall from \eqref{eqn : subprincipal symbol} that the subprincipal symbol of $P_{sc}$,   is $-\imath \varepsilon\tau + o(|\mu|)$ on $L_+$, where we denote $$\varepsilon = xe(x, y) .$$ By \eqref{eqn : geometric Legendrian}, $\tilde{\rho} = 0$ amounts to $\mu = 0$ in the interior of $L_+$, the subprincipal symbol of $P_{sc}$ is also of the form $-\imath \tau \varepsilon + o(\tilde{\rho})$. 
 Therefore, the transport equation for the symbol of order $-1/2$, with the half density factor removed, takes the form
$$-\imath (\mathcal{L}_{V_l} - n \tau +  \varepsilon_l \tau + o(\tilde{\rho}))  q_3^{(0)}   = 0.$$ Here $\varepsilon_l$ means $\varepsilon$ in left variables.  It reduces to the following ODE,
$$- \imath \tau (\partial_{\tilde{\rho}} + o(1))(\tilde{\rho}^{-n/2 + \varepsilon_l/2} q_3^{(0)}) = 0.$$
This shows that $\tilde{\rho}^{-n/2 + \varepsilon_l/2} q_3^{(0)}$ is smooth near $\lf \cap \mf$.

Near $L^\sharp \cap L_+$, we want to understand the order of $s$, which measures the singularity on $L^\sharp$. Recall that we lift $L_+$ to $\hat{L}_+$ by blowing up $L^\sharp \cap L_+$. By \cite[(4.15)]{Hassell-Vasy-AnnFourier}, the sum of left and right Hamilton vector fields $V_l + V_r$ lifts to $$-2 s\partial_s + o(s),$$ where $W$ is smooth and tangent to the boundary of $\hat{L}_+$. To take advantage of this, we wish to combine the left and right transport equations, since the solution to the left equation also solves the right equation. The right transport operator with respect to $\rho'$ takes the form
$$- \imath (\mathcal{L}_{V_r} - n \tau' +   \varepsilon_r\tau' + o(s)),$$
where $\varepsilon_r$ means $\varepsilon$ in right variables, and the Hamilton vector field in right variables is $$V_r = - 2 \tau' \tilde{\rho} \frac{\partial}{\partial \tilde{\rho}} + 2 \tau' \mu' \frac{\partial}{\partial \mu'} - h' \frac{\partial}{\partial \tau'} + (\frac{\partial h'}{\partial \mu'} \frac{\partial}{\partial y'} - \frac{\partial h'}{\partial y'} \frac{\partial}{\partial \mu'}).$$ In addition, if we change the left variable in $\rho$ to the right variable in $\rho'$, then $q_3^{(0)}$ is replaced by $\tilde{\rho}^{-1/2 - n/2} q_3^{(0)}$ in the right transport equation ($\tilde{\rho}^{-1/2 - n/2}$ is the transition function in \eqref{symbol transition}). It then yields an equation of the form
$$-\imath (\mathcal{L}_{V_r} - n \tau' +  \varepsilon_r \tau' + o(s)) ( \tilde{\rho}^{-1/2 -n/2} q_3^{(0)}  ) = 0.$$
By using the expression of the right Hamilton vector field, we reduce this equation to $$ (V_r + \tau' +   \varepsilon_r \tau' + o(s)) q_3^{(0)}  = 0.$$
Combining it with the left transport equation and using the fact $\tau = \tau' = -1$ on $L^\sharp_+$ gives an equation of the form
$$2 \bigg( s\partial_s - (n - 1)/2 + (\varepsilon_l +   \varepsilon_r)/2   + o(s) \bigg) q_3^{(0)}  = 0.$$  This equation may be rewritten as 
$$(\partial_s  + o(1))(s^{-(n-1)/2 +    (\varepsilon_l +   \varepsilon_r)/2} q_3^{(0)} ) = 0.$$ This, together with the left and right transport equations and  $\tilde{\rho} = \rho_{\lf}$, shows that
\begin{equation}\label{eqn : principal symbol of parametrix at bicharacteristic}q_3^{(0)}  \in \rho_{\lf}^{n/2 - \varepsilon_l /2}\rho_{\rf}^{n/2 -    \varepsilon_r/2} s^{(n-1)/2 -   (\varepsilon_l +  \varepsilon_r)/2} C^\infty({}^{s\Phi}\Omega^{1/2} \otimes S^{[-1/2]}(\hat{L}_+)).\end{equation}

In light of the short exact sequence \eqref{eqn : exact seq for fibred intersecting Legendrian distributions}, the solution $q_3^{(0)} $ to the transport equation, as a principal symbol, determines a Legendrian distribution \begin{equation*}  Q_3^{(0)}\in   I_{s\Phi}^{-1/2, (n-2)/2 - \mathbf{e}; ((n-1-\mathbf{e})/2, (n-1-\mathbf{e})/2, 0)}(M_{b, 0}; \hat{L}_+; {}^{s\Phi}\Omega^{1/2}), \end{equation*} such that
\begin{multline*}P_{sc} \circ Q_3^{(0)} - E_{sc, 2} = - E_2^{(1)}  \\ \in  (\log \rho_{\lf} \log s)^2I_{s\Phi}^{3/2, n/2 - \mathbf{e}; ((n+3-\mathbf{e})/2, (n-1-\mathbf{e})/2, 0)}(M_{b, 0}; (L_+, L^\sharp); {}^{s\Phi}\Omega^{1/2}),\end{multline*}
where the remainder $E_2^{(1)}$ has a principal symbol in the form $$e_2^{(1)} \in \rho_{\lf}^{n/2 - \varepsilon_l/2} \rho_{\rf}^{n/2 -\varepsilon_r/2} s^{ (n-1)/2 -(\varepsilon_l +   \varepsilon_r)/2} (\log \rho_{\lf} \log s)^2 C^\infty({}^{s\Phi}\Omega^{1/2} \otimes S^{[3/2]}(\hat{L}_+))$$
Here we use the notation \footnote{We of course can define the polyhomogeneity for Legendrian distributions with the index family $(\mathcal{E}_{\sharp}, \mathcal{E}_{\lf}, \mathcal{E}_{\rf})$ instead of the vanishing orders $(p_\sharp, r_{\lf}, r_{\rf})$ in Section \ref{subsec : sc-Legendrian distributions}. However, the log terms do not appear in the principal terms given in the main theorem but only in the lower order terms and remainders. So we would introduce Legendrian distributions with log terms only within the proof.} \begin{eqnarray*}  \lefteqn{(\log \rho_{\lf})^{k_\lf} (\log \rho_{\rf})^{k_\lf}   I_{s\Phi}^{m; (r_\lf, r_\rf, r_{\bfz})}(M_{b, 0}; L; {}^{s\Phi}\Omega^{1/2})} \\ &&\quad\quad\quad\quad\quad\quad= 
   \left\{\begin{array}{c} u \end{array} : \begin{array}{l}\sigma^m(u) = (\log \rho_{\lf})^{k_\lf} (\log \rho_{\rf})^{k_\lf}   \sigma^m(\tilde{u}) \\ \tilde{u} \in  I_{s\Phi}^{m; (r_\lf, r_\rf, r_{\bfz})}(M_{b, 0}; L; {}^{s\Phi}\Omega^{1/2}) \end{array}\right\} 
	\\ \lefteqn{(\log \rho_{\lf})^{k_\lf} (\log \rho_{\rf})^{k_\lf} (\log s)^{k_\sharp} I_{s\Phi}^{m, p_\sharp; (r_\lf, r_\rf,   r_{\bfz})}(M_{b, 0}, \tilde{L}; {}^{s\Phi}\Omega^{1/2})} \\ &&\quad\quad\quad\quad\quad\quad= \left\{ \begin{array}{c} u \end{array} : \begin{array}{l}\sigma^m(u) = (\log \rho_{\lf})^{k_\lf} (\log \rho_{\rf})^{k_\lf} (\log s)^{k_\sharp} \sigma^m(\tilde{u}) \\ \tilde{u} \in  I_{s\Phi}^{m, p_\sharp; (r_\lf, r_\rf, r_{\bfz})}(M_{b, 0}; \tilde{L}; {}^{s\Phi}\Omega^{1/2}) \end{array}\right\}. 
\end{eqnarray*}
 Note that $\varepsilon$ is not a constant but a function of $(x, y)$. Hence the quadratic log terms $(\log \rho_{\lf} \log s)^2$ are inevitable when the second $y$-derivative in $P_{sc}$ hits $\rho_{\lf}^{ - \varepsilon_l /2}  s^{  -   (\varepsilon_l +  \varepsilon_r)/2}$ of the symbol of $Q_3^{(0)}$.

 For any $k \in \mathbb{Z}_+$, we want to find a Legendrian distribution \begin{align*}Q_3^{(k)}\in   (\log \rho_{\lf} \log s)^{2k}   I_{s\Phi}^{k - 1/2, n/2-1 - \mathbf{e}; ((n-1- \mathbf{e})/2, (n-1- \mathbf{e})/2, 0)}(M_{b, 0}; \hat{L}_+; {}^{s\Phi}\Omega^{1/2}),\end{align*} which solves the left equation above up to a remainder in \begin{multline*}P_{sc} \circ Q_3^{(k)} - E_2^{(k)} = -E_2^{(k+1)} \\ \in  (\log \rho_{\lf} \log s)^{2(k+1)}   I_{s\Phi}^{k + 3/2, n/2- \mathbf{e}; (n+3- \mathbf{e})/2, ((n-1- \mathbf{e})/2, 0)}(M_{b, 0}; (L_+, L^\sharp); {}^{s\Phi}\Omega^{1/2}).\end{multline*} 

This is done by solving a left transport equation of the form
$$- 2 \imath \tau \Big( \tilde{\rho} \partial_{\tilde{\rho}} + (-n/2 + k) +   \varepsilon_l/2 + o(\tilde{\rho})  \Big) q_3^{(k)} = e_2^{(k)}.$$
The right hand side $e_2^{(k)}$ is the principal symbol of $E_2^{(k)}$. Now we want the remainder $E_2^{(k+1)}$ to be of order $k + 3/2$ at $L_+$ and order $(n + 3 -\mathbf{e})/2$ at  $\lf$, this implies \begin{eqnarray*}e_2^{(k)} &\in& \tilde{\rho}^{n/2 -   \varepsilon_l/2 - k + 1 } (\log \tilde{\rho})^{2k} C^\infty({}^{s\Phi}\Omega^{1/2} \otimes S^{[k+3/2]}(\hat{L}_+))\\ q_3^{(k)} &\in& \tilde{\rho}^{n/2 -  \varepsilon_l/2 - k} (\log \tilde{\rho})^{2k} C^\infty({}^{s\Phi}\Omega^{1/2} \otimes S^{[k-1/2]}(\hat{L}_+)).\end{eqnarray*} By combining this with the right transport equation and using the vector field $V_l + V_r$, we have an ODE of the form
\begin{multline*}- 2 \imath \tau' \bigg( - s \partial_s +    (\varepsilon_l +  \varepsilon_r)/2  + (- \frac{n-1}{2} + k) + o(s)\bigg) q_3^{(k)} \\ \in s^{(n+1 - \varepsilon_l - \varepsilon_r)/2 - k} \rho_{\lf}^{(n -   \varepsilon_l)/2 - k + 1} \rho_{\rf}^{(n-   \varepsilon_r)/2 - k} (\log \rho_{\lf} \log s)^{2(k+1)} C^\infty({}^{s\Phi}\Omega^{1/2} \otimes S^{[k+3/2]}(\hat{L}_+))
.\end{multline*} We remark that there is one more order for $s$ in this ODE than $e_2^{(k)}$ given in the symbol map. It is because $V_l + V_r$ vanishes on $L^\sharp$ such that the sum of the right hand side terms of the two equations lies in $$  (\log \rho_{\lf} \log s)^{2k}    I_{s\Phi}^{k + 1/2, n/2 + 1 - \mathbf{e}; ((n+3 - \mathbf{e})/2, (n-1 - \mathbf{e})/2, 0)} (M_{b, 0}; (L_+, L^\sharp); {}^{s\Phi}\Omega^{1/2}).$$ The solution obeys $$q_3^{(k)} \in s^{(n-1-\varepsilon_l -  \varepsilon_r)/2-k} \rho_{\lf}^{(n- \varepsilon_l)/2-k }  \rho_{\rf}^{(n-   \varepsilon_r)/2 - k}    (\log \rho_{\lf} \log s)^{2k} C^\infty({}^{s\Phi}\Omega^{1/2} \otimes S^{[k-1/2]}(\hat{L}_+))$$ as is required for the parametrix.

 By asymptotically summing these correction terms, we obtain an approximation $G_{sc, 3} = \sum_{k=0}^\infty Q_3^{(k)}$ to the resolvent such that \begin{multline*}P_{sc} \circ G_{sc, 3} - E_{sc, 2} = - E_{sc, 3}    \\ \in   (\log \rho_{\lf} \log s)^{\infty} I^{\infty, n/2 + 1 - \mathbf{e}; ((n+3 - \mathbf{e})/2, (n-1 - \mathbf{e})/2, 0)} (M_{b, 0}, (L_+, L^\sharp); {}^{s\Phi}\Omega^{1/2}).\end{multline*}  This means that the error at the scattering wavefront set at $L_+$ has been solved away.

In addition to removing the error at $L_+$, we can further make the remainder $E_3$ vanish to infinite order at $\mf$. Near $\mf \cap \rf$, we have the boundary defining functions   $\rho_{\mf} = \rho$ and $\rho_{\rf} = \rho'/\rho$, whilst $\rho'$ is the total boundary defining function. By definition, the intersecting Legendrian distribution $G_{sc, 3}$, in this region, takes the form
 $$  e^{\imath / \rho'} \rho'^{(n - 1)/2- \varepsilon_r/2}       \sum_{k \geq 0} \rho'^{k} ( \log  s)^{2k} I_{s\Phi}^{-n/4 - k + \mathbf{e}, n/4 -1/2 - k}(M_{b, 0}; (L_{y'}, L^\sharp); {}^{s\Phi}\Omega^{1/2}).$$ Here $L_{y'}$ is the fibre Legendrian of ${}^{b, sc}T^\ast(F_\rf; \rf)$ with smooth parameters $y'$. In the meantime, the remainder, also near $\mathrm{rf} \cap \mathrm{mf}$, takes the form $$E_{sc, 3} \sim e^{\imath /\rho'} e^{\imath / \rho} \rho^{(n + 3)/2- \varepsilon_l/2} \rho'^{(n-1)/2- \varepsilon_r/2}      C^\infty(M_{b, 0}; {}^{s\Phi}\Omega^{1/2}).$$
 On the other hand, it is obvious that the application of $P_{sc}$ to a function in $(\rho, y)$ will increase the vanishing order in $\rho$. In particular, \begin{align}\label{eqn : standard sc computation}P_{sc} (  e^{\imath / \rho} \rho^{(n - 1)/2 - \varepsilon_l/2}   ) =   e^{\imath / \rho} \rho^{(n + 3)/2 - \varepsilon_l/2} C^\infty(M_{b, 0};  {}^{s\Phi}\Omega^{1/2}).\end{align}   Such an error can be solved away by a term of the form  $$e^{\imath /\rho'} e^{\imath / \rho} \rho^{(n + 1)/2- \varepsilon_l/2} \rho'^{(n-1)/2- \varepsilon_r/2}     C^\infty(M_{b, 0};  {}^{s\Phi}\Omega^{1/2}).$$ We can repeat this procedure to make the remainder vanish in $\rho$ to infinite order. Since $\rho$ is locally the boundary defining function of $\mf$, we have removed the error at $\mf$.
%\[ e^{\imath /\rho'} e^{\imath / \rho} \rho^{(n - 1)/2} \rho'^{(n-1)/2}   \rho_{\mathrm{rf}}^{- \varepsilon_r/2}      \sum_{k \geq 0} \rho^{k}   C^\infty(M_{b, 0};  {}^{s\Phi}\Omega^{1/2}) .\]

\subsection{Error at $L^\sharp$}

This step is to solve away the error at $L^\sharp$, i.e. where $\tilde{\rho} = 0$ and $\tau = 1,$ if we live near $\lf \cap \mf$. This amounts to solving the transport equation
$$  \bigg( -\imath\tilde{\rho} \partial_{\tilde{\rho}} - \imath (1/2 + n/2 + k - (2n)/4) -\imath  \varepsilon_l/2 + o(\tilde{\rho} ) \bigg) q_4^{(k)} = e_3^{(k)}, \quad \mbox{for $k \in \mathbb{Z}_+,$}$$ as the order of $e_3^{(k)}$ at $L^\sharp$ is $n/2 + k - \mathbf{e}$. For $k=1$,  the short exact sequence of Legendrian distributions yields the symbol $e_3^{(1)}$ of $E_3$, which takes the form $$  \tilde{\rho}^{1/2 + \varepsilon_l/2}       \rho^\infty_{\mf} C^\infty(L^\sharp).$$ Then one can find a solution to the transport equation, $$q_4^{(1)} \in   \tilde{\rho}^{-3/2 + \varepsilon_l/2}        \rho^\infty_{\mf} C^\infty(L^\sharp),$$ such that the Legendrian distribution $$Q_4^{(1)} \in   \rho^\infty_{\mf}    I_{s\Phi}^{n/2 + 1 -\mathbf{e}; ((n-1 -\mathbf{e})/2, (n-1-\mathbf{e})/2, 0)} (M_{b, 0}; L^\sharp; {}^{s\Phi}\Omega^{1/2}).$$ This reduces the remainder to $$E_3^{(2)} \in  \rho^\infty_{\mf}   I_{s\Phi}^{n/2 + 2 -\mathbf{e}; ((n+3-\mathbf{e})/2, (n-1-\mathbf{e})/2, 0)}(M_{b, 0}; L^\sharp; {}^{s\Phi}\Omega^{1/2}).$$ Inductively, if the remainder $E_3^{(k)}$ satisfies $$E_3^{(k)} \in    \rho^\infty_{\mf}     I_{s\Phi}^{n/2 + k- \mathbf{e}; ((n+3 - \mathbf{e})/2, (n-1- \mathbf{e})/2, 0)}(M_{b, 0}; L^\sharp; {}^{s\Phi}\Omega^{1/2}),$$ then one can find a solution to the transport equation $$q_4^{(k)} \in \tilde{\rho}^{-1/2 - k + \varepsilon_l/2} \rho^\infty_{\mf}  C^\infty (M_{b, 0}; L^\sharp; {}^{s\Phi}\Omega^{1/2}).$$ This gives a Legendrian distribution $$Q_4^{(k)} \in   \rho^\infty_{\mf}  I_{s\Phi}^{n/2 + k - \mathbf{e}; ((n-1- \mathbf{e})/2, (n-1- \mathbf{e})/2, 0)} (M_{b, 0}; L^\sharp; {}^{s\Phi}\Omega^{1/2}),$$ which reduces the remainder to \[ E_3^{(k+1)} \in    \rho^\infty_{\mf}     I_{s\Phi}^{n/2 + k + 1- \mathbf{e}; ((n + 3- \mathbf{e})/2, (n - 1- \mathbf{e})/2, 0)} (M_{b, 0}; L^\sharp; {}^{s\Phi}\Omega^{1/2}).\] 

Taking an asymptotic sum yields a parametrix $G_{sc, 4} = \sum Q_4^{(k)}$ and a remainder $E_{sc, 4}$ lying in
$$ e^{\imath / \rho'} \rho_{\lf}^{(n-1)/2 - \varepsilon_l/2} \rho_{\rf}^{(n-1)/2 - \varepsilon_r/2}    \rho^\infty_{\mf}   C^\infty(M_{b, 0}; {}^{b, sc}\Omega^{1/2}).$$  

In addition, the error at $\lf$ can be removed by an analogous argument using \eqref{eqn : standard sc computation} as for $\mf$. Then the remainder is reduced to $$E_{sc, 4} \in e^{\imath / \rho'}  \rho_{\rf}^{(n-1)/2 - \varepsilon_r/2}   \rho_{\lf}^{\infty} \rho^\infty_{\mf}   C^\infty(M_{b, 0}; {}^{b, sc}\Omega^{1/2}).$$  

\subsection{Error at $\bfz$}

 The error at $\mathrm{bf}_0$ can be solved away by a standard normal operator argument. In fact, the normal operator $\mathcal{N}_{\mathrm{bf}_0}(P_{sc})$ is just $\Delta_{\mathfrak{C}} - (1 + \imath 0)$. The parametrices we have constructed actually include the region near $\mathrm{bf}_0$ but it might not agree with $R_{\mathfrak{C}, sc}$ at $\mathrm{bf}_0 \cap (\mathrm{mf} \cup \mathrm{lf} \cup \mathrm{rf})$ due to the non-uniqueness of diffractive waves. Therefore, we further take $$\mathcal{N}_{\mathrm{bf}_0}(G_{sc, 5}) =  \big(R_{\mathfrak{C}, sc} - \mathcal{N}_{\mathrm{bf}_0}(G_{sc, 1} + G_{sc, 2} + G_{sc, 3} + G_{sc, 4})\big) \circ \mathcal{N}_{\mathrm{bf}_0}(E_{sc, 4}),$$ which is a Legendrian distribution associated only with the diffractive Legendrian $L^\sharp$.
The restriction of the remainder $E_{sc, 4}$ to $\mathrm{bf}_0$ is of the form, $$\mathcal{N}_{\mathrm{bf}_0}(E_{sc, 4}) \in e^{\imath / \rho'}  \rho_{\mathrm{rf}}^{(n-1)/2}  \rho_{\mathrm{lf}}^\infty \rho_{\mathrm{mf}}^\infty    C^\infty(\mathrm{bf}_0, {}^{sc}\Omega^{1/2}).$$
Then $\mathcal{N}_{\mathrm{bf}_0}(G_{sc, 5})$ solves the normal operator equation with a remainder $E_{sc}$ vanishing at $\mathrm{bf}_0$, whilst it lies in the space 
$$       
I_{s\Phi}^{(n-2)/2; ((n-1)/2, (n-1)/2, 0)}(M_{b, 0};   L^\sharp; {}^{s\Phi}\Omega^{1/2}),$$ which is even better than existing parametrices due to $\varepsilon|_{\bfz} = 0$.

Finally, the desired parametrix is just \[G_{sc} = G_{sc, 1} + G_{sc, 2} + G_{sc, 3} + G_{sc, 4}  + G_{sc, 5}\] with a remainder of the form $$E_{sc} \in e^{\imath / \rho'}  \rho_{\rf}^{(n-1)/2 - \varepsilon_r/2}  \rho_{\bfz}  \rho_{\lf}^{\infty} \rho^\infty_{\mf}   C^\infty(M_{b, 0}; {}^{b, sc}\Omega^{1/2}).$$  
\end{proof}

\section{The $(sc, b)$-regime parametrix}\label{sec : sc-b parametrix}

Recall that the $(sc, b)$-regime is a neighbourhood of $\{r' < 1\} \cup \{\rho < 1\}$ in $M_{b, 0}$ and furnished with local coordinates $(\rho, y, r', y', \hbar)$. It has three boundary hypersurfaces $\rb, \rb_0, \bfz$ with boundary defining functions $r', \rho, \hbar/\rho$ respectively, where the face $\rb_0$ is renamed $\lf$ in the $sc$-regime since it is defined by the left boundary defining function $\rho$. The compressed density ${}^{b, sc}\Omega(M_{b, 0})$ locally takes the form $$\bigg|\frac{d\rho dy dr' dy'd\hbar}{\rho^{n+1}r'\hbar}\bigg|.$$ 

The operator $\tilde{P}_\hbar$, as a left operator on $M_{b, 0}$, is also a $sc$-operator in the $(sc, b)$-regime, like $P_{sc}$ in the $sc$-regime. The face $\rb_0$ in this regime is viewed as the interior of $\lf$ in the $sc$-regime. The $sc$-calculus for Legendrian distributions extends to the whole $\rb_0$ indeed. We can construct the following parametrix for $\tilde{P}_\hbar$ in this regime.  

\begin{proposition}\label{prop : sc-b parametrix}
	In the $(sc, b)$-regime of $M_{b, 0}$, there is a parametrix $G_{(sc, b)}$ lying in     \[\rho_{\rb}^{n/2} I_{s\Phi}^{n/2 + 1 - \mathbf{e}; ((n-1-\mathbf{e})/2, \infty, 0)}(M_{b, 0}; L^\sharp|_{\rb_0}; {}^{s\Phi}\Omega^{1/2}).\]  such that \[\tilde{P}_\hbar \circ G_{(sc, b)} - \chi_b \Id \chi_{sc} = - E_{(sc, b)}     \in    \rho_{\bfz} \rho_{\rb}^{n/2}  \rho_{\rb_0}^{\infty}      C^\infty(M_{b, 0}; {}^{b, sc}\Omega^{1/2}).\]
\end{proposition}

\begin{proof}
We shall first remove completely the error at $\rb_0$ and then employ the normal operator at $\bfz$ to obtain the vanishing order of the paramerix at $\rb$.

	\subsection{Boundary error removal}
Noting that	$\lf\subset\rb_0 = \overline{\{\rho=0\}}$, one can extend the construction of $G_{sc, 4}$ and $G_{sc, 5}$  to  the $(sc, b)$-regime near $\rb_0$. In fact, the face $\rb_0$ in the $(sc, b)$-regime can be regarded as the interior of $\lf$ (away from $\mf$) in the $sc$-regime. The only change of the argument would be the use of coordinates near the interior $\lf$ away from $\mf$  rather than those near the corner $\mf \cap \lf$. This yields a parametrix $G_{(sc, b)}$ which is of the form \[I_{s\Phi}^{n/2 + 1 - \mathbf{e}; ((n-1-\mathbf{e})/2, \infty, 0)}(M_{b, 0}; L^\sharp|_{\rb_0}; {}^{s\Phi}\Omega^{1/2}).\]
Then we remove the error at $\rb_0$ by using the computation \eqref{eqn : standard sc computation} and have the remainder vanishing at $\rb_0$ and $\bfz$ in exactly the same way as for $\lf$ and $\bfz$ in the $sc$-regime.

	\subsection{Vanishing order at $\rb$}	
Since $x$ is the boundary defining function of $\bfz$, the normal operator  $\mathcal{N}_{\bfz}(\tilde{P}_{\hbar})$   is   equal to \[- (\rho^2 \partial_{\rho})^2 + (n-1)\rho^3 \partial_{\rho} + \rho^2 \Delta_{h_0} - (1 + \imath 0).\]

This is equal to $\Delta_{\mathfrak{C}} - (1 + \imath 0)$ near the infinity of the metric cone $\mathfrak{C}$. The resolvent of $\mathcal{N}_{\bfz}(\tilde{P}_{\hbar})$ can be calculated readily by separation of variables. Given in \cite[(5.5) (5.6)]{Guillarmou-Hassell-Sikora-TAMS2013}, it reads \[	(\mathcal{N}_{\bfz}(\tilde{P}_\hbar))^{-1} = \frac{\pi \imath r'}{2\rho} \sum_{j=0}^\infty \Pi_{\mathrm{E}_j}(y, y')  \mathrm{J}_{\nu_j}(r') \mathrm{Ha}_{\nu_j}^{(1)}(1/\rho)   \bigg|\frac{d\rho dydr'dy'd\hbar}{\rho r'\hbar}\bigg|^{1/2}.\] 
 Using again the asymptotic behaviours \eqref{eqn : Bessel} and \eqref{eqn : Hankel} of Bessel and Hankel functions, we obtain that $$  (\mathcal{N}_{\bfz}(\tilde{P}_\hbar))^{-1} \sim \rho^{(n-1)/2}r'^{n/2} e^{\imath /\rho} \bigg|\frac{d\rho dydr'dy'd\hbar}{\rho^{n+1}r'\hbar}\bigg|^{1/2}.$$ This implies $G_{(b, sc)}$ is taken to vanish to order $n/2$ at $\rb$, to match the parametrix at $\bfz$ as well as that in the $b$-regime.\footnote{Be aware that one cannot solve away the error at $\rb$!} 
\end{proof}

%An identical argument near $\rb_0$, i.e. $\{r' < 1 < r\}$,  gives a parametrix $$G_{b, 4} \in \rho_{\rb_0}^{(n-1)/2}\rho_{\lb_0}^\infty\rho_{\lb}^\infty \rho_{\mathrm{bf}}^\infty\Psi_b^{-\infty, (\infty, \check{\mathcal{E}_b})}(M_{sc, b, 0}; {}^b\Omega^{1/2})$$ which makes the remainder also vanish to infinite order at $\rb_0$, that is  $$P_b \circ G_{b, 4} - E_{b, 3} = - E_{b, 4} \in   \rho_{\mathrm{bf}}^\infty \rho_{\lb_0}^\infty \rho_{\rb_0}^\infty \Psi_b^{-\infty, (\infty, \hat{\mathcal{E}}_b)}(M_{sc, b, 0}; {}^b\Omega^{1/2}).$$

%Finally we obtain a parametrix $G_b$ with a remainder $E_b$ such that \begin{eqnarray*}r^{-1}G_br'^{-1} = G_{b, 1, 2} + G_{b, 3} + G_{b, 4} &\in&  \Psi_b^{-2, (\hat{\mathcal{E}}_b, \check{\mathcal{E}}_b)}(M_{sc, b, 0}; {}^b\Omega^{1/2}),  \\ r^{-1} E_b r'^{-1} = E_{b, 4} &\in&   \rho_{\mathrm{bf}}^\infty  \rho_{\lb_0}^\infty \rho_{\rb_0}^\infty \Psi_b^{-\infty, (\infty, \hat{\mathcal{E}}_b)}(M_{sc, b, 0}; {}^b\Omega^{1/2}).\end{eqnarray*} Then $G_b$ and $E_b$ satisfy the conditions in Proposition \ref{thm : b-parametrix}, since $rr'$ lifts to $\rho_{\mathrm{bf}}^2 \rho_{\lb} \rho_{\rb}$ in the $b$-regime of $M_{sc, b, 0}$.

\section{The true resolvent}\label{sec : true resolvent}

Through Proposition \ref{thm : b-parametrix}, Proposition \ref{prop : b-sc parametrix}, Proposition \ref{thm : sc-parametrix}, and Proposition \ref{prop : sc-b parametrix},
we have constructed the parametrices $G_{b}$, $G_{(b, sc)}$, $G_{sc}$, and $G_{(sc, b)}$ with the remainders $E_{b}$, $E_{(b, sc)}$, $E_{sc}$, and $E_{(sc, b)}$ in the $b$-regime, the ${(b, sc)}$-regime, the ${sc}$-regime, and the ${(sc, b)}$-regime respectively.  They satisfy
  \begin{multline*}\tilde{P}_\hbar \circ (G_{b}+ G_{(b, sc)} + G_{sc} + G_{(sc, b)})  -   \Id = - \tilde{E} = - E_{b} - E_{(b, sc)} - E_{sc} - E_{(sc, b)}, \\ E_{b} + E_{(b, sc)} \in     \rho_{\lb_0}^{\infty}   \Psi_b^{-\infty, (\infty, \hat{\mathcal{E}}_b + 1, \infty, \infty)} (M_{b, 0}; {}^{b, sc}\Omega^{1/2} ),\quad \mbox{near $\{r < 1\}$},\\   E_{sc} + E_{(sc, b)} \in \rho_{\bfz} 
e^{\imath / \rho'}  \rho_{\rf}^{(n-1- \varepsilon_r)/2} \rho_{\rb}^{n/2}  \rho^\infty_{\lf} \rho^\infty_{\mf}   C^\infty (M_{b, 0}; {}^{b, sc}\Omega^{1/2} ), \quad{\mbox{near $\{\rho < 1\}$}}.\end{multline*}

From this we obtain the true resolvent. 

\begin{theorem}\label{thm : full main theorem}
The resolvent kernel $R$ of $P_\hbar$ is a distribution on $M_{b, 0}$ consisting of
\begin{itemize}
\item $R_b = \chi_b R \chi_b$ is supported near the $b$-regime  and satisfies
\begin{itemize}
\item $R_b$ is  conormal of order $-2$ to  $\diag_{M_{b, 0}}$,
\item $R_b$ is polyhomogeneous conormal to $\lb, \rb, \mathrm{bf}, \bfz$  and lies in the space  $$\mathcal{A}^{(\hat{\mathcal{E}}_b + 1, \check{\mathcal{E}}_b + 1, \tilde{\mathcal{E}}_b + 2, \mathbb{N}_0)}(M_{b, 0}; {}^{b, sc}\Omega^{1/2}),$$ where the index sets $\hat{\mathcal{E}}_b, \check{\mathcal{E}}_b, \tilde{\mathcal{E}}_b$ are defined by \eqref{eqn : index set lb}, \eqref{eqn : index set rb}, \eqref{eqn : index set bf} respectively, 
\item $R_b$ vanishes to order $2$ at $\mathrm{bf}$, but also to order $n/2$ at $\lb$ and $\rb$;
\end{itemize}

\item $R_{(b,sc)} = \chi_{b} R \chi_{sc}$ is supported near the $(b, sc)$-regime such that \begin{itemize}
	\item $R_{(b,sc)}$ vanishes to order $n/2$ at $\lb$ and to order $(n-1)/2$ at $\lb_0$;
\end{itemize}

\item $R_{sc} = \chi_{sc} R \chi_{sc}$ is  supported near the $sc$-regime and satisfies
\begin{itemize}
\item $R_{sc}$ is conormal of order $-2$ to  $\diag_{M_{sc, b, 0}}$,
\item $R_{sc}$, near $\diag_{M_{b, 0}} \cap \mf$, is intersecting Legendrian distributions associated with the diagonal conormal bundle and  geometric Legendrian $L$, lying in the space $$I_{sc, c}^{-1/2} (\tilde{M}_{b, 0}; (N^\ast \diag_{M_{b, 0}}, L^\circ); {}^{b, sc}\Omega^{1/2}),$$ 
\item $R_{sc}$, near $(\lf \cap \mf) \cup (\rf \cap \mf)$, is intersecting Legendrian distributions associated with the pair of geometric Legendrian $L$ and diffractive Legendrian $L^\sharp$ explicitly given in \eqref{eqn : geometric Legendrian} and \eqref{eqn : diffractive Legendrian}, lying in the space
$$      I_{s\Phi}^{-1/2, (n-2)/2 - \mathbf{e}; ((n-1- \mathbf{e})/2, (n-1- \mathbf{e})/2, 0)}(M_{b, 0}; (L, L^\sharp); {}^{s\Phi}\Omega^{1/2}), $$ where $\mathbf{e}$ is defined in \eqref{eqn : stability}, 
\item The loss of order $\mathbf{e}$ is caused by the additional oscillations, $$\rho_{\lf}^{- \varepsilon_l/2} \rho_{\rf}^{-  \varepsilon_r/2} \rho_{\sharp}^{- (\varepsilon_l +  \varepsilon_r)/2},\quad \mbox{near $\lb_0, \rb_0, L^\sharp$},$$   of the principal symbol of $R_{sc}$ in \eqref{eqn : principal symbol of parametrix at bicharacteristic}. It occurs     (i.e. $e \not\equiv 0$)  if and only if the metric is of non-product type (i.e. the boundary metric $h(x, y, dy)$ is dependent of $x$), 
\item $R_{sc}$ vanishes  to $(n-1-\mathbf{e})/2$ order at  $\lf$ and $\rf$;
\end{itemize}

\item $R_{(sc,b)} = \chi_{sc} R \chi_{b}$ is supported near the $(sc, b)$-regime such that
\begin{itemize}
	\item $R_{(sc,b)}$ vanishes to order $n/2$ at $\rb$ and to order $(n-1-\mathbf{e})/2$ at  $\rb_0$.\end{itemize}

\end{itemize}
\end{theorem}

\begin{proof}
By standard Fredholm theory, obtaining the true resolvent from parametrices will be done by understanding the inversion of $\Id - \tilde{E}$ and its composition with the parametrices.

First of all,  $\Id - \tilde{E}$ is not necessarily invertible in the $sc$-regime even though $\tilde{E}$ is a compact operator. 
 Hassell-Vasy \cite[Section 6.1]{Hassell-Vasy-AnnFourier} showed that this issue can be fixed by adding to the parametrix a finite rank perturbation.  The remainder $\tilde{E}$ is a Hilbert-Schmidt kernel on the weighted $L^2$ space of the $sc$-regime, say $\rho^l \chi_{sc} L^2(\mathcal{C}; {}^{b, sc}\Omega^{1/2})$  uniformly in $\hbar$ for any $l > (1 + \mathbf{e})/2$ in virtue of the stability condition \eqref{eqn : stability}.\footnote{Recall that $\rho = \hbar/x$ and $\chi_{sc}$ is a cut-off function on $\mathcal{C}$ supported near $\{\rho < 1\}$.}  One can choose $\psi_i$ and $\phi_i$ supported near $\{\rho \leq 1\} \subset \mathcal{C}$ such that $\{\rho^l \psi_i\}$ spans the null space of $\Id - \tilde{E}$  and $\{P_{sc} (\phi_i)\}$ spans the supplementary subspace to the range. Then we correct $G_{sc}$ by adding a finite rank perturbation $G_{sc, p} =\sum_i \phi_i \langle \rho^{2l} \psi_i, \cdot \rangle$, where $\langle \cdot, \cdot \rangle$ denotes the $L^2$-inner product. This results in an improved remainder $E$ such that $$\Id - E =\Id - \tilde{E} + \sum (P_{sc} (\phi_i)) \langle \rho^{2l} \psi_i, \cdot\rangle.$$ is invertible on $\rho^l \chi_{sc} L^2(\mathcal{C}; {}^{b, sc}\Omega^{1/2})$. Moreover, this correction term does not change the structure of the resolvent.
 
Since $\Id - E$ is invertible now,  we can analyse the inverse $$\Id + S = (\Id - E)^{-1}.$$
 The identity  $$\Id = (\Id - E) (\Id + S) = (\Id + S) (\Id - E) $$ implies $S$ is also Hilbert-Schmidt and $$  S = E + E^2 + ESE.$$

 Recall that the remainder $E$ lies in the space $\mathcal{D}_b + \mathcal{D}_{sc}$, where \[ \left.\begin{array}{ll}
   \mathcal{D}_b =   \rho_{\lb_0}^{\infty}  \mathcal{A}^{(\infty, 
    \hat{\mathcal{E}}_b + 1, \infty, \infty)}(M_{b, 0}; {}^{b, sc}\Omega^{1/2}), & \mbox{near $\{r < 1\}$},  \\ \mathcal{D}_{sc} = \rho_{\bfz} 
 	e^{\imath / \rho'} 
 	\rho_{\rb}^{n/2} \rho_{\rf}^{(n-1)/2 -  \varepsilon_r/2} 
 	 \rho^\infty_{\lf} \rho^\infty_{\mf}  C^\infty(M_{b, 0}; {}^{b, sc}\Omega^{1/2}), & \mbox{near $\{\rho < 1\}$}.  \end{array} \right.  \] The term $E^2$ lies in the composition space $$(\mathcal{D}_b + \mathcal{D}_{sc}) \circ (\mathcal{D}_b + \mathcal{D}_{sc}).$$
On the one hand, we see $\mathcal{D}_b \circ \mathcal{D}_b \subset \mathcal{D}_b$ and $\mathcal{D}_{sc} \circ \mathcal{D}_{sc} \subset \mathcal{D}_{sc}$   simply  by the calculus of polyhomogeneous conormal distributions.
 On the other hand, it follows from the vanishing properties that 
  the cross terms obey $\mathcal{D}_b \circ \mathcal{D}_{sc} \subset \mathcal{D}_{sc}$ and $\mathcal{D}_{sc} \circ \mathcal{D}_b \subset \mathcal{D}_b$. Therefore $- E + E^2$ also lies in $\mathcal{D}_b + \mathcal{D}_{sc}$.
 	
 	To analyse $ESE$, it is convenient to blow down $M_{b, 0}$ to $M$.  The  boundary defining functions of boundary faces of $M_{b, 0}$ will be pushed forward to  
 	\begin{eqnarray*}	(\beta_{b, 0})_{\ast} \rho_{\lf}\rho_{\mf} = \langle r \rangle^{-1},
 	&	(\beta_{b, 0})_{\ast} \rho_{\rf}\rho_{\mf} =  \langle r' \rangle^{-1}, & (\beta_{b, 0})_{\ast} \rho_{\rf}  \rho_{\lf} \rho_{\bfz}  = \hbar \end{eqnarray*} in the $sc$-regime and the $(sc, b)$-regime, 
 \begin{eqnarray*}
 (\beta_{b, 0})_{\ast}	\rho_{\lb}\rho_{\mathrm{bf}} = r/\langle r \rangle, 
 \quad	(\beta_{b, 0})_{\ast}\rho_{\rb}\rho_{\mathrm{bf}} = r'/\langle r' \rangle,
 \quad (\beta_{b, 0})_{\ast} \rho_{\lb_0}    \rho_{\bfz}  = \hbar \end{eqnarray*} in the $b$-regime and the $(b, sc)$-regime, where we denote $r = x/\hbar$, $r' = x'/\hbar$ and the Japanese bracket $\langle \cdot \rangle = (1 + \cdot^2)^{1/2}$.
 The push-forward  ${\beta_{b, 0}}_\ast (E) (z, z', \hbar)$ thus lies in the space 
  \begin{equation}\label{eqn : final error space}
  \hbar^\infty  \bigg(\frac{r}{\langle  r \rangle}\bigg)^\infty  \langle r\rangle^{-\infty}  e^{\imath \langle r' \rangle} \langle r'\rangle^{-(n-1)/2 + \varepsilon_r/2} \sum_{(z, k) \in \hat{\mathcal{E}}_b + 1} \bigg(\frac{r'}{\langle r'\rangle}\bigg)^{z} \bigg\langle\log \bigg(\frac{r'}{\langle r'\rangle}\bigg)\bigg\rangle^k    C^\infty (M).   \end{equation} This distribution is weighted square integrable in the right variables, that is in the space $\langle r' \rangle^{-l} L^2(\mathcal{C})$. 
  
Since $S$ is Hilbert-Schmidt on $\langle r \rangle^{-l} L^2(\mathcal{C})$, the kernel of $S(z'', z''', \hbar)$ has to lie in the space $$\langle r'' \rangle^{-l} \langle r''' \rangle^l L^2(M) \quad \mbox{uniformly as $\hbar \rightarrow 0$}, $$ where   $r'' = x''/\hbar,  r''' = x'''/\hbar, z'' = (x'', y''), z''' = (x''', y''')$.

The term ${\beta_{b, 0}}_\ast (ESE)(z, z', \hbar)$ has the expression 
$$   \int_{\mathcal{C}^2}  \Big({\beta_{b, 0}}_\ast(E)(z, z'', \hbar)\Big) \Big(S(z'', z''', \hbar)\Big) \Big({\beta_{b, 0}}_\ast (E)(z''', z', \hbar)\Big)  dg(z'') dg(z''').$$ By Cauchy-Schwartz, this integral takes the form \eqref{eqn : final error space}, i.e. $ESE$ lies in $\mathcal{D}_b + \mathcal{D}_{sc}$. Then $S$ in turn lies in $\mathcal{D}_b + \mathcal{D}_{sc}$.

Finally, the desired true resolvent is \begin{eqnarray*}R = G (\Id + S) \quad \mbox{where $G =	G_{sc} + G_{sc, p} + G_{(sc, b)} + G_b + G_{(b, sc)}$}.
\end{eqnarray*} This is distributions on $M_{b, 0}$,\footnote{Although the $sc$-pseudodifferential part of $G_{sc}$ is defined on $M_{sc, b, 0}$, it doesn't affect the composition with $S$, since $S$ vanishes to infinite order at $\mf$ on $M_{b, 0}$.} whilst $S$ is a distribution on $M$.  To calculate their compositions, we introduce the projection maps from the triple space $M_{b, 0} \times \mathcal{C}$, \begin{align*}\mathcal{P}_G (M_{b, 0} \times \mathcal{C}) &= \beta_{b, 0} (M_{b, 0}) ,\\ \mathcal{P}_S (M_{b, 0} \times \mathcal{C}) &= \beta_\hbar \circ \tilde{\mathcal{P}}_R (M_{b, 0}) \times \mathcal{C},\end{align*} 
where $\beta_\hbar$ is the blow-down map from $\mathcal{C}_b$ to $\mathcal{C} \times [0, 1)$ and $\tilde{\mathcal{P}}_R$ is the stretched right projection from $M_{b, 0}$ to $\mathcal{C}_b$.
In local coordinates, we write $$(z, z'', \hbar) \in \mathcal{P}_G (M_{b, 0} \times \mathcal{C})\quad \mbox{and} \quad(z'', z', \hbar)  \in \mathcal{P}_S (M_{b, 0} \times \mathcal{C}).$$

We pull back $G$ and $S$ to the triple space $M_{b, 0} \times \mathcal{C}$. The pullback $\mathcal{P}_G^\ast (G(z, z'', \hbar))$  does not change as they are irrelevant to the variable from $\mathcal{C}$. But $\mathcal{P}^\ast_S(S(z'', z', \hbar))$ will lift to a distribution of the form
$$\rho_{\mf}^\infty\rho_{\mathrm{bf}}^\infty  \rho_{\rb}^{\infty}\rho_{\rb_0}^\infty    e^{\imath \langle r' \rangle} \langle r'\rangle^{-(n-1)/2 + \varepsilon_r} \sum_{(z, k) \in \hat{\mathcal{E}}_b + 1} \bigg(\frac{r'}{\langle r'\rangle}\bigg)^{z} \bigg\langle\log \bigg(\frac{r'}{\langle r'\rangle}\bigg) \bigg\rangle^k    C^\infty (M_{b, 0} \times \mathcal{C}).$$ 

The composition $G   S (z, z', \hbar)$ on $M_{b, 0}$ has the expression $$\int_{\mathcal{C}} \mathcal{P}_G^\ast \big(G(z, z'', \hbar)\big) \mathcal{P}_S^\ast \big(S(z'', z', \hbar)\big)  dg(z'').$$ This actually makes a $b$-fibration (in the sense of Melrose \cite{pushforward}) from $M_{b, 0} \times \mathcal{C}$ to $M_{b, 0}$ by projecting off the middle variables of $\mathcal{C}^3$ (or the right variables of $M_{b, 0}$). Since the polyhomogeneity and oscillations in $r'$ agree with the behaviour of $G$ in right variables, the push-forward theorem \cite[Theorem 5]{pushforward} guarantees $G   S$ still lies in the space we have described in the theorem.\footnote{The error term $G S$ is actually better than $G$. Since the infinite order vanishing at $\rf$ of $\mathcal{P}^\ast_S S$, we conclude that $G S$ does not have the Legendrian structure over $\rf$. But this is not important here.} 
\end{proof}

\section{Application to Schr\"odinger equations}

%The strategy to prove Theorem \ref{thm : Strichartz} is two-fold. 

%Since $\mathcal{X}$ is a Euclidean space glued with several local non-product conic manifolds $\mathcal{C}$, we shall establish  the local version \eqref{eqn : local Strichartz} on $\mathcal{C}$ first. 
%The idea is to use the semiclassical resolvent  on $\mathcal{C}$ to analyse the Schr\"odinger propagator by the $TT^\ast$ approach, and then apply the classical Keel-Tao endpoint method.

%Strichartz estimates \eqref{eqn : Schrodinger equation} on $\mathcal{X}$  will subsequently be obtained by integrating the local estimates and the well-known Euclidean Strichartz result. This will be done by a standard (dual)local smoothing argument on  $\mathcal{X}$.

\subsection{The Schr\"odinger propagator on $\mathcal{C}$}

Now let us consider the Schr\"odinger propagator $e^{\imath t \Delta_{\mathcal{C}}}$ on a non-product conic manifold $\mathcal{C}$ with a single cone tip, which is viewed as a local neighbourhood of $\mathcal{X}$. By the spectral theorem of projection valued form, the Schwartz kernel of $e^{\imath t \Delta_{\mathcal{C}}}$ reads, $$e^{\imath t  \Delta_{\mathcal{C}}} = \int_{\mathrm{spec}(\Delta_{\mathcal{C}})} e^{\imath t \lambda^2} dE_{\sqrt{\Delta_{\mathcal{C}}}}(\lambda),$$
where $E_{\bullet}(\lambda)$ is the spectral projection corresponding to a self-adjoint operator $\bullet$, $dE_{\bullet}(\lambda)$ is called the spectral measure of $\bullet$, and
this equation means $$\langle e^{\imath t  \Delta_{\mathcal{C}}}\psi, \psi
\rangle = \int_{\mathrm{spec}(\Delta_{\mathcal{C}})} e^{\imath t \lambda^2} \, d \langle E_{\sqrt{\Delta_{\mathcal{C}}}}(\lambda)
\psi, \psi\rangle.$$

Stone's formula relates the spectral measure to the resolvent of $\Delta_{\mathcal{C}}$, $$dE_{\sqrt{\Delta_{\mathcal{C}}}}(\lambda) = \frac{\lambda}{\imath \pi}\Big(\big(\Delta_{\mathcal{C}} - (\lambda + \imath 0)^2\big)^{-1} - \big(\Delta_{\mathcal{C}} - (\lambda - \imath 0)^2\big)^{-1}\Big) d\lambda.$$ By combining this together with $\Delta_{\mathcal{C}} - (\lambda \pm \imath 0)^2  = \lambda^2 P_\hbar(\pm \imath 0)$, we see $$dE_{\sqrt{\Delta_{\mathcal{C}}}}(\lambda) = \frac{1}{\imath \pi}\Big(P_{\hbar}^{-1}(+ \imath 0) - P_{\hbar}^{-1}(- \imath 0)\Big) \frac{d\lambda}{\lambda}.$$

 For finite spectral parameters, Loya \cite[Theorem 6.1]{Loya} proved that $(\Delta_{\mathcal{C}} - (\lambda \pm \imath 0)^2)^{-1}$ is meromorphic in $\mathbb{C}$ and has only finite rank singularities. Then the spectral measure $dE_{\sqrt{\Delta_{\mathcal{C}}}}(\lambda)$ is also meromorphic and has only finite rank singularities. When the spectral parameter goes to infinity, we have constructed the resolvent  $P^{-1}_{\hbar}(\pm \imath 0)$, as in Theorem \ref{thm : full main theorem}. The subtraction in Stone's formula cancels out the diagonal singularities. Therefore the spectral measure is distributions defined on $M_{b, 0}$, satisfying
\begin{itemize}
	\item 
	$\chi_{b} dE_{\sqrt{\Delta_{\mathcal{C}}}}(\lambda)$, supported in the $b$-regime and $(b, sc)$-regime, is a conormal distribution to $\lb$ and $\rb$, vanishing at $\lb_0$, $\rb_0$, and $\bfz$;
	\item $\chi_{sc} dE_{\sqrt{\Delta_{\mathcal{C}}}}(\lambda)$, supported in the $sc$-regime and $(sc, b)$-regime, is intersecting Legendrian distributions with conic points, microlocally supported in the Legendrian pair $\tilde{L}$, vanishing at $\lf$ and $\rf$.
\end{itemize}

Hence it is convenient to decompose the Schr\"odinger propagator into cut-off operators, \begin{eqnarray*}
e^{\imath t \Delta_{\mathcal{C}}}_{\mathrm{low}} &=& \int \chi_{\{\lambda < 2\}} e^{\imath t \lambda^2} dE_{\sqrt{\Delta_{\mathcal{C}}}}(\lambda),  \\
e^{\imath t \Delta_{\mathcal{C}}}_{\mathrm{high}} &=& \int \chi_{\{\lambda > 1\}} e^{\imath t \lambda^2} dE_{\sqrt{\Delta_{\mathcal{C}}}}(\lambda),
\end{eqnarray*}
where $\chi_{\{\lambda > 1\}}$ and $\chi_{\{\lambda < 2\}}$ are smooth cut-off functions supported in $ (1, \infty)$ and  $[0, 2)$ and satisfying $\chi_{\{\lambda < 2\}} + \chi_{\{\lambda > 1\}} = 1$ for any $\lambda \geq 0$.

It follows that  the low energy cut-off $e^{\imath t \Delta_{\mathcal{C}}}_{\mathrm{low}}$ is a bounded operator both from $L^2(\mathcal{C})$ to $L^2(\mathcal{C})$ and $L^1(\mathcal{C})$ to $L^\infty(\mathcal{C})$ uniformly in $t$. 

The high energy cut-off $e^{\imath t \Delta_{\mathcal{C}}}_{\mathrm{high}}$ is more involved. We encounter two hurdles. On the one hand, the spectral measure will have the conic singularities near $L^\sharp \cap L$. On the other hand, the expression of the Schr\"odinger propagator transforms the spectral parameter $\lambda$ to the time parameter $t$. Unlike the resolvent and the spectral measure, one has to employ the quadratic scattering cotangent bundle to describe the Schr\"odinger propagator in the phase space as in \cite{Wunsch-Duke1999, Hassell-Wunsch-annals, Hassell-Wunsch}. 

However, understanding this finer structure does not seem necessary for our purpose. Instead, we shall adopt the classical   $TT^\ast$ approach and Keel-Tao's endpoint method to establish Strichartz estimates.  This can be achieved by 	(micro)localizing the spectral measure around the diagonal as in \cite{Guillarmou-Hassell-Sikora, Hassell-Zhang}.\footnote{Apart from removing conic singularities, the microlocalization on asymptotically conic manifolds in \cite{Guillarmou-Hassell-Sikora, Hassell-Zhang} also aims to avoid the potential conjugate points near the infinity. However, such an issue does not arise near the cone tip, since the geodesics are all simple here.}

\subsection{Microlocalization and the microsupport}The microlocalization refers to the refined localization on cotangent bundles by introducing smooth cut-offs also on fibre variables not only physical variables.  

The microsupport is useful to formulate the microlocalization by the conjugation. The microsupport of a distribution is defined by negating the right fibre variables of its wavefront set. Specifically, given a distribution $u$ in the $sc$-regime, we say $q' = (z, z', \zeta, \zeta')$ lies in the microsupport $\mathrm{WF}_{sc}'(u)$ if and only if $q = (z, z', \zeta, -\zeta') $ lies in the wave front set $\mathrm{WF}_{sc}(u)$. 

For a Legendrian distribution $U$ associated with a Legendrian $L$, its $sc$-wavefront set $\mathrm{WF}_{sc}(U)$ is contained in $L$. Then the $sc$-microsupport $\mathrm{WF}_{sc}'(U)$ is contained in $L' = \{q' : q \in L\}$.

For a $sc$-pseudodifferential operator $Q$, its $sc$-wavefront set $\mathrm{WF}_{sc}(U)$ is contained in the diagonal conormal bundle $\{(\rho, y, \rho, y, \tau, \mu, -\tau, -\mu)\}$. Since both left and right projections to the single space are  diffeomorphisms, \begin{eqnarray*}
&\pi_L& :  (\rho, y, \rho', y', \tau, \mu, \tau', \mu')   \rightarrow (\rho, y, \tau, \mu)\\ &\pi_R& :  (\rho, y, \rho', y', \tau, \mu, \tau', \mu')    \rightarrow (\rho', y', \tau', \mu'),
\end{eqnarray*} we can view the $sc$-microsupport $\mathrm{WF}_{sc}'(Q)$ as a subset of  $\{(\rho, y, \tau, \mu)\}$ in the left space, and the $sc$-microsupport  $\mathrm{WF}_{sc}'(Q^\ast)$ as a subset of  $\{(\rho', y', \tau', \mu')\}$ in the right space where the fibre variables are negated.

Conjugating Legendrian distributions $U$ by a $sc$-pseudodifferential operator $Q$, we can microlocalize the Legendrian distributions, in terms of the microsupport of $Q$ and $U$,
 \begin{eqnarray}\label{eqn : microlocalizing left} \mathrm{WF}'_{sc}(QU) &\subset& \pi_L^{-1} \mathrm{WF}'_{sc}(Q) \cap \mathrm{WF}_{sc}'(U)\\ \label{eqn : microlocalizing right} \mathrm{WF}'_{sc}(UQ^\ast) &\subset& \pi_R^{-1} \mathrm{WF}'_{sc}(Q^\ast) \cap \mathrm{WF}_{sc}'(U)\\
 \label{eqn : microlocalization}
\mathrm{WF}_{sc}'(QUQ^\ast) &\subset& \pi_L^{-1} \mathrm{WF}_{sc}'(Q)    \cap \mathrm{WF}_{sc}'(U) \cap \pi_R^{-1} \mathrm{WF}_{sc}'(Q^\ast).
\end{eqnarray}

\subsection{The conjugated spectral measure} 

Consider the conjugation of the spectral measure, $Q_\alpha(\lambda) dE_{\sqrt{\Delta_{\mathcal{C}}}}(\lambda) Q_\alpha^\ast(\lambda),$ by a family of pseudodifferential operators $\{Q_\alpha(\lambda)\}_\alpha$.
In light of \eqref{eqn : microlocalization}, we can view this as the desired (micro)localized spectral measure by choosing the support of every $Q_\alpha(\lambda)$ sufficiently small in the physical space as well as in the phase space. Consequently, $Q_\alpha(\lambda) dE_{\sqrt{\Delta_{\mathcal{C}}}}(\lambda) Q_\alpha^\ast(\lambda)$ will have to be supported in a neighbourhood of the diagonal with $\mf \cap \diag_{M_{b, 0}}$ removed (which is cancelled out in Stone's formula). In particular, we can make them supported away from the conic points in the $sc$-regime. This significantly simplifies the microlocal structure in the $sc$-regime, as $Q_j(\lambda) dE_{\sqrt{\Delta_{\mathcal{C}}}}(\lambda) Q_j^\ast(\lambda)$ is a Legendrian distribution  microlocally supported only on the geometric Legendrian $L$. 

Specifically, we shall prove that
\begin{proposition}\label{prop : microlocalized spectral measure}
	For $\lambda > 1$, there exists a finite collection of pseudodifferential operators $\{Q_\alpha(\lambda)\}_{\alpha \in \{b, tr, 0, 1, \cdots, N\}}$   such that it makes a finite partition of identity on $L^2(\mathcal{C})$, $$\Id = Q_b(\lambda) + Q_{tr}(\lambda) + Q_0(\lambda) + \sum_{j = 1}^N Q_j(\lambda),$$  and the kernel of each conjugated spectral measure on $M = \mathcal{C}^2 \times [0,1)$ takes the form $$Q_\alpha(\lambda)dE_{\sqrt{\Delta_{\mathcal{C}}}}(\lambda)Q_\alpha^\ast(\lambda)(z, z') = \lambda^{n-1 + \mathbf{e}/2} \bigg(\sum_{\pm} e^{\pm \imath \lambda d_{\mathcal{C}}(z, z')} a_\pm(z, z', \lambda) + b(z, z', \lambda) \bigg),$$ where $\mathbf{e}$ is defined in \eqref{eqn : stability}, $d_{\mathcal{C}}(\cdot, \cdot)$ is the distance function on $\mathcal{C}$ and the amplitudes satisfy the size estimates \begin{eqnarray}\label{eqn : estimates for a}
	|\partial_\lambda^k a_\pm(\lambda, z, z')| &\lesssim&  \lambda^{-k} (1 + \lambda d_{\mathcal{C}}(z, z'))^{-(n-1)/2},\\ \label{eqn : estimates for b}|\partial_\lambda^k b(\lambda, z, z')| &\lesssim&  (1 + \lambda d_{\mathcal{C}}(z, z'))^{-K}, \quad \mbox{for any $K \in \mathbb{N}$}.
	\end{eqnarray}
\end{proposition}

\begin{proof}
We apply the ideas from \cite[Theorem 1.13]{Guillarmou-Hassell-Sikora} and its improvement \cite[Proposition 1.5]{Hassell-Zhang} to our analysis on $M_{b, 0}$.

\subsubsection*{Partition of identity}

Before showing the estimates, we start from elucidating how we find the partition of identity $\{Q_\alpha(\lambda)\}_{\alpha}$. We view these pseudodifferential operators as operators compactly supported on $\mathcal{C}$ with a parameter $\hbar = 1/\lambda$. Then the conjugation $Q_\alpha(\lambda) \bullet Q_\alpha^\ast(\lambda)$ is viewed as a restriction to a neighbourhood of the diagonal in $M_{b, 0}$. The diagonal of $M_{b, 0}$ is a submanifold which intersects the faces $\{\mathrm{bf}, \bfz, \mf\}$. However, $\bfz$ is viewed as the interior since both $0 < r < \infty$ and $0 < r' < \infty$.  

First of all, we divide $\mathcal{C}$, with a parameter $\hbar \in [0, 1)$, into three overlapping parts, \begin{eqnarray*}\mathcal{C}_b(\hbar) &=& \{r = x/\hbar < \epsilon\}\\ \mathcal{C}_{tr}(\hbar) &=& \{r = x/\hbar \geq \epsilon/2\} \cap \{  \rho = \hbar/x \geq \epsilon/2 \}\\ \mathcal{C}_{sc}(\hbar) &=& \{\rho = \hbar/x < \epsilon\}   \end{eqnarray*}for some sufficiently small $\epsilon >0$.\footnote{The interior $\mathcal{C} \setminus \mathcal{K}$ is thought of as a part in $\mathcal{C}_{sc}$.}  Then there exists a partition of identity on $\mathcal{C}$, $$\Id = Q_b(\lambda) + Q_{tr}(\lambda) + Q_{sc}(\lambda),$$ with $0$-th degree  pseudodifferential operators  $Q_b(\lambda)$, $Q_{tr}(\lambda)$, and $Q_{sc}(\lambda)$ supported in $\mathcal{C}_b(\hbar)$, $\mathcal{C}_{tr}(\hbar)$, and $\mathcal{C}_{sc}(\hbar)$ respectively. This leads to the following consequences : 
 \begin{itemize} 
\item the conjugation $Q_b(\lambda) \bullet Q_b^\ast(\lambda)$ on $M_{b, 0}$ is supported in a neighbourhood of $\mathrm{bf}$,
 \item the conjugation $Q_{tr}(\lambda) \bullet  Q_{tr}^\ast(\lambda)$ on $M_{b, 0}$  is supported away from the boundaries,\footnote{The interior of the face $\bfz$ is viewed as the interior!} 
  \item the conjugation $Q_{sc}(\lambda)\bullet Q_{sc}^\ast(\lambda)$ on $M_{b, 0}$  is supported in a neighbourhood of $\mf$.\end{itemize}

Furthermore, we want to refine the decomposition  near $\mf$ to make the conjugation near $\mf$ away from the diffractive Legendrian $L^\sharp$ i.e. remove the conic singularities occur at $L^\sharp \cap L$. To accomplish this, we need to work on the phase space restricted to $\{\rho <  \epsilon\}$. Recall, from \eqref{eqn : geometric Legendrian} and \eqref{eqn : diffractive Legendrian}, that the geometric Legendrian is contained in the pull-back to  ${}^{s\Phi}T^\ast M_{b, 0}$ of the characteristic variety $\Sigma(P_{sc}) \subset {}^{sc}T^\ast \mathcal{C}|_{\{\rho < \epsilon\}}$, $$\Sigma(P_{sc}) = \{(\rho, y, \tau, \mu) : \rho = 0,  \tau^2 + h^{ij} \mu_i \mu_j = 1\},$$ whilst the conic singularities of the resolvent and spectral measure occur on $(\mf\cap \lf) \cup (\mf\cap \rf) \subset M_{b, 0}$, in terms of microsupport, at $$ (L^\sharp)' \cap L' = \{(\tilde{\rho}, y, y', \hbar, \tau, \tau', \mu, \mu', 0) : \mu =  \mu' = 0, \tau = \pm 1, \tau' = \mp 1\}  \subset \mathrm{WF}'_{sc}(dE_{\sqrt{\Delta_{\mathcal{C}}}}),$$ where $\tilde{\rho} = \rho/\rho'$.  Hence we decompose ${}^{sc}T^\ast \mathcal{C}|_{\{\rho < \epsilon\}}$ into the following overlapping subsets \begin{eqnarray*}
\mathcal{I}_{0} &=&	\{(\rho, y, \tau, \mu) : \rho < \epsilon,  \tau^2 + h^{ij} \mu_i \mu_j \leq 1 - \epsilon/2 \,\,\mbox{or}\,\, \tau^2 + h^{ij} \mu_i \mu_j \geq 1 +  \epsilon/2\},\\
\mathcal{I}_j &=& \{(\rho, y, \tau, \mu) : \rho < \epsilon,  1 - \epsilon < \tau^2 + h^{ij} \mu_i \mu_j < 1 + \epsilon, \tau \in I_j \},
\end{eqnarray*} 
where we write $$\{\tau \in \mathbb{R} : 1 - \epsilon < \tau^2 < 1 + \epsilon\} = \bigcup_{j = 1}^N I_j \quad \mbox{with open connected intervals $|I_j| < \epsilon < 1$}.$$ Then we choose a partition of $Q_{sc}(\lambda)$, $$Q_{sc}(\lambda) = Q_0(\lambda) + \sum_{j = 1}^N Q_j(\lambda)$$ such that $Q_0(\lambda)$ and $Q_j(\lambda)$ are  pseudodifferential operators with a microsupport contained in $\mathcal{I}_0$ and $\mathcal{I}_j.$
In virtue of \eqref{eqn : microlocalization}, the gain of such a decomposition is 
 \begin{itemize} 
	\item the conjugation $Q_0(\lambda) \bullet Q_0^\ast(\lambda)$ on $M_{b, 0}$ is microlocally supported away from the characteristic variety of $P_{sc}$ and thus a smooth function,
	\item the microsupport of the conjugation $Q_{j}(\lambda)\bullet Q_{j}^\ast(\lambda)$ on $M_{b, 0}$  is  away from $(L^\sharp)'$, as $|\tau- \tau'| \leq |I_j| < \epsilon$, i.e. $Q_{j}(\lambda)\bullet Q_{j}^\ast(\lambda)$ is microlocally supported away from the diffractive Legendrian $L^\sharp.$\end{itemize}

\subsubsection*{Distance function on $\mathcal{C}$}

Since $\mathcal{C}$ is bounded and features only simple geodesics, the explicit distance function for two points $z=(x, y)$ and $z'=(x',y')$ sufficiently close in $(\mathcal{C}, g)$ reads \begin{eqnarray*}d_{\mathcal{C}}(x, y, x', y') &=& \sqrt{x^2 + x'^2 - 2xx'\cos d_{\mathcal{Y}}(y, y')}\\ 
	&=& \hbar \sqrt{r^2 + r'^2 - 2rr'\cos d_{\mathcal{Y}}(y, y')}\\ 
	&=& \hbar \sqrt{1/\rho^2 + 1/\rho'^2 - 2\cos d_{\mathcal{Y}}(y, y')/(\rho\rho')},\end{eqnarray*} if we use $r = \rho^{-1} = x/\hbar$ and $r' = \rho'^{-1} = x'/\hbar$.   
In the projective coordinates $(\tilde{r} = r/r' = x/x', y, r' = x'/\hbar, y' )$, the distance function lifts to,  $$ d_b(\tilde{r}, y, r', y', \hbar) = \hbar r' \sqrt{\tilde{r}^2 + 1 - 2\tilde{r}\cos d_{\mathcal{Y}}(y, y')},$$ near  $\mathrm{bf}$ in the $b$-regime. 
In the projective coordinates $(\tilde{\rho} = \rho/\rho'=x'/x, y, \rho = \hbar/x, y' )$, the distance function lifts to,  $$ d_{sc}(\rho, \tilde{\rho}, y, y', \hbar) = \frac{\hbar}{\rho} \sqrt{1 + \tilde{\rho}^{2} - 2\tilde{\rho}\cos d_{\mathcal{Y}}(y, y')},$$ near $\mf$ in the $sc$-regime.

\subsubsection*{Conjugation by $Q_{tr}(\lambda)$} This  conjugation $Q_{tr}(\lambda) dE_{\sqrt{\Delta_{\mathcal{C}}}}(\lambda) Q_{tr}^\ast(\lambda)$ is localized  in the interior of $M_{b, 0}$, $$\{r \geq \epsilon/2  \} \cap \{r' \geq \epsilon/2 \} \cap \{\rho \geq \epsilon/2   \} \cap \{\rho' \geq \epsilon/2\} \subset M_{b, 0}.$$ These restrictions guarantee $(1 + \lambda d_\mathcal{C}(z, z'))^{-K}$ is bounded from below. In the meantime, the conjugated spectral measure $Q_{tr}(\lambda) dE_{\sqrt{\Delta_{\mathcal{C}}}}(\lambda) Q_{tr}^\ast(\lambda)$ is a smooth Riemannian half density of the form $$f |r^{n-1} r'^{n-1} dr dy dr' dy'|^{1/2} \frac{d\lambda}{\lambda} = \lambda^{n-1} f  |x^{n-1} x'^{n-1}dxdydx'dy'|^{1/2}d\lambda,$$  for some $f \in C^\infty(M)$. So $Q_{tr}(\lambda) dE_{\sqrt{\Delta_{\mathcal{C}}}}(\lambda) Q_{tr}^\ast(\lambda)$ is of the form $\lambda^{n - 1} b(z, z', \lambda)$ and \eqref{eqn : estimates for b} holds.

\subsubsection*{Conjugation by $Q_b(\lambda)$}  This conjugation $Q_b(\lambda) dE_{\sqrt{\Delta_{\mathcal{C}}}}(\lambda) Q_b^\ast(\lambda)$ is localized in a neighbourhood of $\mathrm{bf}$ in the $b$-regime, $$\{r < \epsilon\} \cap \{r' < \epsilon\} \subset M_{b, 0}.$$ In this region,  $(1 + \lambda d_\mathcal{C}(z, z'))^{-K}$ is likewise bounded from below. The conjugated spectral measure $Q_b(\lambda) dE_{\sqrt{\Delta_{\mathcal{C}}}}(\lambda) Q_b^\ast(\lambda)$ is a distribution on $M_{b, 0}$ polyhomogeneous conormal to $\lb$ and $\rb$. Hence it takes the form 
\begin{eqnarray*}\sum_{(z_l, k_l) \in \hat{\mathcal{E}}_b + 1}  \rho_{\lb}^{z_l}   (\log \rho_{\lb})^{k_l}
    a_{(z_l, k_l)}(\rho_{\lb}, y, r', y', \hbar) \bigg|\frac{d\rho_{\lb} dydr'dy'}{\rho_{\lb}r'}\bigg|^{1/2} \frac{d\lambda}{\lambda} && \mbox{away from $\rb$}, \\
  \sum_{ (z_r, k_r) \in \check{\mathcal{E}}_b + 1}  \rho_{\rb}^{z_r}   (\log \rho_{\rb})^{k_r}
  a_{(z_r, k_r)}(r, y, \rho_{\rb}, y', \hbar) \bigg|\frac{dr dyd\rho_{\rb}dy'}{r\rho_{\rb}}\bigg|^{1/2} \frac{d\lambda}{\lambda} && \mbox{away from $\lb$}. \end{eqnarray*} 
The blow down from $M_{b, 0}$ to $M$  pushes $Q_b(\lambda) dE_{\sqrt{\Delta_{\mathcal{C}}}}(\lambda) Q_b^\ast(\lambda)$ forward to \begin{eqnarray*}\sum_{\substack{(z_l, k_l) \in \hat{\mathcal{E}}_b + 1 \\ (z_r, k_r)\in \check{\mathcal{E}}_b + 1}}  (x/\hbar)^{z_l}   (\log (x/\hbar))^{k_l} (x'/\hbar)^{z_r}   (\log (x'/\hbar))^{k_r}
 \frac{a_{(z_l, k_l, z_r, k_r)}(x, y, x', y', \hbar) }{x^{n/2}x'^{n/2}}  |dg dg' |^{1/2}  \frac{d\lambda}{\lambda}. 
\end{eqnarray*} Then the fact $\inf \hat{\mathcal{E}}_b = \inf \check{\mathcal{E}}_b = n/2-1$ yields that $Q_b(\lambda) dE_{\sqrt{\Delta_{\mathcal{C}}}}(\lambda) Q_b^\ast(\lambda)$ also takes the form $\lambda^{n - 1} b(z, z', \lambda)$ with the amplitude estimates \eqref{eqn : estimates for b}.

\subsubsection*{Conjugation by $Q_0(\lambda)$} This conjugation $Q_0(\lambda) dE_{\sqrt{\Delta_{\mathcal{C}}}}(\lambda) Q_0^\ast(\lambda)$ is localized in a neighbourhood of $\mf$ in the $sc$-regime in the physical space $M_{b, 0}$, $$\{\rho < \epsilon\} \cap \{\rho' < \epsilon\} \subset M_{b, 0}.$$ However, in the phase space level, it is microlocalized only in the elliptic region of $P_{sc}$, $$\{\rho < \epsilon, \rho' < \epsilon, \sigma_{\partial}(P_{sc}) \leq   - \epsilon/2\} \cup \{\rho < \epsilon, \rho' < \epsilon, \sigma_{\partial}(P_{sc}) \geq   \epsilon/2\} \subset {}^{b, sc} T^\ast M_{b, 0},$$ where $\sigma_\partial(P_{sc}) = \tau^2 + h^{ij}\mu_i\mu_j - 1$. Therefore, $Q_0(\lambda) dE_{\sqrt{\Delta_{\mathcal{C}}}}(\lambda) Q_0^\ast(\lambda)$ is a smooth section of the $sc$-half density on $M_{b, 0}$.

We can write the conjugated spectral measure as \begin{eqnarray*}
 Q_0(\lambda) dE_{\sqrt{\Delta_{\mathcal{C}}}}(\lambda) Q_0^\ast(\lambda) &=& T_1 + T_2 + T_3 +T_4\\
T_1 &=&   dE_{\sqrt{\Delta_{\mathcal{C}}}}(\lambda) \\
T_2 &=& - (\Id - Q_0(\lambda)) dE_{\sqrt{\Delta_{\mathcal{C}}}}(\lambda) \\
T_3 &=& -  dE_{\sqrt{\Delta_{\mathcal{C}}}}(\lambda) (\Id - Q_0^\ast(\lambda)) \\T_4 &=&  (\Id - Q_0(\lambda)) dE_{\sqrt{\Delta_{\mathcal{C}}}}(\lambda) (\Id - Q_0^\ast(\lambda)).
\end{eqnarray*} Since $\mathrm{WF}_{sc}(Q_0(\lambda)) \cap \mathrm{WF}_{sc}(dE_{\sqrt{\Delta_{\mathcal{C}}}}(\lambda))=\emptyset$, the microlocalizing property \eqref{eqn : microlocalizing left}  shows that $T_1 + T_2$ and $T_3 + T_4$ vanish to infinite order at $\lf$ and $\mf$, whilst  \eqref{eqn : microlocalizing right} proves that $T_1 + T_3$ and $T_2 + T_4$ vanish to infinite order at $\lf$ and $\mf$. Therefore $Q_0(\lambda) dE_{\sqrt{\Delta_{\mathcal{C}}}}(\lambda) Q_0^\ast(\lambda)$ vanishes to infinite order at $\lf$, $\rf$ and $\mf$. Then it is obvious that $ Q_0(\lambda) dE_{\sqrt{\Delta_{\mathcal{C}}}}(\lambda) Q_0^\ast(\lambda)$ takes the form $\lambda^{n - 1} b(z, z', \lambda)$ with $b$ satisfying \eqref{eqn : estimates for b}.

\subsubsection*{Conjugation by $Q_j(\lambda)$} Every conjugation
 $Q_j(\lambda) dE_{\sqrt{\Delta_{\mathcal{C}}}}(\lambda) Q_j^\ast(\lambda)$ is microlocalized in a neighbourhood of the interior of  $L \setminus L^\sharp$ over $\mf\cup\lf\cup\rf$ in the $sc$-regime.

The geometric Legendrian is the geodesic flow emanating from the intersection of the diagonal conormal bundle $N^\ast \diag_{M_{b, 0}}$ with the characteristic variety of $P_{sc}$  over $\mf$, $$\partial_0 L = \{(\tilde{\rho}, y, \rho', y, \hbar, \tau, \mu, -\tau, -\mu, 0) : \tilde{\rho} = 1, \rho' = 0, \tau^2 + h^{ij}\mu_i\mu_j = 1\}.$$ As a boundary of $L$, $\partial_0 L$ doesn't project diffeomorphically  to the base. Since the geodesics in $\mathcal{C}$ are all simple i.e. the exponential map is a diffeomorphism, $L\setminus \partial_0 L$ projects diffeomorphically to the base. Furthermore,  the proof of \cite[Proposition 6.2, Lemma 6.4, Lemma 6.5]{Guillarmou-Hassell-Sikora} and \cite[Proposition 2.6]{Hassell-Zhang} can be used verbatim to elucidate the Legendrian structure.

\begin{lemma}
	The Legendrian pair $(N^\ast \partial_{sc} \diag_{M_{b, 0}}, L)$ has a clean intersection $\partial_0 L$. The projection $\pi : L \rightarrow \mf$ satisfies that \begin{itemize}
		\item $\pi$ is a diffeomorphism on $L \setminus \partial_0 L$, 
		\item $\det d\pi$ vanishes to order $n-1$ at $\partial_0 L$. 
	\end{itemize}It follows that the geometric Legendrian $L$ is furnished with local coordinates,
\begin{itemize}
	\item  $(\tilde{\rho}, y, y', \hbar)$ away from $\partial_0 L$,
	\item $(y, \textbf{w}_1, \textbf{k}_2, \cdots, \textbf{k}_n, \hbar)$, in a neighbourhood of $\partial_0 L$ with $\partial_0 L$ removed, where $\textbf{w} = (\tilde{\rho} - 1, y - y')$ and $\textbf{k}$ is the dual coordinates on the fibre of ${}^{b, sc}T^\ast_{\mf} M_{b, 0}$. Here we need relabel the components of $\textbf{w}$ and $\textbf{k}$ variables to ensure $d \textbf{k}_1$ vanishes at a point in $\partial_0 L$, as $d\pi$ is not of full rank. 
\end{itemize}
	Consequently, the geometric Legendrian is parametrized,  \begin{itemize}
	\item away from $\partial_0 L$, by the distance function $d_{sc}(\rho, \tilde{\rho}, y, y', \hbar)$ 
	\item in a neighbourhood of $\partial_0 L$ with $\partial_0 L$ removed, by a function $\Phi(y, \textbf{w}, \textbf{v})$ with $\textbf{v} = (\textbf{v}_2, \cdots, \textbf{v}_n) \in \mathbb{R}^{n-1}$ such that \begin{eqnarray} \label{eqn : derivative of phase}
	d_{\textbf{v}_j} \Phi(y, \textbf{w}, \textbf{v}) &=& \textbf{w}_j + O(\textbf{w}_1),\\
	\label{eqn : expansion of phase}\Phi(y, \textbf{w}, \textbf{v}) &=& \sum_{j=2}^n \textbf{v}_j d_{\textbf{v}_j} \Phi(y, \textbf{w}, \textbf{v}) + O(\textbf{w}_1),\\ \label{eqn : nondegeneracy of phase}
	d_{\textbf{v}_j\textbf{v}_k}^2 \Phi(y, \textbf{w}, \textbf{v}) &=& \textbf{w}_1 A_{jk}(y, \textbf{w}, \textbf{v}), \quad \mbox{with a non-degenerate matrix $(A_{jk})_{jk}$,}\\ \label{eqn : phase and distance} \Phi(y, \textbf{w}, \textbf{v})/\rho &=& \pm \lambda d_{sc}(\rho, \tilde{\rho}, y, y', \hbar), \quad \mbox{ when $d_\textbf{v} \Phi = 0$}.
	\end{eqnarray}
\end{itemize}
Finally, the conjugated spectral measure $Q_j(\lambda) dE_{\sqrt{\Delta_{\mathcal{C}}}}(\lambda) Q_j^\ast(\lambda)$ is a Legendrian distribution lying in $ I^{-1/2; ((n-1-\mathbf{e})/2, (n-1-\mathbf{e})/2, 0)}(M_{b, 0}; L \setminus L^\sharp; {}^{b, sc}\Omega^{1/2})$ tensored with $|d\lambda/\lambda|^{1/2}$ and has the microlocal expression, away from $\rf$, \begin{itemize}
		\item when the conjugation $Q_j(\lambda) \bullet Q_j^\ast(\lambda)$ is microlocally supported away from a neighbourhood of $\partial_0 L$, \begin{equation}\label{eqn : projective type}\rho'^{-1/2 + n/2} \tilde{\rho}^{ (n - 1 - \varepsilon_l)/2} e^{\imath \pm \lambda d_{sc}(\rho, \tilde{\rho}, y, y', \hbar)} \tilde{a}(\tilde{\rho}, y, \rho', y', \hbar) \frac{|d\rho d\rho' dy dy'|^{1/2}}{\rho^{(n+1)/2} \rho'^{(n+1)/2}} \frac{d\lambda}{\lambda},\end{equation} where $\tilde{a}(\tilde{\rho}, y, \rho', y', \hbar)$ is a compactly supported smooth function in the $sc$-regime,
	 
		\item when the conjugation $Q_j(\lambda) \bullet Q_j^\ast(\lambda)$ is microlocally supported in a neighbourhood of $\partial_0 L$ with $\partial_0 L$ removed, \begin{equation}\label{eqn : n-1 vanish type}\rho'^{-1/2 - (n-1)/2 + n/2}   \int_{\mathbb{R}^{n-1}} e^{\imath \Phi(y, \textbf{w}, \textbf{v})/\rho} \tilde{a}(\rho, y, \textbf{w}, \textbf{v}, \hbar) d\textbf{v} \frac{|d\rho d\rho' dy dy'|^{1/2}}{\rho^{(n+1)/2} \rho'^{(n+1)/2}}\frac{d\lambda}{\lambda},\end{equation} where $\tilde{a}(\rho, y, \textbf{w}, \textbf{v}, \hbar)$ is a compactly supported smooth function in the $sc$-regime.
	\end{itemize} 
\end{lemma}

 Now we consider the estimates for the conjugated spectral measure of type \eqref{eqn : projective type}. Such an expression  we choose implies that we live in the $sc$-regime near $\mf \cap \lf$, i.e. $\tilde{\rho} = \rho/\rho' \rightarrow 0$. The local expression \eqref{eqn : projective type} further reduces to
 \begin{equation*}  \tilde{\rho}^{-\varepsilon_l/2}\rho^{ (n - 1)/2} e^{\imath \pm \lambda d_{sc}(\rho, \tilde{\rho}, y, y', \hbar)} \tilde{a}(\tilde{\rho}, y, \rho', y', \hbar) \bigg|\frac{d\rho d\rho' dy dy'}{\rho^{n+1} \rho'^{n+1}}\bigg|^{1/2}\frac{d\lambda}{\lambda}.\end{equation*}

  Using the local expression of the distance function $d_{sc}$, we see that the factor $\rho^{(n-1)/2}$ is bounded from above by $$\rho^{(n-1)/2} \lesssim (1 + 1/\rho)^{-(n-1)/2} \lesssim (1 +  \lambda d_{sc}(\rho, \tilde{\rho}, y, y', \hbar))^{-(n-1)/2}.$$ We thus take $a_\pm = \rho^{(n-1)/2} \tilde{a}$, which satisfies the amplitude estimates \eqref{eqn : estimates for a}.

We now blow down $M_{b, 0}$ to $M = \{(z, z', \hbar)\}$ and change $\hbar$ to $\lambda^{-1}$ as we need. It follows that  \eqref{eqn : projective type} is pushed forward to, \begin{equation*}\label{eqn : projective type 1}  \lambda^{n-1} e^{\imath \pm \lambda d_{\mathcal{C}}(z, z')}  x'^{ - \varepsilon_l/2}  \tilde{\tilde{a}}(z, z', \lambda) x^{(n-1)/2} x'^{(n-1)/2}  |dx dx' dy dy'|^{1/2}d\lambda.\end{equation*} The factor $x^{(n-1)/2} x'^{(n-1)/2}  |dx dx' dy dy'|^{1/2}$ gives a Riemannian half density  $|dgdg'|^{1/2}$, and the stability condition \eqref{eqn : stability} implies $x'^{ - \varepsilon_l/2} \leq \lambda^{\mathbf{e}/2}$ in the $sc$-regime. It follows that we get an oscillatory function $\lambda^{n-1 + \mathbf{e}/2} e^{\pm \imath \lambda d_{\mathcal{C}}(z, z')} a_\pm$ up to a smooth function.

The other type \eqref{eqn : n-1 vanish type} also takes the form $\lambda^{n-1} e^{\pm \imath \lambda d_{\mathcal{C}}(z, z')} a_\pm$ but it needs a bit more work due to the complication of the Legendrian structure at the intersection with the diagonal. Recall that we now live in a small neighbourhood of $\partial_0 L$. In a small neighbourhood of $\{\textbf{w}=0\}$, we have $|\textbf{w}|/\rho \sim \lambda d_{\mathcal{C}}(z, z')$. In fact, it is easy to deduce that
\begin{eqnarray*}
	\frac{|\textbf{w}|}{\rho} &=& \frac{1}{\rho}\sqrt{(\frac{\rho}{\rho'} - 1)^2 + (y-y')^2}\\
	&=& \sqrt{\frac{1}{\rho^2} + \frac{1}{\rho'^2} - \frac{2}{\rho\rho'}(1 - \frac{\rho'}{2\rho}(y-y')^2)}
	\\&\sim&\sqrt{\frac{1}{\rho^2} + \frac{1}{\rho'^2} - \frac{2}{\rho\rho'} \cos d_Y(y, y')}\\
	&=& \lambda d_{\mathcal{C}}(z, z')
\end{eqnarray*}

If we blow down to $M$ again,\footnote{Even though we no longer live in the stretched spaces, we still adopt the notations $\rho = 1/(\lambda x)$ and $\textbf{w} = (x'/x - 1, y - y')$ for the sake of convenience.} the distribution \eqref{eqn : n-1 vanish type} is pushed forward to
$$ \lambda^{n-1}    \bigg(\int_{\mathbb{R}^{n-1}} e^{\imath \Phi(y, \textbf{w}, \textbf{v})/\rho} \tilde{a}(\rho, y, \textbf{w}, \textbf{v}, \lambda) d\textbf{v} \bigg) x^{(n-1)/2} x'^{(n-1)/2 } | dx dx' dy dy'|^{1/2}   d\lambda.$$
 We complete the proof of the proposition by doing estimates for the integral in the parentheses in three cases.
\begin{enumerate}[(a)]
\item  $|\textbf{w}|  \lesssim \rho$.  This condition implies $\lambda d_{\mathcal{C}}(z, z') \lesssim 1$, i.e. $(1 + \lambda d_{\mathcal{C}}(z, z'))^{-K}$ is bounded from below by a positive constant. Then the amplitude estimate for $b$ is clearly satisfied when $k=0$. For $k > 0$, it suffices to show $$(\lambda \partial_\lambda)^k \bigg(\int_{\mathbb{R}^{n-1}} e^{\imath \Phi(y, \textbf{w}, \textbf{v})/\rho} \tilde{a}(\rho, y, \textbf{w}, \textbf{v}, \lambda) dv\bigg) \lesssim 1.$$ If the derivative $\lambda \partial_\lambda$ hits the amplitude, it affects nothing. If it acts on the oscillatory term once, by \eqref{eqn : expansion of phase} it will create a factor of $$\Phi/\rho = \textbf{v} \cdot d_\textbf{v} \Phi / \rho + O(\textbf{w}_1/\rho).$$ Noting $$\frac{\textbf{v} \cdot d_\textbf{v}\Phi}{\rho} e^{\imath \Phi/\rho} = - \imath \textbf{v} \cdot d_\textbf{v} e^{\imath \Phi/\rho},$$ we perform integration by parts and get a finite integral.

\item $|\textbf{w}| \gtrsim  \rho$ and $|\textbf{w}| \gtrsim |\textbf{w}_1|$. In light of \eqref{eqn : derivative of phase}, we have $d_{\textbf{v}_i} \Phi \neq 0$. Choose $i$ such that $|d_{\textbf{v}_i} \Phi| \gtrsim |\textbf{w}|$. When $k = 0$, by using $$e^{\imath \Phi/\rho} = \bigg(\frac{\rho d_{\textbf{v}_j}}{\imath d_{\textbf{v}_j} \Phi}\bigg)^k e^{\imath \Phi/\rho},$$ we make integration by parts in $\textbf{v}$ and obtain a finite integration, multiplied by the factor $(\rho/|\textbf{w}|)^K$ for any $K \in \mathbb{Z}_+$. This is bounded from above by $(1 + |\textbf{w}|/\rho)^{-K}$. If we blow down to $M$, this quantity is equivalent to $ (1 + \lambda d_{\mathcal{C}}(z, z'))^{-K}$. When $k > 0$, the derivatives are treated similarly to Case (a).

\item $|\textbf{w}| \sim |\textbf{w}_1|$ and $|\textbf{w}_1| \geq \rho$. It suffices to prove for $k \in \mathbb{N}$, $$
 (\lambda\partial_\lambda)^k \bigg(   \int_{\mathbb{R}^{n-1}} e^{\imath \omega \tilde{\Phi}} a(z, z', \textbf{v}, \lambda) dv \bigg)  
 \lesssim  \omega^{-(n-1)/2},$$ where $\omega =  \lambda x |\textbf{w}_1|$ and $\tilde{\Phi} = |\textbf{w}_1|^{-1}(\Phi(x'/x, y, y', \textbf{v}) \mp x^{-1} d_\mathcal{C}(z, z'))$. Under \eqref{eqn : phase and distance} and \eqref{eqn : nondegeneracy of phase}, this oscillatory integral estimates were proved by  a finer integration by parts argument. See \cite[4-3]{Hassell-Zhang}.

\end{enumerate}

\end{proof}

\subsection{Local Strichartz estimates on $\mathcal{C}$}

The exponent pair $(\mathrm{q}, \mathrm{r})$ is said to be sharp $\mathbf{\sigma}$-admissible if it satisfies $$\mathrm{q}, \mathrm{r} \geq 2, \quad (\mathrm{q}, \mathrm{r}, \mathbf{\sigma}) \neq (2, \infty, 1), \quad \frac{1}{\mathrm{q} } + \frac{\sigma}{ \mathrm{r}} = \frac{\mathbf{\sigma}}{2}.$$  In particular, $(\mathrm{q}, \mathrm{r})$ is said to be Schr\"odinger admissible if it is sharp $n/2$-admissible as  mentioned in the introduction.

Using Proposition \ref{prop : microlocalized spectral measure}, we can establish   Strichartz estimates on $\mathcal{C}$, \begin{equation}
\|e^{\imath t \Delta} f\|_{L^{q}([0,1); L^{\tilde{r}}(\mathcal{C}))} \lesssim \|f\|_{L^2(\mathcal{C})}, \label{eqn : local Strichartz}
\end{equation}for any sharp $(n/2 + \mathbf{e}/4)$-admissible pair $(q, \tilde{r})$. 

This can be done by showing the following dispersive estimates and energy estimates for  conjugated Schr\"odinger propagators from  corresponding conjugated spectral measure, \begin{proposition}\label{prop : microlocalized propagator}Suppose  $\{Q_\alpha(\lambda)\}_\alpha$ is a finite partition of identity as in Proposition \ref{prop : microlocalized spectral measure} and write
	$$U_\alpha(t) = \int_{\lambda > 1}e^{\imath t \lambda^2}Q_\alpha(\lambda) dE_{\sqrt{\Delta_{\mathcal{C}}}}(\lambda).$$ Then the  Schr\"odinger propagator conjugated by $Q_{\alpha}(\lambda)$ features the factorization, \begin{equation}\label{eqn : factorization}
	\int_{\lambda > 1}e^{\imath t \lambda^2}Q_\alpha(\lambda) dE_{\sqrt{\Delta_{\mathcal{C}}}}(\lambda) Q_\alpha^\ast(\lambda) = U_\alpha(t) U_\alpha^\ast(t),\end{equation}    and  uniformly obeys energy estimates
	\begin{equation}\label{eqn : energy estimates}
	\|U_\alpha(t)f  \|_{L^2(\mathcal{C})} \lesssim \|f\|_{L^2(\mathcal{C})},
	\end{equation}	
	as well as dispersive estimates for small times $0< t < 1$,
	\begin{equation}\label{eqn : dispersive estimates}
	\|U_\alpha(t_1)U_\alpha^\ast(t_2) f\|_{L^\infty(\mathcal{C})} \lesssim |t_1 - t_2|^{-n/2 - \mathbf{e}/4} \|f\|_{L^1(\mathcal{C})}. 
	\end{equation}
\end{proposition}
Given Proposition \ref{prop : microlocalized spectral measure}, the proof of \eqref{eqn : factorization}, \eqref{eqn : energy estimates} and \eqref{eqn : dispersive estimates} has little to do with the conic geometry but is purely an analysis argument. We can employ the proof for \cite[Proposition 5.1, Lemma 5.3, Proposition 6.1]{Hassell-Zhang} verbatim. 

Given \eqref{eqn : energy estimates}  and \eqref{eqn : dispersive estimates} as well as the  energy and dispersive estimates for $e^{\imath t \Delta_{\mathcal{C}}}_{\mathrm{low}}$, Keel-Tao endpoint method \cite[Theorem 1.2]{Keel-Tao} leads to that  \begin{eqnarray*}\|U_\alpha(t) f\|_{L^q([0,1); L^{\tilde{r}}(\mathcal{C}))} &\lesssim& \|f\|_{L^2(\mathcal{C})}  \quad \mbox{for each $\alpha$,}\\  \|e^{\imath t \Delta_{\mathcal{C}}}_{\mathrm{low}} f\|_{L^q([0,1); L^{\tilde{r}}(\mathcal{C}))} &\lesssim& \|f\|_{L^2(\mathcal{C})}.\end{eqnarray*}  Since $\{Q_\alpha(\lambda)\}_\alpha$ is a finite partition of identity,  \eqref{eqn : local Strichartz} results from $$e^{\imath t \Delta_{\mathcal{C}}} = e^{\imath t \Delta_{\mathcal{C}}}_{\mathrm{high}} + e^{\imath t \Delta_{\mathcal{C}}}_{\mathrm{low}} \quad \mbox{and} \quad e^{\imath t \Delta_{\mathcal{C}}}_{\mathrm{high}} = \sum_\alpha U_\alpha(t).$$

Let $W^{k, p}(\mathcal{C})$ be the usual $p$-Sobolev space on $\mathcal{C}$, that is $$W^{k, p}(\mathcal{C}) = \bigg\{f \in L^p(\mathcal{C}) : \sum_{l = 0}^{[k]} \|\partial^l f\|_{L^p(\mathcal{C})}   + \Big(\int_{\mathcal{C}^2} \frac{|f(z) - f(z')|^p}{|z-z'|^{ [k]p + n}} \,dg(z) dg(z')\Big)^{1/p} < \infty\bigg\}.$$   The classical Sobolev embedding theorem can be extended to $\mathcal{C}$ (See \cite[Theorem 5.4]{Adams}), that is $$W^{k_1, p_1}(\mathcal{C}) \subset W^{k_2, p_2}(\mathcal{C}), \quad \mbox{for}  \quad\frac{1}{p_1} - \frac{k_1}{n} = \frac{1}{p_2} - \frac{k_2}{n}.$$ This allows us to rewrite 
 \eqref{eqn : local Strichartz} as  Schr\"odinger admissible Strichartz estimates with a loss. Suppose $(q, r)$ is Schr\"odinger admissible. It follows that
 $$\frac{n/2 + \mathbf{e}/4}{\tilde{r}} - \frac{n/2}{r}   = \frac{\mathbf{e}}{8}.$$  Then it is immediate from Sobolev embeddings and $(n/2 + \mathbf{e}/4)$-admissibility that  \begin{equation}
 \|e^{\imath t \Delta} f\|_{L^{q}([0,1); W^{-\mathbf{k}, r}(\mathcal{C}))} \lesssim \|f\|_{L^2(\mathcal{C})}, \label{eqn : local Strichartz with loss}
 \end{equation} with $\mathbf{k}$ defined in \eqref{def : bf k}.

\subsection{Global Strichartz estimates on $\mathcal{X}$}

\begin{proof}[Proof of Theorem \ref{thm : Strichartz}]

To employ local estimates \eqref{eqn : local Strichartz}, we choose a finite partition of unity $1 = \sum_{i = 0}^m \chi_i$ on $\mathcal{X}$ such that \begin{itemize}
	\item each $\chi_i$ for $i = 1, \cdots, m$, is a smooth function compactly supported within $\mathcal{C}_i$ and equal to $1$ on $\mathcal{K}_i$,
	\item $\chi_0$ is a smooth function supported within $\mathcal{X} \setminus \bigcup_{i=1}^m \mathcal{K}_i$ and equal to $1$ on $\mathcal{X} \setminus \bigcup_{i=1}^m \mathcal{C}_i$.
\end{itemize}

Now we denote $u_i = \chi_i u$ for $i=0, 1, \cdots, m$. Applying the Schr\"odinger operator to each $u_i$ gives the Cauchy problem for inhomogeneous Schr\"odinger equations, $$\left\{\begin{array}{l}
\imath \partial_t u_i(t, z) + \Delta_{\mathcal{X}} u_i(t, z) = v_i(t, z)
\\u_i(0, z) = \chi_i f(z) 
\end{array}\right.,$$ where $v_i(t, z) = [\Delta_{\mathcal{X}}, \chi_i] u(t, z)$. 

We write $u_i = u_i' + u_i'',$ such that 
\begin{itemize}\item  $u_i'$ is the solution to 
	
	$$\left\{\begin{array}{l}
	\imath \partial_t u_i'(t, z) + \Delta_{\mathcal{X}} u_i'(t, z) = 0
	\\u_i'(0, z) = \chi_i f(z) 
	\end{array}\right.;$$

\item $u_i''$ is the solution to

$$\left\{\begin{array}{l}
\imath \partial_t u_i''(t, z) + \Delta_{\mathcal{X}} u_i''(t, z) = v_i(t, z)
\\u_i''(0, z) = 0
\end{array}\right..$$
\end{itemize}

By \eqref{eqn : local Strichartz with loss} and Strichartz estimates on $\mathbb{R}^n$, we know, $$\|e^{\imath t \Delta_{\mathcal{X}}} u_i'\|_{L^q([0,1); W^{- \mathbf{k}, r}(\mathcal{X}))} \lesssim \|\chi_i f\|_{L^2(\mathcal{X})},$$ where $(q, r)$ is Schr\"odinger admissible.

Then it remains to deal with $u_i''$. Duhamel's principle yields that  $$u_i''(t, z) = \int_0^t e^{-\imath (t-s) \Delta_{\mathcal{X}}} v_i(s, z) ds = e^{-\imath  \Delta_i} \int_0^t e^{\imath s \Delta_i} v_i(s, z) ds,$$ 
where $\Delta_0 = \Delta_{\mathbb{R}^n}$ and $\Delta_i = \Delta_{\mathcal{C}_i}$ for $i = 1, \cdots, m$.
By  \eqref{eqn : local Strichartz} and Strichartz estimates on $\mathbb{R}^n$, it suffices to show $$\bigg\|\int_0^t e^{\imath s \Delta_i} v_i(s, z) ds\bigg\|_{L^2(\mathcal{X})} \lesssim \|f\|_{L^2(\mathcal{X})}.$$ Christ-Kiselev's lemma \cite[Theorem 1.2]{Christ-Kiselev}
further reduces this inequality to $$\bigg\|\int_0^1 e^{\imath s \Delta_i} v_i(s, z)\,ds \bigg\|_{L^2(\mathcal{X})} \lesssim \|f\|_{L^2(\mathcal{X})}.$$
 Now we need the local smoothing estimates \eqref{eqn : local smoothing} and dual local smoothing estimates $$
 \bigg\|\int_0^1 e^{\imath s \Delta_{\mathcal{X}}} \chi F(s, z) \,ds\bigg\|_{L^2(\mathcal{X})} \lesssim \|F\|_{L^2(\mathbb{R}, \mathcal{D}^{-1/2}(\mathcal{X}))},$$ as $[\Delta_{\mathcal{X}}, \chi_i]$ is a differential operator of order 1 supported within $\bigcup_{j=1}^m\mathcal{C}_j \setminus \mathcal{K}_j^\circ$.  
Consequently, we obtain  $$\bigg\|\int_{\mathbb{R}} e^{\imath s \Delta_i} v_i(s, z)\,ds \bigg\|_{L^2(\mathcal{X})} \lesssim \|v_i\|_{L^2(\mathbb{R}, \mathcal{D}^{-1/2}(\mathcal{X}))} .$$
In addition, $v_i$ satisfies $$\|v_i\|_{L^2(\mathbb{R}, \mathcal{D}^{-1/2}(\mathcal{X}))} \lesssim \|u_i\|_{L^2(\mathbb{R}, \mathcal{D}^{1/2}(\mathcal{X}))}.$$ Therefore the proof of the theorem will be finished by a final application of local smoothing estimates to $u_i$
$$\|u_i\|_{L^2(\mathbb{R}, \mathcal{D}^{1/2}(\mathcal{X}))} \lesssim \|f\|_{L^2(\mathcal{X})}.$$
\end{proof}

\bibliographystyle{abbrv}
\bibliography{main}

\end{document}